%% file: main.tex
\documentclass[11pt]{article}
\usepackage[T1]{fontenc}
\usepackage{multirow}
\usepackage{booktabs,longtable,xcolor}
\usepackage{adjustbox}
\usepackage{threeparttable}
\usepackage{amsmath,amsthm,amsfonts,amssymb,amscd}
\usepackage[scr]{rsfso}
\usepackage{algorithm}
\usepackage{algpseudocode}
\makeatletter
\let\STATE\State

\makeatother

\usepackage[utf8]{inputenc} 
\usepackage{array} 
\usepackage{stackrel}
 \usepackage{amsmath}
\usepackage{enumerate} 
\usepackage{color}
\usepackage{float}
\usepackage{soul}
\usepackage{subcaption}
\usepackage[color]{showkeys}
\usepackage{graphicx}
\graphicspath{{./}{figures/}}
\usepackage{mathpazo}
\usepackage{fancyhdr}
\usepackage{empheq}
\usepackage{hyperref}
\hypersetup{
  colorlinks=true,
  linkcolor=red,
  citecolor=blue,
  urlcolor=blue
}
\usepackage[left=2cm,right=2cm,top=2cm,bottom=2cm]{geometry}
\usepackage{float}
\usepackage{booktabs}

\def\XX{\mathbb{X}}
\def\YY{\mathbb{Y}}

\def\hat{\widehat}
\def\Tilde{\widetilde}
\def\Bar{\overline}
\def\la{\langle}
\def\ra{\rangle}
\def\ve{\varepsilon}
\def\B{\mathbb{B}}
\def\h{\hfill\Box}
\def\R{\mathbb{R}}
\def\ox{\bar{x}}

\def\oz{\bar{z}}

\def\ov{\bar{v}}

\def\conv{\mbox{\rm conv}\,}
\def\cone{\mbox{\rm cone}\,}
\def\ri{\mbox{\rm ri}\,}

\def\Im{\mbox{\rm Im}\,}
\def\para{\mbox{\rm par}\,}
\def\gph{\mbox{\rm gph}\,}
\def\epi{\mbox{\rm epi}\,}

\def\Span{\mbox{\rm span}\,}
\def\aff{\mbox{\rm aff}\,}

\def\dom{\mbox{\rm dom}\,}
\def\Ker{\mbox{\rm Ker}\,}

\def\cl*co{\mbox{\rm cl}^*\mbox{\rm co}\,}

\def\cl{\mbox{\rm cl}\,}
\def\d{\mbox{\rm dist}}

\def\h{\hfill\triangle}

\def\oR{\Bar{\R}}

\def\al{\alpha}

\def\N{\mathbb{N}}

\def\hs7{\hspace*{7pt}}

\def\Id{\mathbb{I}}

\DeclareMathOperator{\prox}{\mathbf{prox}}
\renewcommand{\theequation}{\thesection.\arabic{equation}}

\def\h{\hfill\Box}
\def\kk{\kappa}
\DeclareMathOperator*{\argmin}{arg\,min}
\usepackage{url}
\begin{document}

\newtheorem{Theorem}{Theorem}[section]
\newtheorem{Assumption}[Theorem]{Assumption}
\newtheorem{Conjecture}[Theorem]{Conjecture}
\newtheorem{Proposition}[Theorem]{Proposition}
\newtheorem{Remark}[Theorem]{Remark}
\newtheorem{Lemma}[Theorem]{Lemma}
\newtheorem{Corollary}[Theorem]{Corollary}
\newtheorem{Definition}[Theorem]{Definition}
\newtheorem{Example}[Theorem]{Example}
\newtheorem{Fact}[Theorem]{Fact}
\newtheorem*{pf}{Proof}
\renewcommand{\theequation}
{\thesection.\arabic{equation}}
\normalsize
\normalfont
\medskip
\def\endproof{$\h$\vspace*{0.1in}}

\title{\bf A Globally Convergent Newton Method
for Nonsmooth Optimization
in the absence of Smooth Surrogates}
\date{}
\author{TRAN T. A. NGHIA\footnote{Department of Mathematics and Statistics, Oakland University, Rochester, MI 48309, USA; email: nttran@oakland.edu}
\and NGHIA V. VO\footnote{Department of Mathematics and Statistics, Oakland University, Rochester, MI 48309, USA; email: nghiavo@oakland.edu}\and  KHOA V. H. VU\footnote{Department of Electrical and Computer Engineering, Oakland University, Rochester, MI 48309, USA; email: khoavu@oakland.edu} }

\maketitle
\vspace{-0.2in}
\begin{abstract}
We propose a globally convergent Newton method for convex composite optimization problems with polyhedral regularizers. The method rests on the observation that the proximal gradient step yields not only a new point but also a subgradient, which determines the face on which the linearization of the regularizer is exact. Minimizing a second-order model of the objective over this face produces a Newton direction that involves no generalized second-order information about the regularizer and is a descent direction for the objective itself. The method is therefore globalized by an Armijo line search on the original objective which, unlike forward--backward envelope techniques, requires only evaluations of the smooth component. We establish global convergence and local quadratic convergence to tilt-stable minimizers without any nondegeneracy condition, and show that the unit step is always accepted when the smooth component is quadratic, which yields a line-search-free method. The subproblems require only projections onto faces, which are explicit for the Lasso and nonnegative least squares and computable in linear time for the fused Lasso. Numerical experiments show that the method compares favorably with state-of-the-art solvers.

\end{abstract}

\noindent{\bf Mathematics Subject Classification (2020)}. 49J52, 49J53, 49M15, 65K10, 90C25
\vspace{0.1in}

\noindent{\bf Keywords:} Nonsmooth Newton methods, Polyhedral regularization, Generalized Hessian, Variational Analysis, Convex Optimization 

\section{Introduction}\label{sec:intro}

In this paper, we develop a second-order method for solving the composite optimization problem
\begin{equation}\label{eq:intro-CP}
\min_{x\in\mathbb{X}}\ \varphi(x):=f(x)+g(x),
\end{equation}
where $\mathbb{X}$ is a Euclidean space, $f:\mathbb{X}\to\mathbb{R}$ is a convex, twice continuously differentiable function with a Lipschitz continuous gradient, and $g:\mathbb{X}\to\overline{\mathbb{R}}:=\mathbb{R}\cup\{+\infty\}$ is a proper convex polyhedral function. Problems of this form arise throughout statistics, signal and image processing, and machine learning, where a smooth data-fidelity term is combined with a regularizer to promote a desired structure of the solution.

Polyhedral regularizers are among the most popular choices when the solution is expected to be sparse, piecewise constant, or grouped by magnitude. A classical example is the Least Absolute Shrinkage and Selection Operator (Lasso)~\cite{Tibshirani96}, in which $g$ is the $\ell_1$ norm. To promote piecewise constancy as well, the fused Lasso~\cite{Tibshirani05} adds the one-dimensional total variation seminorm, which penalizes differences between consecutive components. To cluster coefficients by magnitude, the Octagonal Shrinkage and Clustering Algorithm for Regression (OSCAR)~\cite{Bondell08} and Sorted-$\ell_1$ Penalized Estimation (SLOPE)~\cite{Bogdan2015} penalize the coefficients according to their sorted absolute values.

There are numerous iterative strategies for solving~\eqref{eq:intro-CP}, which can broadly be divided into first-order and second-order methods. First-order methods are the most widely used. Examples include ISTA and FISTA~\cite{BeckTeboulle09}, the alternating direction method of multipliers~\cite{Boyd11}, and primal--dual splitting~\cite{ChambollePock11}. However, first-order methods converge at best sublinearly in general, with the optimal worst-case rate $O(1/k^2)$ for the objective values~\cite{NemirovskiYudin83,Nesterov18}, 
and linearly under quadratic growth~\cite{drusvyatskiy18}. They may therefore be slow when high accuracy is required, which motivates the development of fast second-order methods.

Extending Newton's method to~\eqref{eq:intro-CP}, however, raises two difficulties. The first concerns the model. At each iteration, Newton's method minimizes a quadratic approximation of the objective built from its Hessian. Because $g$ is nondifferentiable, it is not clear which second-order information about $g$ should enter the model. The second concerns globalization. For smooth problems, a Newton direction computed with a positive definite Hessian, or an approximation of it, is a descent direction, and a line search on the objective guarantees global convergence. For composite problems, this guarantee no longer holds and thus globalization mechanism is then required for convergence.

Proximal Newton-type methods, such as those of Lee, Sun, and Saunders~\cite{LeeSunSaunders14} and of Byrd, Nocedal, and Oztoprak~\cite{byrd2016inexact}, overcome both difficulties by approximating only $f$ and keeping $g$ intact. At each iteration, they solve
\begin{equation}\label{eq:intro-PN}
\min_{x\in\mathbb{X}}\Big\{\langle\nabla f(x_k),x-x_k\rangle+\frac12\langle H_k(x-x_k),x-x_k\rangle+g(x)\Big\},
\end{equation}
where $H_k$ is the Hessian of $f$ at $x_k$ or an approximation of it. This model requires no second-order information about $g$. Moreover, since $g$ enters~\eqref{eq:intro-PN} exactly, the step from $x_k$ to a solution of the subproblem is a descent direction for $\varphi$ itself, and a line search on the original objective globalizes the method. Mordukhovich, Yuan, Zeng, and Zhang~\cite{mordukhovich2023globally} recently established global convergence and superlinear rates for this approach under subregularity-type conditions. The price is that each subproblem~\eqref{eq:intro-PN} is itself a nonsmooth composite problem of the same type as~\eqref{eq:intro-CP}, and solving it iteratively accounts for most of the computational cost.

A second family of methods avoids this cost by applying Newton's method directly to the optimality condition $0\in\nabla f(x)+\partial g(x)$, reformulated as a nonsmooth equation through the proximal mapping of $g$. Each iteration then reduces to solving a linear system. Semismooth Newton methods~\cite{QiSun93,FacchineiPang03} build this system from an element of the Clarke generalized Jacobian of the proximal mapping and converge locally superlinearly under suitable regularity conditions. Combined with an augmented Lagrangian method applied to the dual problem, this approach underlies the highly efficient SSNAL solvers of Li, Sun, and Toh~\cite{LST18,LST18Fused}. Its effectiveness, however, hinges on computing generalized Jacobian elements, which requires a problem-specific derivation that can be delicate for structured regularizers such as the total variation seminorm.

A different line of work, rooted in second-order variational analysis, replaces the Hessian of $g$ by its generalized Hessian, namely the coderivative of $\partial g$~\cite{Mordukhovich80,Mordukhovich06}. The positive definiteness of the generalized Hessian characterizes tilt-stable minimizers in the sense of Poliquin and Rockafellar~\cite{PoliquinRock98}. This characterization underlies the coderivative-based Newton methods of Mordukhovich and Sarabi~\cite{MordukhovichSarabi21} and of Khanh, Mordukhovich, Phat, and Tran~\cite{KhanhMordukhovichPhat24,KhanhMordukhovichPhat25}. Closely related are the semismooth$^*$ and SCD Newton methods of Gfrerer and Outrata~\cite{GfrererSIOPT21,GfrererCOA25,Gfrerer25}. These methods converge locally superlinearly under tilt stability or metric regularity, although computing generalized Hessians or subspace-containing derivatives of complex regularizers remains a challenge in practice. Since their directions are computed from second-order information alone, globalization requires an additional device. The prevailing choice is the forward--backward envelope~\cite{P13,S17,T18}, a continuously differentiable function sharing its minimizers with $\varphi$ on which a line search can be performed. This envelope is powerful and extends to nonconvex and nonpolyhedral problems. Using it as a merit function in the design of optimization methods, however, may come at a higher computational cost; see Remark~\ref{rem:FBE_comparison}. 

\newpage
A third line of work develops Newton methods on manifolds. When $g$ is partly smooth in the sense of Lewis~\cite{Lewis02}, it is smooth along an active manifold, which the proximal gradient method identifies in finitely many iterations provided that the nondegeneracy condition $-\nabla f(\bar x)\in\operatorname{ri}\partial g(\bar x)$ holds at the minimizer $\bar x$~\cite{HareLewis04,LiangFadiliPeyre17}. Lewis and Wylie~\cite{LewisWylie2021} and Bareilles, Iutzeler, and Malick~\cite{BIM23} exploit this property by switching to Newton steps on the identified manifold, where $\varphi$ is smooth and classical line searches apply. Nondegeneracy, however, involves the unknown minimizer and thus cannot be verified in advance; when it fails, the performance of identification-based methods may deteriorate significantly~\cite{FMP18}.

In summary, existing Newton-type methods for~\eqref{eq:intro-CP} must compromise between the cost of their subproblems, the nature of their globalization, and the assumptions underlying their fast local convergence. Semismooth and coderivative-based Newton methods solve linear systems and converge superlinearly, but the latter are typically globalized through the forward--backward envelope. Manifold-based methods converge quadratically and perform a line search on $\varphi$ along the identified manifold, but identification requires nondegeneracy. Proximal Newton methods perform a line search on $\varphi$ itself, at the price of a nonsmooth subproblem at every iteration. This raises the following question:
\begin{center}
\emph{Can a Newton-type method for~\eqref{eq:intro-CP} avoid nonsmooth composite subproblems, converge globally without a smooth surrogate, and converge quadratically without any nondegeneracy assumption?}
\end{center}

\medskip
\noindent\textbf{Our Contributions.} In this paper, we answer this question affirmatively for polyhedral regularizers. Our approach rests on a simple observation about the proximal gradient step. Given $x_k$, let
\[
y_k:=\operatorname{prox}_{\alpha g}\big(x_k-\alpha\nabla f(x_k)\big)\quad\text{and}\quad z_k:=\frac{x_k-y_k}{\alpha}-\nabla f(x_k).
\]
The pair $(y_k,z_k)$ lies on the graph of $\partial g$ and therefore carries more information than the point $y_k$ alone: the subgradient $z_k$ defines an affine minorant $x\mapsto g(y_k)+\langle z_k,x-y_k\rangle$ of $g$ that touches $g$ at $y_k$. The crucial fact is that the set on which this minorant touches $g$ is known explicitly. Indeed, the Fenchel--Young identity yields
\begin{equation}\label{eq:intro-exact}
g(x)=g(y_k)+\langle z_k,x-y_k\rangle\quad\Longleftrightarrow\quad x\in\partial g^*(z_k)=\big\{x\in\mathbb{X}\ \big|\ z_k\in\partial g(x)\big\},
\end{equation}
while outside $\partial g^*(z_k)$ the minorant lies strictly below $g$. In other words, the subdifferential $\partial g^*(z_k)$ is the largest region on which the linearization of $g$ at $y_k$ is exact, and on this region the nonsmoothness of $\varphi$ disappears entirely, $\varphi$ being the sum of $f$ and an affine function. Since $g$ is polyhedral, this region is a polyhedron determined explicitly by $z_k$, which we call the face selected by $z_k$.

It is then natural to take the Newton step for $\varphi$ within this face:
\begin{equation}\label{eq:intro-sub}
\widehat{x}_k\in\operatorname*{argmin}_{x\in\partial g^*(z_k)}\ \langle\nabla f(y_k)+z_k,\,x-y_k\rangle+\frac12\langle H_k(x-y_k),x-y_k\rangle ,
\end{equation}
where $H_k$ is $\nabla^2 f(y_k)$ or a positive definite regularization of it. In view of~\eqref{eq:intro-exact}, problem~\eqref{eq:intro-sub} and the proximal Newton subproblem~\eqref{eq:intro-PN} centered at $y_k$ have identical objectives on the face, up to an additive constant. Problem~\eqref{eq:intro-sub} thus retains the exact model of $g$ used by proximal Newton methods while discarding precisely the region where this model is nonsmooth; it is a quadratic program over a polyhedron, free of any generalized second-order object of $g$. Nor is the restriction a loss near a solution. Let $\bar x$ be a minimizer of~\eqref{eq:intro-CP} and set $\bar z:=-\nabla f(\bar x)\in\partial g(\bar x)$. Since $g$ is polyhedral, the graph of $\partial g$ coincides, near $(\bar x,\bar z)$, with a translate of a finite union of polyhedral cones \cite{MordukhovichSarabi16}. This local conic structure yields some $\varepsilon>0$ such that
\begin{equation}\label{eq:intro-face}
\bar x\in\partial g^*(z)\subset\partial g^*(\bar z)\qquad\text{for all } (y,z)\in\operatorname{gph}\partial g\cap\big(\mathbb{B}_\varepsilon(\bar x)\times\mathbb{B}_\varepsilon(\bar z)\big).
\end{equation}
Hence, as soon as $(y_k,z_k)$ is close to $(\bar x,\bar z)$, the face $\partial g^*(z_k)$ contains the solution, and~\eqref{eq:intro-sub} becomes a Newton step for a smooth problem over a polyhedron containing $\bar x$. The curvature needed for such a step is supplied by tilt stability, which for~\eqref{eq:intro-CP} implies that 
\begin{equation}\label{eq:intro-tilt}
\langle\nabla^2 f(\bar x)w,w\rangle>0\qquad\text{for all } w\in\operatorname{par}\partial g^*(\bar z)\setminus\{0\},
\end{equation}
where $\mathrm{par}\partial g^* (\bar{z})$ is the parallel subspace of $\partial g^* (\bar{z})$ in $\XX$; see \cite[Theorem 4.3]{CuiHoheiseiNghiaSun24}. By~\eqref{eq:intro-face}, every direction within $\partial g^*(z_k)$ lies in $\operatorname{par}\partial g^*(\bar z)$, so~\eqref{eq:intro-tilt} provides uniform curvature of $f$ along the faces visited by the method. Together, \eqref{eq:intro-face} and~\eqref{eq:intro-tilt} drive the \emph{local quadratic convergence} of the method. Moreover, when $f$ is quadratic, the Newton step over $\partial g^*(z_k)$ is exact: once the selected face contains $\bar x$, the next iterate is $\bar x$ itself, and the method achieves \emph{finite termination}. Neither property requires $-\nabla f(\bar x)$ to lie in the relative interior of $\partial g(\bar x)$. To illustrate this, consider the problem
\begin{equation}\label{eq:intro-ex}
\min_{x\in\mathbb{R}^2}\ \varphi(x):=\frac12\,x^\top\begin{pmatrix}2&3\\3&5\end{pmatrix}x+\begin{pmatrix}-4\\-1\end{pmatrix}^{\!\top}x+2\|x\|_1,
\end{equation}
whose unique minimizer is $\bar x=(1,0)$. Here $f$ is quadratic and $-\nabla f(\bar x)=(2,-2)$, which lies on the relative boundary of $\partial(2\|\cdot\|_1)(\bar x)=\{2\}\times[-2,2]$. Thus the nondegeneracy condition fails at $\bar x$, although $\bar x$ is tilt stable since the quadratic form is positive definite. The left plot of Figure~\ref{fig:intro} shows the iterates of several methods started from $x_0=(-4,-2)$. GESN reaches $\bar x$ in a single iteration, in accordance with the finite termination property, whereas ZeroFPR \cite{T18}, SSNAL \cite{LST18}, GSSN \cite{GfrererCOA25}, ISTA, and FISTA \cite{BeckTeboulle09} require considerably more iterations. This behavior does not depend on the particular starting point. Over 225 different starting points, the median number of iterations of GESN remains the lowest among all methods, as shown in the right plot of Figure~\ref{fig:intro}. This brings us to the question of globalization.

\begin{figure}[H]
  \centering
  \begin{subfigure}[t]{0.559\linewidth}
    \centering
    \includegraphics[width=\linewidth]{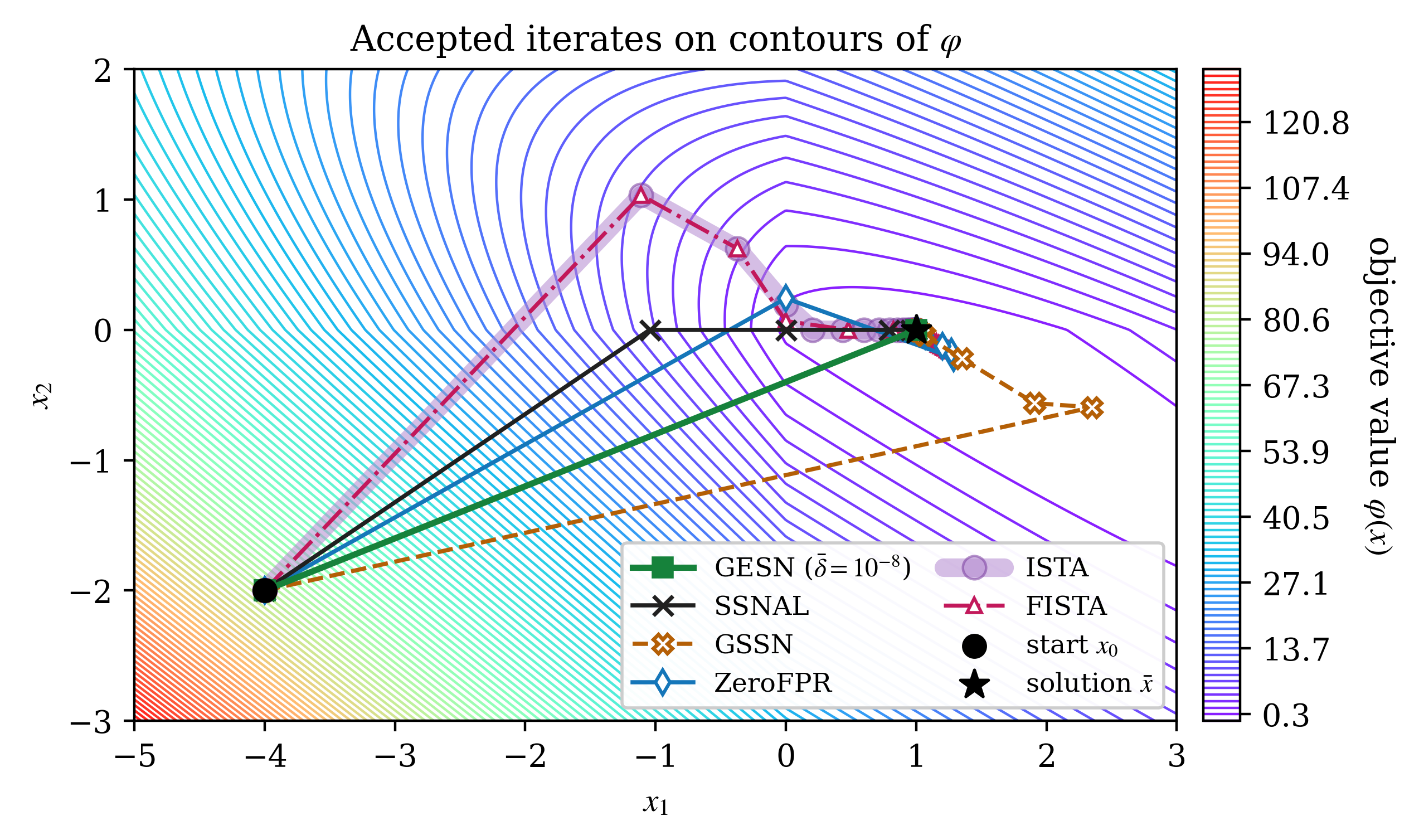}
  \end{subfigure}\hfill
  \begin{subfigure}[t]{0.421\linewidth}
    \centering
    \includegraphics[width=\linewidth]{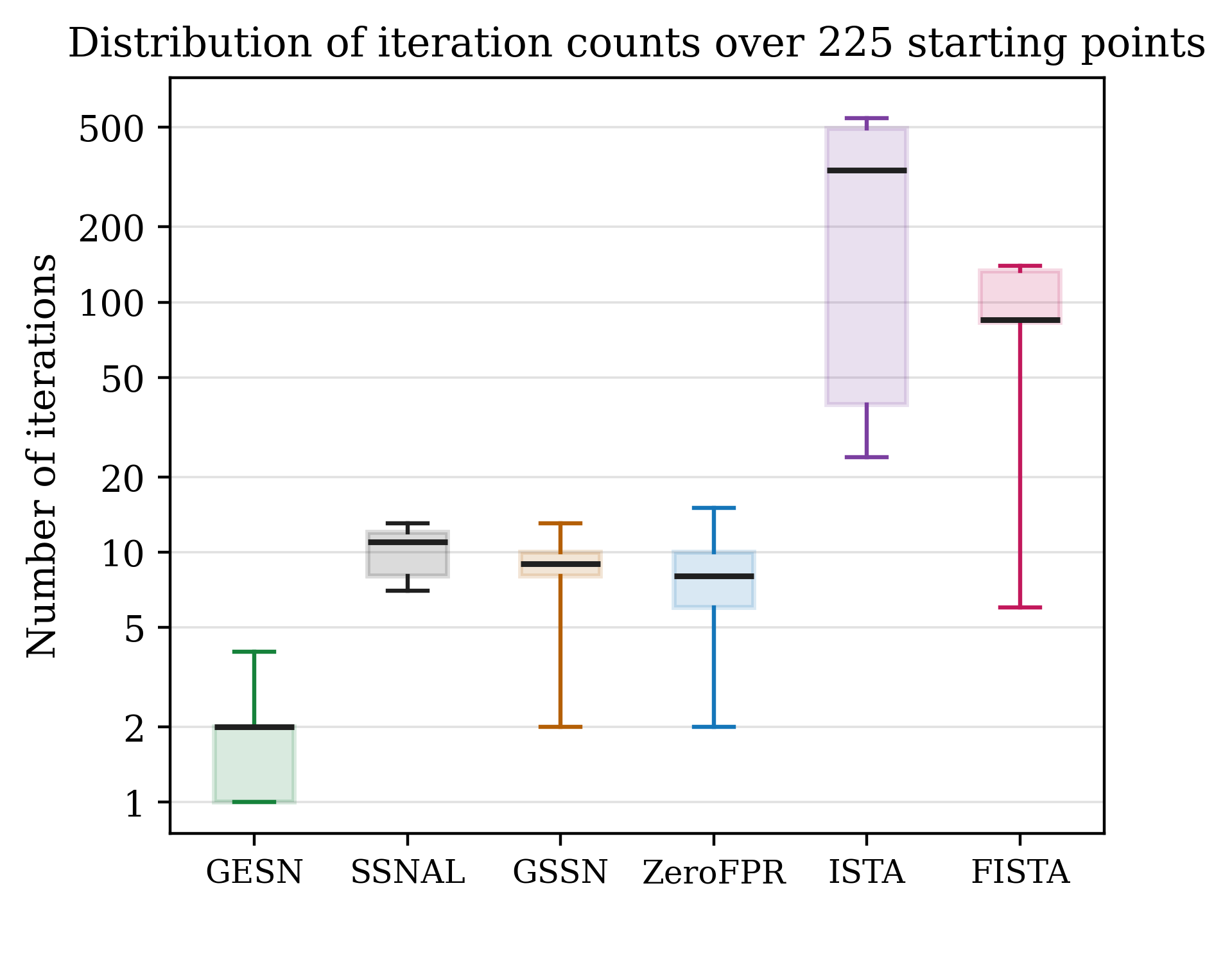}
  \end{subfigure}
  \caption{Left: Trajectories from $x_0=(-4,-2)$ of GESN (ours) versus SSNAL\cite{LST18}, GSSN \cite{GfrererCOA25},
  ZeroFPR \cite{T18}, ISTA and FISTA \cite{BeckTeboulle09} on problem~\eqref{eq:intro-ex}. Right: Distribution of the number of iterations until convergence for
  $225$ starting points sampled on a uniform $15\times15$ grid over
  $[-3.75,6.25]\times[-4.75,5.25]$.}
  \label{fig:intro}
\end{figure}

The face restriction also settles this question. Since $y_k$ and $\widehat x_k$ both lie in $\partial g^*(z_k)$, on which $g$ is affine as mentioned in \eqref{eq:intro-exact}, $\varphi$ behaves along the segment joining them exactly as a smooth function does. Combining this observation with the optimality condition of~\eqref{eq:intro-sub} gives, for $d_k:=y_k-\widehat x_k$, a nonsmooth counterpart of the descent identity of the classical Newton method:
\begin{equation}\label{eq:intro-descent}
\varphi'(y_k;-d_k)=-\langle\nabla f(y_k)+z_k,\,d_k\rangle\le-\langle H_k d_k,d_k\rangle .
\end{equation}
Consequently, whenever $H_k$ is positive definite, $-d_k$ is a descent direction for $\varphi$ itself, and the Armijo condition for $\varphi$ along it reduces to a test involving $f$ only:
\begin{equation}\label{eq:intro-LS}
f(y_k-\beta d_k)\le f(y_k)+\beta\langle z_k,d_k\rangle-\sigma\beta\langle z_k+\nabla f(y_k),d_k\rangle .
\end{equation}
Here, the contributions are twofold. First, the proposed approach dispenses
with the forward-backward envelope and its associated parameter,
thereby allowing a range of proximal gradient stepsizes twice as
large as that permitted by the coderivative-based
\cite{KhanhMordukhovichPhat25}. Second, the proposed
line-search condition reduces the computational cost of backtracking,
as each trial requires only a reevaluation of $f$; see
Remark~\ref{rem:FBE_comparison}. For quadratic $f$, GESN further
dispenses with backtracking altogether: the unit stepsize is accepted
at every iteration.

The rest of the paper is organized as follows. Section~\ref{sec:prelim} recalls preliminaries from convex and variational analysis. Section~\ref{sec:polyhedral} introduces the face-restricted Newton method and establishes its local quadratic convergence, while Section~\ref{sec:global} develops GESN and its global and local convergence theory. Section~\ref{sec:efficient_implementation} is devoted to the subproblem solver, Section~\ref{sec:numerical} reports numerical experiments, and Section~\ref{sec:conclude} concludes the paper.
\section{Preliminaries}\label{sec:prelim}

Throughout this paper, we denote by $\XX$ a finite-dimensional Euclidean space endowed with the inner product $\la\cdot,\cdot\ra$ and the corresponding Euclidean norm $\|\cdot\|$. For a subspace $S \subseteq \XX$, the symbol $P_{S}$ stands for the {\em orthogonal projection} onto $S$ and $S^{\bot}$ is the {\em orthogonal complement} of $S$. Given a nonempty set $C\subseteq \XX$, the {\em span}, {\em affine hull}, {\em convex conic hull}, {\em convex hull}, {\em closure} and {\em relative interior} of $C$ are denoted, respectively, by $\Span C$, $\aff C$, $\cone C$, $\conv C$, $\cl C$ and $\ri C$.  For a linear operator $A:\XX\to \YY$, we write $\Ker A$ and $\Im A$, respectively, for the {\em kernel/null space} and {\em image/range} of $A$. The set $\mathbb{B}_{\varepsilon}(\bar{x})$ is  the {\em closed ball} in $\XX$ centered at $\bar{x}\in \XX$ with radius $\varepsilon>0$. 
\subsection{Basic definitions and preliminaries}
Let $C$ be  a closed convex set in $\XX$. The {\em support function} of $C$ is given by 
\begin{equation}\label{eq:support}
\sigma_C(v) := \sup_{x\in C}\,\la v,x \ra \quad \forall v\in \XX.
\end{equation}
The {\em normal cone} to $C$ at $x \in C$ is defined as 
\begin{equation}\label{def:N}
N_{C}(x) := \{v \in \XX \mid \la v,y - x\ra \le 0, \forall y \in C  \}.
\end{equation}
The {\em parallel subspace} of $C\subseteq \XX$ is given by 
\begin{equation}\label{def:Par}
 {\rm par}\, C:= \Span(C - x) \quad \text{ for all }\quad  x\in C.   
\end{equation}
Note that the parallel subspace of any convex cone $K\subset \XX$ can be calculated as ${\rm par}\, K = \Span K = K - K.$

For a proper, lower semi-continuous, convex function $\varphi:\XX\to \oR$ with $\,\,\oR:=\R\cup\{+\infty\}$, the {\em effective domain} and the {\em epigraph} of $\varphi$ are defined, respectively, by
\[
\dom \varphi:= \{x \in \XX \mid \varphi(x) < \infty \} \quad \text{and} \quad \epi \varphi:= \{(x,t) \in \XX \times \R \mid \varphi(x) \le t \}.
\]
The \emph{convex subdifferential} of $\varphi$ at $x\in \dom\varphi$ is defined by
\begin{equation}\label{def:convex-subdiff}
\partial \varphi(x):=\big\{v\in \XX \mid \varphi(y)\ge \varphi(x)+\la v,y-x\ra,\ \forall y\in \XX\big\}.
\end{equation}
The normal cone $N_C(x)$ in \eqref{def:N} is indeed the subdifferential \eqref{def:convex-subdiff} of the {\em indicator function} $\delta_{C}(\cdot)$ at $x$, which is defined by $\delta_{C}(x) = 0$ whenever $x \in C$, and $\infty$ otherwise.

A proper function $f:\XX\to\overline{\R}$ is {\em strongly convex} if there is a constant $\sigma >0$ such that for all $x,y \in \XX$ and $\lambda \in[0,1]$, we have
\[
f((1-\lambda)x+\lambda y) \leq (1-\lambda)f(x) + \lambda f(y) -\dfrac{\sigma}{2}\lambda(1-\lambda)\|x-y\|^2.
\]
When $f$ is strongly convex with constant $\sigma >0$, we say $f$ is $\sigma$-strongly convex and $\sigma$ is the modulus of strong convexity. If, in addition, $f$ is twice differentiable, it is well-known that $f$ is $\sigma$-strongly convex if and only if $$\la\nabla^2 f(x)w,w\ra \geq \sigma\|w\|^2 \quad \text{for all} \quad x,w \in \XX.$$
When a proper, l.s.c function $f$ is strongly convex, it admits a unique global minimizer; see, e.g., \cite[Theorem~5.25]{Beck-1storder}.

The \emph{Legendre-Fenchel conjugate} of $\varphi$ is the convex function $\varphi^{*}: \XX \to \Bar{\R}$ defined by
\[
\varphi^*(v) := \sup_{x\in \XX}\, \{\la v, x\ra - \varphi(x)\}\ \text{ for all }\ v\in \XX.
\]
When $\varphi(x)=\delta_C(x)$ is the indicator function to a convex set $C$, its conjugate $\delta_C^*$ is exactly the support function $\sigma_C$ \eqref{eq:support} to $C$. 

A fundamental result in Convex Analysis is the Fenchel-Young's identity occasionally used in our paper, which states that 
\begin{equation}\label{def:fenyou}
    v \in \partial \varphi(x) \quad \text{if and only if} \quad \varphi(x) + \varphi^{*}(v) = \la v,x \ra .
\end{equation} 

Let us next recall the definition of {\em proximal operator} of a convex function $\varphi$, being defined as
\begin{equation}\label{def:prox}
\prox_{\alpha \varphi}(x) := \argmin_{y\in \XX}\left\{\varphi(y) + \frac{1}{2\alpha}\|y-x\|^2 \right\} \quad \forall\, x\in \XX, \forall\, \al>0.
\end{equation}
The proximal mapping $\prox_{\alpha \varphi}:\XX\to \XX$ associated with a convex function $\varphi$ is a single-valued mapping with full domain. Moreover, it can be expressed equivalently as the {\em resolvent operator} $(\Id + \alpha\partial \varphi)^{-1}$, where $\Id:\XX\to \XX$ is the {\em identity operator}. By the first-order optimality condition, we have
\begin{equation}\label{eq:res}
y=\prox_{\alpha\varphi}(x)
\quad\text{if and only if}\quad
\frac{x-y}{\alpha}\in \partial\varphi(y).
\end{equation}

The (one-sided) {\em directional derivative} of an extended-real-valued function $\varphi: \XX \to \overline{\R}$ at the point $\ox\in \dom \varphi$ in the direction $d\in \XX$ is the limit 
\begin{equation}\label{eq:direc-deri}
    \varphi' (\ox;d):= \lim_{\beta\to 0^+} \dfrac{\varphi (\ox+\beta d) - \varphi (\ox)}{\beta},
\end{equation}
if the limit exists as either a real number or $\pm \infty$. In the case $\varphi$ is convex, the directional derivative $\varphi' (\ox;d)$ exists in every direction $d\in \XX$.




Finally, we recall, for completeness, several notions of
convergence rates.

\begin{Definition}[Rates of convergence]
Let $\{x^k\}\subset\mathbb{R}^n$ be a sequence converging to $\bar{x}$ as
$k\to\infty$, with $x^k\neq\bar{x}$ for all $k\in\mathbb{N}$. The convergence
rate is said to be:
\begin{enumerate}[\bf (i)]
    \item \textup{R-linear} if
    \[
        \limsup_{k\to\infty}
        \bigl(\|x^k-\bar{x}\|\bigr)^{1/k}
        <1,
    \]
    i.e., there exist $\mu\in(0,1)$, $c>0$, and $k_0\in\mathbb{N}$ such that $\|x^k-\bar{x}\| 
        \leq c\mu^k$ for all $k\geq k_0$.

    \item \textup{Q-linear} if
    \[
        \limsup_{k\to\infty}
        \frac{\|x^{k+1}-\bar{x}\|}
             {\|x^k-\bar{x}\|}
        <1,
    \]
    i.e., there exist $\mu\in(0,1)$ and $k_0\in\mathbb{N}$ such that $\|x^{k+1}-\bar{x}\|
        \leq \mu\|x^k-\bar{x}\|$ for all $k\geq k_0$

    \item \textup{Q-quadratic} if
    \[
        \limsup_{k\to\infty}
        \frac{\|x^{k+1}-\bar{x}\|}
             {\|x^k-\bar{x}\|^2}
        <\infty,
    \]
    i.e., there exist $c>0$ and $k_0\in\mathbb{N}$ such that $ \|x^{k+1}-\bar{x}\|
        \leq c\|x^k-\bar{x}\|^2$ for all $k\geq k_0.$
\end{enumerate}
\end{Definition}


\subsection{Tools from Variational Analysis}
In this section, we mainly follow the notations from the classical monographs \cite{Mordukhovich06,Mordukhovich24,Rockafellar98} on Variational Analysis. Given a closed (possibly nonconvex) set $\Omega\subset \XX$ and $\ox\in \Omega$. The {\em regular normal cone} (a.k.a. Fr\'echet normal cone) to $\Omega$ at $\ox$ is defined by 
\begin{equation}\label{def:Frechet-cone}
    \widehat{N}_\Omega (\ox):=\left\{v\in \XX\;\Bigg|\;\limsup_{x\xrightarrow{\Omega}\ox}\dfrac{\la v, x - \ox\ra}{\|x-\ox\|} \le 0\right\},
\end{equation}
where $x\xrightarrow{\Omega}\ox$ means that $x\in \Omega$ and $x\to \ox$. The {\em limiting normal cone}  (a.k.a. Mordukhovich normal cone) to $\Omega$ at $\ox\in \Omega$ is given by 
\begin{equation}\label{def:limit-cone}
    N_{\Omega}(\ox) := \big\{v\in \XX \mid \exists\,  \{x_k\}\xrightarrow{\Omega}\ox, \{v_k\}\to v \text{ such that }v_k \in \widehat{N}_{\Omega}(x_k)\; \forall k\in \N\big\}.
\end{equation}
Both the regular and limiting normal cones return exactly to the convex normal cone introduced in \eqref{def:N} when the set $\Omega$ is convex, but they are different in general.

For a set-valued mapping $F: \mathbb{X} \rightrightarrows \mathbb{X}$ from $\XX$ to itself, we associate with it the {\em graph} given by
\[
\gph F := \{(x,y) \in \mathbb{X} \times \mathbb{X} \mid y\in F(x)\}.
\]
The {\em limiting coderivative} of a set-valued mapping $F:\XX \rightrightarrows \XX$ at $(\ox,\ov)\in \gph F$ is defined as
\begin{equation}\label{def:limit-code}
D^{*}F(\ox|\,\ov)(w) := \{z \in \XX \mid (z,-w) \in N_{\mathrm{gph}\,F}(\ox,\ov) \} \quad \text{for all} \quad w \in \XX.
\end{equation}
When $F$ is the convex subdifferential mapping $\partial \varphi$ and $\ov\in \partial \varphi(\ox)$, the limiting coderivative of $\partial \varphi$ at $\ox$ for $\ov$ is known as the {\em generalized Hessian} (a.k.a. {\em second-order limiting subdifferential}) \cite{Boris92,Mordukhovich06,PoliquinRock98} of $\varphi$ at $\ox$ for $\ov$ denoted by:
\begin{equation}\label{def:secondsub}
\partial^{2}\varphi(\ox|\,\ov)(w) := (D^{*}\partial \varphi)(\ox|\,\ov)(w) \quad \text{for all} \quad w \in \XX,
\end{equation}
which plays an important role in our paper and several nonsmooth Newton methods developed recently in \cite{KhanhMordukhovichPhat25,KhanhMordukhovichPhat24,MordukhovichSarabi21}. It is worth noting that if $\varphi$ is convex and twice continuously differentiable around $\ox$ with $\ov = \nabla \varphi(\ox)$, then its generalized Hessian reduces to the regular Hessian mapping
\[
\partial^{2}\varphi(\ox|\, \ov)(w) =  \nabla^2 \varphi(\ox)w \quad \text{for all} \quad w \in \XX.
\]

If $\varphi$ is not twice differentiable, designing Newton methods is a huge challenge, as the regular Hessian does not exist anymore. Naturally, replacing the Hessian by the generalized Hessian \eqref{def:secondsub} is a possibility; see, e.g., \cite{MordukhovichSarabi21,Mordukhovich24,KhanhMordukhovichPhat24,KhanhMordukhovichPhat25}. This actually leads us to the definition of {\em tilt stability} at a minimizer $\ox$ introduced by Poliquin and Rockafellar in \cite{PoliquinRock98}, which is shown to be equivalent to the positive definiteness of $\partial^2 \varphi(\ox|\,0)$; see \cite[Theorem~1.3]{PoliquinRock98} or  Theorem~\ref{theo:Tilt} below. 

\begin{Definition}[Tilt stability]\label{Den:Tilt} Given an l.s.c. convex function $\varphi:\mathbb{X} \to \Bar{\R}$, a point $\ox \in \dom \varphi$ is called a \emph{tilt-stable} minimizer of $\varphi$ if there exists $\gamma > 0$ such that the mapping 
\begin{equation*}
M_{\gamma}(v) := \argmin\{\varphi(x) - \la v,x \ra \mid x \in \B_{\gamma}(\ox)\} \quad \text{for } \quad v \in \XX
\end{equation*}
is single-valued and Lipschitz continuous on some neighborhood of $0 \in \XX$ with $M_{\gamma}(0) = \{\ox\}.$
\end{Definition}

By convexity of $\varphi$, together with the property $M_{\gamma}(0)=\{\ox\}$, it follows immediately that any tilt-stable minimizer $\ox$ is necessarily the unique global minimizer of $\varphi$. Accordingly, we refer to $\bar{x}$ as the tilt-stable minimizer of $\varphi$.

Tilt stability also lies behind the theory of nonsmooth Newton methods developed recently in \cite{MordukhovichSarabi21, KhanhMordukhovichPhat24,GfrererCOA25}. The following theorem, a direct consequence of \cite[Theorem~3.5]{BorisNghia15} and \cite[Example~13.30]{Rockafellar98}, gives a complete characterization of tilt stability and will be useful for our later analysis.  

\begin{Theorem}{\rm (A characterization of tilt stability via generalized Hessian)}\label{theo:Tilt} Consider an l.s.c. convex function $\varphi:\mathbb{X}\to \overline{\R}$ with $\ox\in \dom \varphi$ and $0\in \partial \varphi (\ox)$. Then $\ox$ is the \emph{tilt stable minimizer} of $\varphi$ if and only if the generalized Hessian $\partial^2 \varphi$ is positive definite around $(\ox,0)$ in the sense that there exist $\kappa,r>0$ satisfying 
\begin{equation}\label{eq:tilt}
    \la z,w\ra \ge \kappa\|w\|^2 \quad \forall z\in \partial^2 \varphi (x|\,v)(w),\; \forall\, (x,v)\in \mathbb{B}_{r}(\ox,0)\cap \gph \partial \varphi.
\end{equation}
\end{Theorem}

\section{Newton methods with effective subspaces for polyhedral regularization}\label{sec:polyhedral}

In this section, we propose some nonsmooth Newton methods for solving the following composite optimization problem
\begin{equation}\label{p:CP}
    \min_{x\in \XX}\quad \varphi(x):=f(x)+g(x).
\end{equation}
Throughout this section, unless otherwise stated, we assume that $f:\XX\to \R$ is a twice continuously differentiable and convex function, whereas $g:\XX\to \overline{\R}$ is a proper piecewise linear/polyhedral convex function. Recall that a proper convex function $g$ is {\em piecewise linear/polyhedral} precisely when its epigraph is polyhedral, namely, when $\epi g$ can be represented as the intersection of finitely many closed half-spaces in $\XX\times\R$; see \cite[Definition~2.47 and Theorem~2.49]{Rockafellar98}.

Recently, Nghia et al. \cite{NVV25} proposed a nonsmooth Newton method for problem \eqref{p:CP} based on effective subspaces. Given an approximation pair $(y_k,z_k)\in\mathrm{gph}\,\partial g$ sufficiently close to the optimality pair $(\bar{x},-\nabla f(\bar{x}))$, their method constructs the effective subspace $L_k:=\mathrm{par}\,\partial g^*(z_k)$ and computes a Newton direction $-d_k$ by minimizing a quadratic model over $L_k$. This approach avoids computing generalized second-order information of $g$ and achieves local quadratic convergence under tilt stability. We first recall the method of \cite{NVV25} in {\bf Algorithm~\ref{algo:poly1}}.

\begin{algorithm}[H]
\caption{Newton method with effective subspaces}
\label{algo:poly1}
Input: $x_0\in \mathbb{X}, \eta > 0$.
\begin{algorithmic}[1]
  \STATE Find $(y_k,z_k)\in \gph \partial g$ satisfying
  \begin{equation}\label{eq:bound-cond0}
      \|y_k - \ox\| + \|z_k - \oz\| \le \eta \|x_k - \ox\|.
  \end{equation}
  \STATE Define the effective subspace $L_k:={\rm par}\, \partial g^*(z_k)$ and find an optimal solution $d_k$ of the problem
  \begin{equation}\label{p:Sub}
  \min_{d\in L_k}\quad \frac{1}{2}\la \nabla^2 f(y_k)d,d\ra-\la z_k+\nabla f(y_k),d\ra. 
  \end{equation}
  \STATE Update $x_{k+1}=y_k-d_k$. 
\end{algorithmic}
\end{algorithm}

The Newton direction $-d_k$ generated by the classical Newton's method 
\begin{equation}\label{Al:Newton}
x_0\in \XX,\quad\quad x_{k+1}=x_k-d_k\quad \mbox{with}\quad d_k:=[\nabla^2 \varphi(x_k)]^{-1}\nabla \varphi (x_k),
\end{equation}
is a descent direction, given that the Hessian $\nabla^2 \varphi (x_k)$ is positive definite. In contrast, this property does not generally hold for the effective subspace Newton direction computed in {\bf Algorithm~\ref{algo:poly1}}. This is illustrated below.

\begin{Example}[Failure of descent with the effective-subspace subproblem]
\label{ex:old-subproblem-linesearch-failure}
{\rm Consider the composite problem
\[
\min_{(x_1,x_2)\in\mathbb R^2}\ \varphi(x_1,x_2):=\frac12 x_1^2-x_1x_2+x_2^2+x_1-x_2+\delta_{\mathbb R^2_+}(x_1,x_2),
\]
where $f(x_1,x_2):=\frac12 x_1^2-x_1x_2+x_2^2+x_1-x_2$ and $g(x_1,x_2):=\delta_{\mathbb R^2_+}(x_1,x_2)$. The gradient and Hessian of $f$ are, respectively,
\[
\nabla f(x_1,x_2)
=
\begin{pmatrix}
x_1-x_2+1\\
-x_1+2x_2-1
\end{pmatrix} \quad \text{and} \quad
\nabla^2 f(x)=\begin{pmatrix}1&-1\\-1&2\end{pmatrix}.
\] 
Let $\alpha=1$ and choose $x_0=(0,1)$. Since $\nabla f(x_0)=(0,1)$, we get
\[
y_0=\prox_g(x_0-\nabla f(x_0))
=P_{\mathbb R^2_+}(0,0)=(0,0).
\]
The corresponding subgradient is
\[
\begin{aligned}
z_0 &=x_0-y_0-\nabla f(x_0)=(0,0), \text{ satisfying}\\
z_0 &\in \mathbb R^2_- = N_{\mathbb R^2_+}(0,0)= \partial \delta_{\R^2_+}(0,0) = \partial g(y_0).
\end{aligned}
\]

We now compute the effective-subspace Newton direction. Since $g=\delta_{\mathbb R^2_+}$, we have $g^*=\delta_{\mathbb R^2_-}$ and thus
$\partial g^*(z_0)=N_{\mathbb R^2_-}(0,0)=\R^2_+$. This gives 
\[
L_0:=\para\partial g^*(z_0)= \R^2.
\]
Moreover, $\nabla f(y_0)=(1,-1)$, and $z_0+\nabla f(y_0)=(1,-1)$. Hence, the effective-subspace constrained subproblem \eqref{p:Sub} reduces to
\[
\min_{d\in\mathbb R^2}\,\,\,\,\, \frac{1}{2} d^\top
\begin{pmatrix}
1 & -1\\
-1 & 2
\end{pmatrix}
d
-
\left\langle
\begin{pmatrix}
1\
-1
\end{pmatrix},
d
\right\rangle,
\]
whose unique minimizer is $d_0:=(1,0)$. However, vector $-d_0$ is not a descent direction for
$\varphi$ at $y_0$. Indeed, the one-sided directional derivative \eqref{eq:direc-deri} of $g$ at $y_0$ in the direction of $v \in\R^2$ is
\begin{equation}\label{eq:ex_directional}
g'(y_0;v)
=\delta'_{\mathbb R^2_+}(y_0;v)
=
\begin{cases}
    0, & v\in \mathbb R^2_+,\\
    +\infty, & v\notin \mathbb R^2_+.
\end{cases}
\end{equation}
Since $-d_0=(-1,0)\notin\mathbb R^2_+$, we have
\[
\varphi'(y_0;-d_0)
=\langle \nabla f(y_0),-d_0\rangle+g'(y_0;-d_0)
=+\infty.
\]
Thus, $-d_0$ is not a descent direction for $\varphi$ at $y_k$. Furthermore, for any $\beta>0$,
\[
g(y_0-\beta d_0) =  g\left(-\beta,0\right) = +\infty
\]
and therefore $\varphi(y_0-\beta d_0)=+\infty$. Consequently, any Armijo-type line search starting from $y_0$ along $-d_0$ fails to terminate after finitely many steps.}
\end{Example}

This is one of the crucial motivations for our modified Newton method below.

\begin{algorithm}[H]
\caption{Newton method with effective sets}
\label{algo:modified-Dk}
Input: $x_0\in \mathbb{X}, \alpha > 0$.
\begin{algorithmic}[1]
  \STATE Compute 
  \begin{equation}\label{eq:ykzk}
  y_k := \prox_{\alpha g}\big(x_k-\alpha\nabla f(x_k)\big) \quad \text{and} \quad z_k := \frac{x_k-y_k}{\alpha} - \nabla f(x_k).
  \end{equation}

  \STATE Let $\mathcal{D}_k := y_k - \partial g^* (z_k)$ and compute $d_k\in\mathbb R^{n}$ the solution of
    \begin{equation}\label{eq:new_constrained_reg_quad}
        \min_{d\in\mathcal{D}_k}\quad       \frac{1}{2} \la\nabla^2 f(y_k)d,d\ra
            -
            \left\langle z_k+\nabla f(y_k),d\right\rangle.
    \end{equation}
  \STATE Update $x_{k+1}=y_k-d_k$. 
\end{algorithmic}
\label{algo:mod_local}
\end{algorithm}

Any solution $d_k$ of the subproblem in \eqref{eq:new_constrained_reg_quad} satisfies the first-order optimality condition, i.e.,
\begin{equation*}
    - \nabla^2 f(y_k)d_k + z_k + \nabla f(y_k) \in N_{\mathcal{D}_k}(d_k),
\end{equation*}
which is equivalent to having 
\begin{equation*}
    \langle -\nabla^2 f(y_k) d_k + z_k + \nabla f(y_k), d-d_k \rangle \le 0 \qquad \forall\, d\in \mathcal{D}_k,
\end{equation*}
or one may prefer
\begin{equation}\label{eq:first-optim}
    \langle \nabla^2 f(y_k) d_k - \big(z_k + \nabla f(y_k)\big), d-d_k \rangle \ge 0 \qquad \forall\, d\in \mathcal{D}_k.
\end{equation}

Next, let us examine the effective set $\mathcal{D}_k$. 
\begin{Proposition}\label{prop:D_k}
When $g$ is merely a proper l.s.c. convex function, we have the following assertions for each $k\in \N\cup \{0\}$:
\begin{enumerate}
    \item[{\rm {\bf (i)}}] $\Span \mathcal{D}_k = \Span \big(\partial g^* (z_k) - y_k\big) =  L_k$.
    \item[{\rm {\bf (ii)}}] $\mathcal{D}_k \subset \mathrm{Ker}\, \partial^2 g(y_k|\,z_k)$.
\end{enumerate}
\end{Proposition}
\begin{proof}
Assertion {\bf (i)} is clear from the definition of $\mathcal{D}_k$ and $L_k$. To verify {\bf (ii)}, pick $d\in \mathcal{D}_k$ and write $d=y_k - x$ for some $x\in \partial g^* (z_k)$. By the Fenchel-Young's identity, we have $(y_k,z_k),(x,z_k)\in \gph \partial g$. Now take any pair $(v,w)\in \gph \partial g$ and observe by the monotonicity of $\partial g$ that 
\begin{equation*}
    \la v-y_k,w-z_k\ra + \la d,w - z_k\ra = \la v-x,w-z_k\ra \ge 0,
\end{equation*}
and hence $\la -d,w-z_k\ra \le \la v-y_k,w-z_k\ra$. Therefore,
\begin{equation*}
\begin{aligned}
    \big\la (0,-d),(v-y_k,w-z_k)\big\ra = \la -d, w-z_k\ra &\le \la v-y_k,w-z_k\ra\\ &\le \dfrac{1}{2}\big(\|v-y_k\|^2 + \|w-z_k\|^2\big) \\
    &= \dfrac{1}{2}\big\|(v-y_k,w-z_k)\big\|^2,\quad \forall\, (v,w)\in \gph \partial g.
\end{aligned}
\end{equation*}
This yields in turn that $$\limsup_{(v,w)\xrightarrow{{\rm gph}\,\partial g}(y_k,z_k)} \dfrac{\big\la (0,-d),(v,w)-(y_k,z_k)\big\ra}{\big\|(v,w)-(y_k,z_k)\big\|}= 0.$$
By \eqref{def:Frechet-cone}--\eqref{def:limit-cone}, we have $(0,-d)\in \widehat{N}_{{\rm gph}\, \partial g}(y_k,z_k)\subset N_{{\rm gph}\, \partial g}(y_k,z_k)$. It follows that $0\in \partial^2 g(y_k|\, z_k)(d)$, i.e., $d\in \mathrm{Ker}\, \partial^2 g(y_k|\, z_k)$. The proof is complete.
\end{proof}

\begin{Remark}\rm 
The replacement of the effective subspace
$L_k=\operatorname{par}\partial g^*(z_k)$ in {\bf Algorithm~\ref{algo:poly1}} by the effective set $\mathcal D_k=y_k-\partial g^*(z_k)$ in {\bf Algorithm~\ref{algo:modified-Dk}}
is motivated by the following two observations.

\begin{enumerate}
\item
Since $\operatorname{span}\mathcal D_k=L_k$, the set $\mathcal D_k$
generates the same effective directions as $L_k$. However, while $L_k$
retains only the linear directions parallel to the polyhedral face
$\partial g^*(z_k)$, the set $\mathcal D_k$ preserves its boundary and
one-sided feasible geometry. Thus, it rules out directions that belong
to the whole subspace $L_k$ but leave the relevant face.

\item
If $d_k\in\mathcal D_k$, then
$y_k,y_k-d_k\in\partial g^*(z_k)$. By convexity of $\partial g^* (z_k)$, we also have
\[
    y_k-\beta d_k
    =(1-\beta)y_k+\beta(y_k-d_k)
    \in\partial g^*(z_k)
    \qquad\forall\,\beta\in[0,1].
\]
The Fenchel-Young's identity \eqref{def:fenyou} asserts that
\begin{equation*}
\begin{aligned}
    g(y_k - \beta d_k) + g^* (z_k) &= \langle y_k - \beta d_k, z_k\rangle,\\
    g(y_k) + g^* (z_k) &= \langle y_k,z_k\rangle.
\end{aligned}
\end{equation*}
Subtracting the second equation by the first one yields
\begin{equation*}
    g(y_k-\beta d_k) - g(y_k) = -\beta \langle d_k, z_k\rangle,
\end{equation*}
i.e.,
\begin{equation}\label{eq:g-affine}
    g(y_k-\beta d_k)
    =
    g(y_k)-\beta\langle d_k,z_k\rangle,
    \qquad \beta\in[0,1].
\end{equation}
Hence, the map $\beta \mapsto g(y_k-\beta d_k)$ is affine with respect to $\beta\in [0,1]$. This property is particularly useful for the
line-search procedure in {\bf Algorithm~\ref{alg:Mod_regularized_ESN}}, since the Armijo condition for
$\varphi=f+g$ can then be implemented using evaluations of the smooth
component $f$ only.
\end{enumerate}
\end{Remark}

The following example demonstrates that the modified Newton step in {\bf Algorithm~\ref{algo:modified-Dk}} is not necessarily either the coderivative-based Newton step in \cite[Algorithm~2,~Step~4]{KhanhMordukhovichPhat25} or the graphical derivative-based step in \cite[Algorithm~5.1,~Step~2]{MordukhovichSarabi21}.

\begin{Example}\label{ex:not_coderivative_baed}\rm 
Let $\mathbb{X}:=\mathbb{R}^2$, $C:=[0,1]^2$, and consider the convex functions
\[
    g:=\delta_C,
    \qquad
    f(x_1,x_2):=\frac{1}{2}(x_1^2+x_2^2)-\frac{3}{2}x_1.
\]
Then $f$ is $1$-strongly convex and twice continuously differentiable, with
\begin{equation*}
    \nabla f(x_1,x_2)
    =
    \begin{pmatrix}
        x_1-3/2\\
        x_2
    \end{pmatrix},
    \qquad
    \nabla^2 f(x_1,x_2)=\mathbb{I}_2.
\end{equation*}
The unique minimizer of $\varphi:=f+g$ is $\bar{x}:=(1,0)$. Indeed,
$\nabla f(\bar{x})=\begin{pmatrix}-1/2 \\0\end{pmatrix}$, and hence
\begin{equation*}
    \bar{z}:=-\nabla f(\bar{x})
    =
    (1/2,0)
    \in N_C(\bar{x})
    =
    \partial g(\bar{x}).
\end{equation*}

We apply {\bf Algorithm~\ref{algo:modified-Dk}} with the initial point $x_0:=(0,0)$. For simplicity, let us 
choose the pair $(y_0,z_0)$ by the proximal-gradient step with
$\alpha:=1/3$. Since $g=\delta_C$, one has
$\prox_{\alpha g}=P_C$. Therefore, 
\begin{equation*}
    y_0
    =
    P_C\big(x_0-\alpha\nabla f(x_0)\big) = P_C\big((0,0)-1/3 (-3/2,0)\big) =
    P_C(1/2,0)
    = (1/2,0).
\end{equation*}
The associated subgradient is
\begin{equation*}
    z_0=
    \frac{x_0-y_0}{\alpha}-\nabla f(x_0) \notag=
    3(-1/2,0) - (-3/2,0) = (0,0).
\end{equation*}

We next compute the effective set $\mathcal{D}_0$. Since
$g^*=\sigma_C$ and $z_0=(0,0)$, we have
\begin{equation*}
    \partial g^*(z_0)
    =
    \partial \sigma_C(0,0)
    =
    C
    =
    [0,1]^2.
\end{equation*}
Consequently,
\begin{equation*}
    \mathcal{D}_0:=
    y_0 - \partial g^* (z_0)=
    y_0-[0,1]^2
    = (1/2,0) - [0,1]\times [0,1]=
    [-1/2,1/2]\times [-1,0].
\end{equation*}

At $y_0$, we have
\begin{equation}\label{eq:mate1}
    z_0+\nabla f(y_0)
    =
    \begin{pmatrix} 0 \\ 0 \end{pmatrix} + \begin{pmatrix} 1/2 - 3/2 \\ 0 \end{pmatrix} = \begin{pmatrix} -1 \\ 0 \end{pmatrix},
    \qquad
    \nabla^2 f(y_0)=\mathbb{I}_2.
\end{equation}
Hence the quadratic subproblem in {\bf Algorithm~\ref{algo:modified-Dk}} becomes
\begin{equation*}
    \min_{(d_1,d_2)\in [-1/2,1/2]\times [-1,0]}\,\,\,
        \frac{1}{2}(d_1^2+d_2^2)+d_1,
\end{equation*}
whose unique minimizer is $d_0:= (-1/2,0)$. Using \eqref{eq:mate1}, let us compute the residual
\begin{equation*}
    r_0
    :=
    \nabla^2 f(y_0)d_0 -z_0-\nabla f(y_0)  = -(1/2,0) + (1,0) = (1/2,0).
\end{equation*}

{\bf 1.} We now show that the pair $(r_0,d_0) = \big((1/2,0),(-1/2,0)\big)$ does not belong to
$N_{\mathrm{gph}\,\partial g}(y_0,z_0)$. Indeed,  the critical cone to $C$ at $y_0$ for $z_0$ is
\begin{equation*}
    K
    :=
    T_C(y_0)\cap\{z_0\}^{\perp}=
    T_C(1/2,0)
    =
    \mathbb{R}\times\mathbb{R}_+.
\end{equation*}
The only two closed faces of $K$ are
\begin{equation*}
    F:=\mathbb{R}\times\{0\}
    \qquad\mbox{and}\qquad
    K=\mathbb{R}\times\mathbb{R}_+.
\end{equation*}
Since $g=\delta_C$, we may employ \cite[Proposition~4.4]{PoliquinRock98} to derive that 
\begin{equation*}
    N_{\mathrm{gph}\, \partial g}(y_0,z_0)
    =\big\{(F_1-F_2)^*\times (F_1-F_2)|\, \mbox{$F_1, F_2$ are closed faces of $K$ with } F_2\subset F_1\big\}.
    \label{eq:counter-graph-normal-formula}
\end{equation*}
The only possible difference sets $F_1-F_2$ are $F-F$, $K-F$, and $K-K$, which are, respectively,
\begin{equation*}
    F,\qquad K,\qquad \text{and} \qquad\mathbb{R}^2.
\end{equation*}
Their polars are, respectively,
\begin{equation*}
    F^*=\{0\}\times\mathbb{R},
    \qquad
    K^*=\{0\}\times\mathbb{R}_-,
    \qquad \text{and}\qquad
    (\mathbb{R}^2)^*=\{(0,0)\}.
\end{equation*}
Consequently, every $(u,v)\in N_{\mathrm{gph}\, \partial g}(y_0,z_0)$ necessarily
satisfies $u_1=0$.
On the other hand, the first component of the pair $(r_0,d_0)$ is $r_0 = (1/2,0)$,
whose first coordinate is nonzero. Therefore,
\[
    \big(
        \nabla^2 f(y_0)d_0-z_0-\nabla f(y_0),
        d_0
    \big) = (r_0,d_0)
    \notin
    N_{\mathrm{gph}\,\partial g}(y_0,z_0),
\]
i.e.,
\begin{equation}\label{eq:not-coderivative-based}
    \nabla^2 f(y_0)d_0 - z_0 - \nabla f(y_0) \notin \partial^2 g(y_0|\,z_0)(-d_0).
\end{equation}
This shows that the direction $-d_0$ generated by {\bf Algorithm~\ref{algo:modified-Dk}} does not,
in general, satisfy the coderivative-based Newton inclusion (e.g., see \cite[Algorithm~2,~Step~4]{KhanhMordukhovichPhat25}), even when
$f$ is strongly convex quadratic and $g$ is the indicator function of
a polyhedral convex set.

{\bf 2.} Let us recall from \cite[Equation~(2.5)]{MordukhovichSarabi21} the definition of the {\em graphical derivative} of $\partial g$ at $(y,z)\in \gph \partial g$ by
\begin{equation*}
    (D\partial g)(y|\, z)(u) := \big\{v\in \XX|\, (u,v)\in T_{\mathrm{gph}\, \partial g}(y,z) \quad \text{for all}\quad u\in \XX\big\}.
\end{equation*}
By \cite[Proposition~13.9~and~Theorem~13.57]{Rockafellar98}, we have the inclusion
\begin{equation*}
    (D\partial g)(y_0|\, z_0)(-d_0) \subset \partial^2 g(y_0|\,z_0)(-d_0).
\end{equation*}
Combining with \eqref{eq:not-coderivative-based}, we arrive at 
\begin{equation*}
    \nabla^2 f(y_0)d_0 - z_0 - \nabla f(y_0) \notin (D\partial g)(y_0|\,z_0)(-d_0).
\end{equation*}
Thus $-d_0$ also violates the graphical derivative-based Newton inclusion in \cite[Algorithm~5.1,~Step~2]{MordukhovichSarabi21}.
\end{Example}

\begin{Theorem}[Quadratic convergence of the Newton method with effective set] \label{theo:quad-poly1}
Consider problem \eqref{p:CP} and suppose that $\bar{x}$ is the tilt stable minimizer of the function $\varphi$ in \eqref{p:CP}, i.e., there exist $\kappa,r>0$ for which \eqref{eq:tilt} holds, and that the Hessian $\nabla^2 f(x)$ is Lipschitz continuous around $\ox$. Then there exists $\ve>0$ such that for any starting point $x_0\in \mathbb{B}_{\ve}(\ox)$, we have the assertions: 

\begin{enumerate}[\bf(i)]
    \item The optimal solution $\{d_k\}$ of problem~\eqref{eq:new_constrained_reg_quad} is well-defined and unique.

    \item  The iterate sequence $\{x_k\}$ and the objective value sequence $\{\varphi(x_k)\}$ generated by {\bf Algorithm~\ref{algo:mod_local}} converges quadratically to $\ox$ and $\varphi(\ox)$, respectively. 

    \item  In case $f(x) = \dfrac{1}{2}\la Ax,x\ra + \la b,x\ra + c$, where $A \in \R^{n \times n}$ is a symmetric positive semidefinite matrix, $b\in \R^n$ and $c\in \R$, then the sequence $\{x_k\}$ generated by {\bf Algorithm~\ref{algo:mod_local}} terminates at $\ox$ after finitely many steps. Consequently, the sequence $\{\varphi(x_k)\}$ terminates at $\varphi(\ox)$ after finitely many steps.
    
\end{enumerate}

\end{Theorem}

\begin{proof}
Since $g$ and $g^*$ are polyhedral convex functions with $\ox\in \dom g$ and $\oz:=-\nabla f(\ox)\in \dom g^*$, by \cite[Proposition~3.3(a)]{MordukhovichSarabi16} we may shrink $r>0$ if necessary so that 
\begin{equation}\label{eq:inclu-polyhedral}
\begin{aligned}
    \partial g(x) &\subset \partial g(\ox)\quad \text{ for all }x\in \dom g \cap \mathbb{B}_r (\ox),\text{ and}\\
    \partial g^* (z) &\subset \partial g^* (\oz) \quad \text{ for all }z\in \dom g^* \cap \mathbb{B}_r (\oz).
\end{aligned}
\end{equation}
Further, let us suppose that $\nabla f$ and $\nabla^2 f$ are Lipschitz continuous with some modulus $\ell>0$ on the closed ball $\mathbb{B}_r (\ox)$. Set $\eta:=\dfrac{(2+\alpha)(1+\ell \alpha)}{\alpha}$.

We choose a constant 
\begin{equation}\label{eq:epsilon_choice}
    0<\ve < \min\left\{r;\;\dfrac{r}{2\eta\left(1+\ell\right)};\; \dfrac{\kappa}{\ell \eta^2}\right\},
\end{equation}
and define $\overline{C}:=\dfrac{\ell\eta^2}{2\kappa}$. Obviously, one has $2\overline{C}\ve < 1$. Fix $x_0\in \mathbb{B}_{\varepsilon}(\ox)$.

First, let us verify {\bf (i)} and the quadratic convergence of $x_k \to \ox$ as in {\bf (ii)}. By induction, assume that $x_k\in \mathbb{B}_{\ve}(\ox)$ for some $k\in \N \cup \{0\}$, and prove that the optimal solution $d_k$ of problem \eqref{p:Sub} exists and is unique, and that
$\|x_{k+1}-\ox\|\le \overline{C}\|x_k -\ox\|^2 < \ve$. The nonexpansiveness property of $\prox_{\alpha g}$ implies that
\begin{equation}\label{eq:y0xbar-rem-1}
\begin{aligned}
    \|y_k - \ox\| &= \left\|\prox_{\alpha g} \big(x_k - \alpha \nabla f(x_k)\big) - \prox_{\alpha g} \big(\ox - \alpha \nabla f(\ox)\big)\right\| \\
    &\le \big\|(x_k -\ox) - \alpha \big(\nabla f(x_k) - \nabla f(\ox)\big)\big\| \le \big(1+\ell {\alpha}\big)\|x_k - \ox\|.
\end{aligned}
\end{equation}
Moreover, it follows from \eqref{eq:y0xbar-rem-1} and the Lipschitz continuity of $\nabla f$ on $\mathbb{B}_{\varepsilon}(\ox)$ that
\begin{equation}\label{eq:z0zbar}
\begin{aligned}
    \|z_k - \oz\| = \left\|\dfrac{x_k - y_k}{\alpha} - \nabla f(x_k) + \nabla f(\ox)\right\| &\le \dfrac{\|x_k - \ox\|}{\alpha}+\dfrac{\|y_k - \ox\|}{\alpha} + \ell \|x_k - \ox\|\le \dfrac{2(1+\ell\alpha)}{\alpha}\|x_k-\ox\|,
\end{aligned}
\end{equation} 
which together with \eqref{eq:y0xbar-rem-1} and \eqref{eq:epsilon_choice} and the definition of $\eta:=\frac{(2+\alpha)(1+\ell\alpha)}{\alpha}$ yield
\begin{equation}\label{eq:y0_modESN}
\begin{aligned}
    \|y_k - \ox\| + \|z_k - \oz\| \le \eta\|x_k-\ox\| \le \eta \ve < \dfrac{r}{2}.
\end{aligned}
\end{equation}
With $v_k:=\nabla f(y_k)+z_k$, we have $v_k \in \partial \varphi (y_k)$. Its norm is bounded as follows:
\begin{equation}\label{eq:v0_modESN}
    \|v_k\|=\|\nabla f(y_k)-\nabla f(\ox) +z_k-\oz\|\le \ell\|y_k-\ox\|+\|z_k-\oz\|\le (1+\ell)\eta\ve<\frac{r}{2}.
\end{equation}
Consequently, the pair $(y_k,v_k)$ belongs to $\gph \partial \varphi \cap \mathbb{B}_r (\ox,0)$. By employing the second-order characterization \eqref{eq:tilt} of tilt stability at $\ox$ and the generalized Hessian sum rule in \cite[Proposition~1.121]{Mordukhovich06}, one gets the relation
\[
\la \nabla^2 f(y_k)w,w\ra+\la u,w\ra \ge \kk\|w\|^2\quad \mbox{for all}\quad u\in \partial^2 g(y_k|\, z_k)(w). 
\]
In particular at $u=0$, we utilize the identity $\Ker \partial^2 g(y_k|\,z_k) = \mathrm{par}\, \partial g^* (z_k)$, which follows from \cite[Equation~(4.3)]{BorisRock12} to deduce further that
\begin{equation}\label{eq:Noz_modESN}
\la \nabla^2 f(y_k)w,w\ra\ge \kk\|w\|^2\quad \mbox{for all}\quad w\in \para\partial g^*(z_k) = L_k = \Span \mathcal{D}_k.
\end{equation}
Set $\Theta_k (d):= \frac{1}{2}\la \nabla^2 f(y_k)d,d\rangle - \la z_k + \nabla f(y_k),d\ra$, for $d\in \XX$. The assertion in \eqref{eq:Noz_modESN} implies that the lower-level set $$\big\{d\in \mathcal{D}_k |\, \Theta_k (d) \le \Theta_k (0) = 0\big\}
$$
is nonempty, closed, and bounded. In particular, it is a compact set. Because $\Theta_k(\cdot)$ is lower semicontinuous and minimizing $\Theta_k (\cdot)$ on such lower-lovel set is equivalent to minimizing $\Theta_k(\cdot)$ on $\mathcal{D}_k$, the optimization problem in \eqref{eq:new_constrained_reg_quad} has at least one solution over $\mathcal{D}_k$. Moreover, we claim that this solution is unique. Indeed, suppose by contrary that $d_1,d_2\in \mathcal{D}_k$ are both minimizers. It follows from \eqref{eq:first-optim} and \eqref{eq:Noz_modESN} that \begin{equation*}
\begin{aligned}
     \kappa \|d_1 - d_2\|^2 &\le \la \nabla^2 f(y_k) (d_1 - d_2),d_1 - d_2\rangle \\
     &= \langle \nabla^2 f(y_k) d_1 - (z_k + \nabla f(y_k)) , d_1 - d_2 \rangle - \langle \nabla^2 f(y_k) d_2 - (z_k + \nabla f(y_k)), d_1 - d_2\rangle \\
     &\le 0 -0 = 0.
\end{aligned}
\end{equation*}
This yields $d_1 = d_2$, and thus guarantees the uniqueness of the optimal solution $d_k$ to the problem in~\eqref{eq:new_constrained_reg_quad}.

Next, we verify the estimates $\|x_{k+1}-\ox\|\le \overline{C}\|x_k-\ox\|^2 < \ve$. Let us denote $H_k:= \nabla^2f(y_k)$. The relations in \eqref{eq:inclu-polyhedral} and \eqref{eq:y0_modESN} together yield the inclusions 
\begin{equation}\label{eq:absorp-gg*}
\begin{aligned}
z_k\in \partial g(y_k) \subset \partial g(\ox), \quad\text{and}\quad
\partial g^* (z_k) \subset \partial g^* (\oz),
\end{aligned}
\end{equation}
which imply further by Fenchel-Young's identity \eqref{def:fenyou} the inclusions  $\ox,y_k \in \partial g^* (z_k)$. In particular, we have $y_k - \ox\in \mathcal{D}_k$. It follows by employing \eqref{eq:first-optim} that
\begin{equation*}
    \langle H_k d_k - (z_k+\nabla f(y_k)), (y_k-\ox) - d_k\rangle \ge 0.
\end{equation*}
Adding $\langle H_k (y_k-\ox - d_k),y_k - \ox - d_k\rangle$ to both sides,
\begin{equation}\label{eq:connect1}
\begin{aligned}
    \la H_k(y_k -\ox) - (z_k + \nabla f(y_k)),y_k - \ox-d_k\ra &\geq \la H_k (y_k - \ox - d_k),y_k - \ox - d_k\ra \\
    &=\la H_k(d_k - (y_k -\ox)), d_k - (y_k -\ox)\ra.
\end{aligned}
\end{equation}

On the one hand, since $\oz = -\nabla f(\ox)$, the left-hand side of \eqref{eq:connect1} can be rewritten as
\begin{equation}\label{eq:connect_equiv}
\begin{aligned}
    &\la H_k(y_k -\ox) - (z_k + \nabla f(y_k)),y_k - \ox-d_k\ra \\
    &= \la H_k(y_k - \ox) + \nabla f(\ox) -\nabla f(y_k),y_k - \ox-d_k\ra - \la z_k-\oz,y_k - \ox-d_k\ra.
\end{aligned}
\end{equation}
Next, we claim that
\begin{equation}\label{eq:projzkzbar-2_modESN}
 \la z_k-\oz,y_k - \ox-d_k\ra=0.
\end{equation}
Indeed, it follows from $d_k \in \mathcal{D}_k$ and \eqref{eq:absorp-gg*} that $y_k - d_k \in \partial g^*(z_k) \subset \partial g^*(\oz)$. We also get from \eqref{eq:absorp-gg*} that $\ox\in \partial g^* (z_k) \subset \partial g^* (\oz)$. Consequently, the inclusion $\{y_k - d_k,\ox\}\subset \partial g^* (z_k) \cap \partial g^* (\oz)$ holds. Together with the monotonicity of $\partial g^*$, it yields the estimates
\begin{equation*}
\begin{aligned}
    \la z_k - \bar{z}, y_k - d_k - \ox\ra &\ge 0,\\
    \la \oz - z_k, y_k-d_k-\ox \ra &\ge 0,
\end{aligned}
\end{equation*}
and hence the identity
\[
\langle z_k - \oz, y_k - d_k - \ox\rangle =0,
\]
i.e., \eqref{eq:projzkzbar-2_modESN} is verified. Next, the Lipschitz continuity of $\nabla^2 f$ on $\mathbb{B}_r (\ox)$ and \eqref{eq:y0_modESN} imply that
\begin{equation}
\begin{aligned}
    \left\|\nabla^2 f(y_k)(y_k-\ox) + \nabla f(\ox)-\nabla f(y_k)\right\|
    &=\left\|\int_0^1 \left[\nabla^2 f (y_k)-\nabla^2 f\big(y_k + t(\ox-y_k)\big)\right](y_k-\ox)\,dt\right\|\\
    &\le \int_0^1 \Big\|\nabla^2 f (y_k)-\nabla^2 f\big(y_k + t(\ox-y_k)\big)\Big\| \|y_k-\ox\|\, dt\\
    &\le \dfrac{\ell}{2}\|y_k-\ox\|^2.
    \label{eq:1-di-2_modESN}
\end{aligned}
\end{equation}
The relations in \eqref{eq:connect_equiv}--\eqref{eq:1-di-2_modESN} yield via the Cauchy-Schwarz inequality the estimate
\begin{equation}\label{eq:left-314}
    \la H_k(y_k -\ox) - (z_k + \nabla f(y_k)),y_k - \ox-d_k\ra \le \dfrac{\ell}{2}\|y_k-\ox\|^2\, \|y_k - \ox - d_k\|.
\end{equation}
By combining \eqref{eq:connect1} with \eqref{eq:left-314}, we arrive at 
\begin{equation}\label{eq:onehand}
    \langle H_k(d_k - (y_k - \ox)), d_k - (y_k - \ox)\rangle \le \dfrac{\ell}{2}\|y_k - \ox\|^2\, \|y_k - \ox-d_k\|.
\end{equation}

On the other hand, since $d_k - (y_k -\ox) \in \mathcal{D}_k - \mathcal{D}_k \subset \mathrm{span}\, \mathcal{D}_k$, applying \eqref{eq:Noz_modESN} with $w:=d_k - (y_k -\ox)$ gives us that 
\begin{equation}\label{eq:tilt_connect}
    \big\la H_k(d_k - (y_k -\ox)),d_k - (y_k -\ox)\big\ra \ge \kappa \|d_k - (y_k -\ox)\|^2.
\end{equation}

We may now combine the estimates in \eqref{eq:epsilon_choice}, \eqref{eq:y0_modESN}, and \eqref{eq:onehand}--\eqref{eq:tilt_connect} to derive that 
\begin{equation}\label{eq:k+1-less-epsilon}
\begin{aligned}
&\|x_{k+1}-\ox\| = \|y_k - d_k - \ox\| = \|d_k - (y_k - \ox)\|\le \dfrac{\ell}{2\kappa}\|y_k-\ox\|^2 \\
&\le \dfrac{\ell \eta^2}{2\kappa} \|x_k-\ox\|^2 = \overline{C}\|x_k-\ox\|^2 \le  \dfrac{1}{2}\|x_k-\ox\| < \varepsilon.
\end{aligned}
\end{equation}
Thus, the induction step are proved. Consequently, the well-definedness of the optimal solution $d_k$ of \eqref{eq:new_constrained_reg_quad}, and the estimate $\|x_{k+1}-\ox\|\le \overline{C}\|x_k-\ox\|^2$, where $\overline{C}$ is independent of $k$, are also verified. Hence, the sequence $\{x_k\}$ converges quadratically to $\ox$.

Next, we show the quadratic convergence of the objective value sequence $\{\varphi (x_k)\}$ to $\varphi (\ox)$. Since $x_k \to \ox$, we have $\{x_k\}\subset \mathbb{B}_{\varepsilon}(\ox)$ for all $k\in \N$ large enough. This guarantees the validation of all the estimates established above for each $k\in \N$ large. In particular, as $\oz \in \partial g(\ox) \cap \partial g(y_k-d_k)$, the Fenchel-Young's identity \eqref{def:fenyou} gives
\begin{equation*}
\begin{aligned}
    g(y_k - d_k) + g^*(\oz) &= \langle \oz, y_k - d_k\ra, \text{ and}\\
    g(\ox) + g^*(\oz) &= \langle \oz, \ox\ra.
\end{aligned}
\end{equation*}
Subtracting lends us the estimate
\[
g(x_{k+1})-g(\ox)=g(y_k -d_k)-g(\ox)= \langle \bar z,y_k -d_k-\ox\ra= \langle \bar z,x_{k+1}-\ox\ra.
\]
Because $\bar z=-\nabla f(\ox)$ and $\nabla f$ is Lipschitz continuous on $\mathbb{B}_r(\ox)$ with modulus $\ell$, we derive from the above and \eqref{eq:k+1-less-epsilon} that
\begin{equation}
\begin{aligned}
\varphi(x_{k+1})-\varphi(\ox)
&=f(x_{k+1})-f(\ox)+g(x_{k+1})-g(\ox) \\
&=f(x_{k+1})-f(\ox)-\langle \nabla f(\ox),x_{k+1}-\ox\ra \leq \dfrac{\ell}{2}\|x_{k+1}-\ox\|^2 \leq \dfrac{\ell}{2}\Bar{C}^2\|x_k-\ox\|^4 \label{eq:varphi_quad1}.
\end{aligned}
\end{equation}
Moreover, since \eqref{eq:tilt} holds, the results in \cite[Theorems~3.2~and~3.5]{BorisNghia15} give rise to the estimate
\begin{equation}\label{eq:quad_growth}
\varphi (x_k) - \varphi(\ox) \geq \dfrac{\kappa}{2}\|x_k-\ox\|^2\quad \text{ for all }\,\, k\in \N \text{ large} .
\end{equation}
Combining \eqref{eq:varphi_quad1} and \eqref{eq:quad_growth} gives
\[
\varphi(x_{k+1})-\varphi(\ox)\leq \dfrac{\ell}{2}\Bar{C}^2\|x_k-\ox\|^4 \leq \dfrac{2\ell}{\kappa^2}\Bar{C}^2 \big[\varphi (x_k) - \varphi(\ox)\big]^2,
\]
which shows the local quadratic convergence rate of $\{\varphi(x_k)\}$ to $\varphi (\ox)$. Hence, the claims in {\bf (ii)} are thus justified.

Finally, to show that {\bf (iii)} holds. Notice that when $f(x) = \dfrac{1}{2}\la Ax,x\ra + \la b,x\ra + c$, we have 
\begin{equation}\label{eq:quad_norm}
H_k(y_k-\ox) + \nabla f(\ox)-\nabla f(y_k) = A(y_k-\ox) + A\ox + b- Ay_k -b = 0. 
\end{equation}
Thus, combining \eqref{eq:connect1}--\eqref{eq:projzkzbar-2_modESN}, \eqref{eq:tilt_connect}, and \eqref{eq:quad_norm} gives us that
\begin{equation*}
\|x_{k+1}-\ox\|^2 = \|y_k -\ox - d_k\|^2\le \dfrac{1}{\kappa}\big\la H_k(y_k - \ox) + \nabla f(\ox) -\nabla f(y_k),y_k - \ox-d_k\big\ra= 0.
\end{equation*}
Hence, we have $x_{k+1} = \ox$ and thus $\varphi(x_{k+1}) = \varphi(\ox)$ after finitely many steps. The proof is complete.
\end{proof}
\vspace{0.05in}

\begin{Remark}\rm 
In the proof of Theorem~\ref{theo:quad-poly1}, since the estimates involving $(x_k,y_k,z_k,d_k)$ in \eqref{eq:y0xbar-rem-1}--\eqref{eq:k+1-less-epsilon} are valid whenever $x_k\in \mathbb{B}_{\varepsilon}(\ox)$, and the induction argument therein has already established that the entire sequence $\{x_k\}_{k=0,1,\ldots}$ remains in $\mathbb{B}_{\varepsilon}(\ox)$, it follows that the estimates in \eqref{eq:y0xbar-rem-1}--\eqref{eq:k+1-less-epsilon} hold for every $k=0,1,\ldots$
\end{Remark}

\vspace{0.05in}


\begin{Proposition}[Locally descent direction]
\label{prop:local-descent}
Let $x_0\in \mathbb{B}_{\varepsilon}(\ox)$, where $\varepsilon>0$ is chosen as in \eqref{eq:epsilon_choice}, and let $\{x_k\},\{d_k\}$ be the sequences generated by {\bf Algorithm~\ref{algo:modified-Dk}}. Suppose that $\bar{x}$ is the tilt stable minimizer of the function $\varphi$, i.e., \eqref{eq:tilt} holds for some $\kappa>0$. Then it holds that
\[
\varphi'(y_k;-d_k) =-\la z_k+\nabla f(y_k),d_k\ra \le -\kappa \|d_k\|^2\qquad \forall\, k=0,1,\ldots
\]
That is, $-d_k$, if nonzero, is a descent direction for $\varphi$ at $y_k$.
\end{Proposition}

\begin{proof}
Since $f$ and $g$ are convex, the composite function $\varphi$ is directionally differentiable. Fix $k\in \N \cup \{0\}$. Using \eqref{eq:g-affine}, the directional derivative of $\varphi$ at $y_k$ in the direction $-d_k $ can be computed as
\begin{equation}\label{eq:varphi_direct}
\begin{aligned}
\varphi'(y_k;-d_k) 
&= \lim_{\beta \to0^+}\dfrac{\varphi(y_k - \beta d_k)-\varphi(y_k)}{\beta} \\
&=\lim_{\beta \to0^+}\dfrac{f(y_k - \beta d_k)-f(y_k)}{\beta} + \lim_{\beta \to0^+}\dfrac{g(y_k - \beta d_k)-g(y_k)}{\beta}\\
&= -\la \nabla f(y_k),d_k\ra -\la z_k,d_k\ra= -\la z_k + \nabla f(y_k),d_k\ra.
\end{aligned}
\end{equation}
Because $0 = y_k - y_k \in y_k - \partial g^* (z_k)$, we have $0 \in \mathcal{D}_k$. Consequently, the first-order optimality condition in \eqref{eq:first-optim} yields
\[
\la \nabla^2 f(y_k)d_k-(z_k + \nabla f(y_k)), -d_k\ra\geq 0.
\]
Combining the above estimate with \eqref{eq:Noz_modESN}, we obtain
\begin{equation}\label{eq:optim_var}
\la z_k + \nabla f(y_k), d_k\ra\geq \la \nabla^2 f(y_k) d_k,d_k\ra \geq \kappa\|d_k\|^2.
\end{equation}
Hence, combining \eqref{eq:varphi_direct} and \eqref{eq:optim_var} gives 
\begin{equation*}
    \varphi' (y_k;-d_k) \le -\kappa \|d_k\|^2,
\end{equation*}
which proves that $-d_k$ is a descent direction of $\varphi$ at $y_k$ provided that $d_k\neq 0$.  
\end{proof}


\section{A globally convergent inexact Newton method with effective sets}\label{sec:global}

In this section, we develop a globalized inexact effective-set Newton method for solving problem~\eqref{p:CP} and provide its global convergence analysis. The method extends the locally convergent schemes of the previous section by allowing the Newton subproblem to be solved approximately while preserving the descent property and fast local convergence of the exact scheme. Unless otherwise stated, we impose in this section the following standing assumptions on the data of problem \eqref{p:CP}.

\begin{Assumption}\label{as:assumption} The following holds for problem \eqref{p:CP}:

\begin{enumerate}[\bf(i)]
\item $f:\XX\to \R$ is a twice continuously differentiable convex function whose gradient is Lipschitz continuous with Lipschitz modulus $\ell>0$.

\item $g:\XX\to \oR$ is a proper, lower semicontinuous and convex function that is continuous on its domain.

\item Given an initial point $x_0$, the level set $\mathcal L_0:=\{x\in X\mid \varphi(x)\le \varphi(x_0)\}$ is bounded and $\inf \varphi >-\infty$.
\end{enumerate}
\end{Assumption}

Given a current iterate $x_k$, let the approximation pair $(y_k,z_k)\in\operatorname{gph}\partial g$
be generated by the proximal-gradient step. Fix $\bar{\delta}>0$. At every iteration $k$ for which $y_k\neq x_k$, define
\[
\delta_k
:=
\min\left\{
\bar\delta;\,\frac{\|y_k-x_k\|}{\alpha}
\right\},
\quad \text{and} \quad
H_k:=\nabla^2f(y_k)+\delta_k \mathbb{I}_n,
\]
and consider the \emph{regularized effective-set subproblem}
\begin{equation}
    \min_{d\in \mathcal{D}_k}
    \quad
        \frac12\langle H_kd,d\rangle
        -\langle z_k+\nabla f(y_k),d\rangle,
    \quad \text{where }
    \mathcal{D}_k:=y_k-\partial g^*(z_k).
    \label{eq:regularized_subproblem}
\end{equation}
Since $\delta_k>0$, the matrix $H_k$ is
positive definite. Therefore, the objective function
\begin{equation*}
    \dfrac{1}{2}\la H_k d,d\ra - \la z_k + \nabla f(y_k),d\ra + \delta_{\mathcal{D}_k}(d)
\end{equation*}
in
\eqref{eq:regularized_subproblem} is strongly convex, and according to \cite[Theorem~5.25]{Beck-1storder}, the subproblem therein
admits a unique optimal solution $d_k^*$.

For each pair $(d,\xi)\in\operatorname{gph}N_{\mathcal{D}_k}$, denote the
\emph{KKT residual} of the regularized subproblem \eqref{eq:regularized_subproblem} by
\begin{equation}
    \mathcal{R}_k(d,\xi)
    :=
    H_kd-\bigl(z_k+\nabla f(y_k)\bigr)+\xi.
    \label{eq:KKT_residual}
\end{equation}
Observe that $\mathcal{R}_k(d,\xi)=0$ for some
$\xi\in N_{\mathcal{D}_k}(d)$ if and only if $d$ is the exact solution of
\eqref{eq:regularized_subproblem}. We thus use the nonnegative
\emph{residual number} $\|\mathcal{R}_k(d,\xi)\|$ as an
\emph{optimality measure} of an approximate solution $d$ to
\eqref{eq:regularized_subproblem}. Following the classical inexact Newton principle in~\cite{DES82}, we select
$d_k\in \mathcal{D}_k$ as an \emph{approximate solution} of
\eqref{eq:regularized_subproblem}, together with
$\xi_k\in N_{\mathcal{D}_k}(d_k)$, by imposing the following two estimates as
\emph{inexact conditions}:
\begin{equation}\label{eq:intro_inexact_cond}
\|\mathcal{R}_k(d_k,\xi_k)\|
\leq
\eta_k\|z_k+\nabla f(y_k)\|
\quad \text{and} \quad
\langle z_k + \nabla f(y_k),d_k\rangle
\geq
(1-\tau)\langle H_kd_k,d_k\rangle,
\end{equation}
where the \emph{forcing term} is defined by $\eta_k
:=\min\left\{\eta;\,c\|z_k+\nabla f(y_k)\|
\right\}$ for some constants $\eta\in(0,1)$ and $c>0$, and where $\tau\in \big(0,\frac{1-2\sigma}{2(1-\sigma)}\big)$, $\sigma\in (0,1/2)$. The first condition in \eqref{eq:intro_inexact_cond} is motivated by the classical inexact Newton framework of~\cite{DES82}, whereas the role of the second condition is discussed in Remark~\ref{rem:sufficient-decrease} below.

Since
\eqref{eq:regularized_subproblem} admits a unique exact
solution $d_k^*$, its first-order optimality condition ensures the
existence of some $\xi_k^*\in N_{\mathcal{D}_k}(d_k^*)$ such that
$\mathcal{R}_k(d_k^*,\xi_k^*)=0$. Hence, the exact solution pair
$(d_k^*,\xi_k^*)$ always fulfills both inexact conditions, and therefore an approximate solution pair $(d,\xi)$ satisfying \eqref{eq:intro_inexact_cond} always exists. Moreover, observe that both inexact conditions in \eqref{eq:intro_inexact_cond} do not involve the exact subproblem solution, they can be readily verified using quantities generated by the inner solver. In Proposition~\ref{prop:well-defined}(i), we further show that any inner solver generating pairs $(d_k^j,\xi_k^j)\in\operatorname{gph}N_{\mathcal{D}_k}$ whose KKT residuals \eqref{eq:KKT_residual} converge to zero produces an approximate solution satisfying both inexact conditions after finitely many inner iterations.

\begin{Remark}[Relations to inexact sufficient-decrease conditions]\label{rem:sufficient-decrease}
{\rm Although our algorithmic framework differs from the Successive Quadratic
Approximation method in~\cite{byrd2016inexact} and the proximal Newton-type
method in~\cite{mordukhovich2023globally}, the second inexact condition in
\eqref{eq:intro_inexact_cond} has a direct connection with the decrease
conditions used in these works. More precisely, when expressed in terms of
the local models $q_k$ and $\ell_k$ in \cite{byrd2016inexact}, our condition is equivalent to the
analogue of the {\em strengthened sufficient-decrease condition}
in~\cite[Equation~(4.8)]{byrd2016inexact}, with the additional requirement
that the trial point $\hat{x}$ belongs to the face $\partial g^*(z_k)$. Moreover, for
any nonzero step, it implies the {\em strict} model-decrease condition used in
\cite[Equation~(14)]{mordukhovich2023globally}. We make these relationships
precise below.
\begin{enumerate}[\bf (i)]
\item {\em In our setting, the second inexact condition in \eqref{eq:intro_inexact_cond} is equivalent to an analogue of the strengthened sufficient-decrease condition in {\rm \cite[Equation~(4.8)]{byrd2016inexact}}, with the trial point $\hat{x}$ additionally required to lie in $\partial g^*(z_k)$}. Indeed, let us define the analogous models of $q_k$ and $\ell_k$ in \cite[Equations~(2.1)~and~(2.8)]{byrd2016inexact}, using $y_k$ as the base point:
\[
\begin{aligned}
q_k(x)
&:=
f(y_k)
+\langle \nabla f(y_k),x-y_k\rangle
+\frac12\langle H_k(x-y_k),x-y_k\rangle
+g(x),\\
\ell_k(x)
&:=
f(y_k)
+\langle \nabla f(y_k),x-y_k\rangle
+g(x).
\end{aligned}
\]
 
The analogue of the second condition in \cite[Equation~(4.8)]{byrd2016inexact} in our setting is
\begin{equation}\label{eq:Byrd-4.8}
\exists\, \hat{x}\in \XX:\,\,\,q_k(\hat x)-q_k(y_k) \leq \zeta(\ell_k(\hat x) - \ell_k(y_k)) \quad \text{for some}\,\,\, \zeta \in (\sigma,1/2), \text{ where } \sigma\in (0,1/2).
\end{equation} 

Next, we recall our second inexact condition in \eqref{eq:intro_inexact_cond} as 
\begin{equation}\label{eq:second-inexact-condition}
    \exists\, d \in \mathcal{D}_k:\,\,\,\langle z_k + \nabla f(y_k),d\rangle\geq (1-\tau)\langle H_kd,d\rangle \quad \text{for some}\,\,\, \tau \in \left(0,\dfrac{1-2\sigma}{2(1-\sigma)}\right),\, \sigma \in (0,1/2).
\end{equation}
Since $\mathcal{D}_k = y_k - \partial g^* (z_k)$, condition \eqref{eq:second-inexact-condition} is equivalent to
\begin{equation}\label{eq:second-inex-2}
\begin{aligned}
    \exists\, \hat{x}\in \partial g^* (z_k):\,\,\, &\langle z_k + \nabla f(y_k),y_k - \hat{x}\rangle\geq (1-\tau)\langle H_k(y_k - \hat{x}),y_k - \hat{x}\rangle \\
    &\text{for some}\,\,\tau \in \left(0,\dfrac{1-2\sigma}{2(1-\sigma)}\right),\, \sigma \in (0,1/2).
\end{aligned}
\end{equation}
Let us show that \eqref{eq:second-inex-2} amounts to having 
\begin{equation}\label{eq:second-inex-2-Byrd}
    \exists\, \hat{x}\in \partial g^*(z_k):\,\,\,q_k(\hat x)-q_k(y_k) \leq \zeta(\ell_k(\hat x) - \ell_k(y_k)) \quad \text{for some}\,\,\, \zeta \in (\sigma,1/2), \text{ where } \sigma\in (0,1/2),
\end{equation}
which is the analogue of \eqref{eq:Byrd-4.8} restricted to trial points
$\hat{x}\in\partial g^*(z_k)$. Indeed, observe for each $\hat{x}\in \partial g^* (z_k)$ that
\begin{equation*}
\begin{aligned}
    q_k (\hat{x}) - q_k (y_k) &= \la \nabla f(y_k), \hat{x}-y_k\ra + \dfrac{1}{2} \la H_k (\hat{x}-y_k),\hat{x}-y_k\ra + g(\hat{x}) - g(y_k)\\
    &= \la \nabla f(y_k), \hat{x}-y_k\ra + \dfrac{1}{2} \la H_k (\hat{x}-y_k),\hat{x}-y_k\ra + \big(\la z_k,\hat{x}\ra - g^* (z_k)\big) -\big(\la z_k, y_k\ra - g^* (z_k) \big)\\
    &= \la  z_k +\nabla f(y_k), \hat{x}-y_k\ra + \dfrac{1}{2} \la H_k (\hat{x}-y_k),\hat{x}-y_k\ra, 
\end{aligned}
\end{equation*}
and similarly,
\begin{equation}\label{eq:ellk-}
    \ell_k (\hat{x}) - \ell_k (y_k) = \la  z_k +\nabla f(y_k), \hat{x}-y_k\ra.
\end{equation}
As a consequence,
\begin{equation}\label{eq:reform-49}
\begin{aligned}
    & \exists\, \zeta \in (\sigma,1/2),\, \sigma\in (0,1/2):\,\, q_k (\hat{x}) - q_k (y_k) \le \zeta (\ell_k (\hat{x}) - \ell_k (y_k)) \\
    \Leftrightarrow\,\, &\exists\, \zeta \in (\sigma,1/2),\, \sigma\in (0,1/2):\,\, \la  z_k +\nabla f(y_k), \hat{x}-y_k\ra + \dfrac{1}{2} \la H_k (\hat{x}-y_k),\hat{x}-y_k\ra \le \zeta \big(\la  z_k +\nabla f(y_k), \hat{x}-y_k\ra\big)\\
    \Leftrightarrow\,\, &\exists\, \zeta \in (\sigma,1/2),\, \sigma\in (0,1/2):\,\, \dfrac{1}{2(1-\zeta)}\la H_k (y_k-\hat{x}),y_k-\hat{x}\ra \le \la z_k + \nabla f(y_k),y_k-\hat{x}\ra.
\end{aligned}
\end{equation}
Letting $\tau:= 1-\dfrac{1}{2(1-\zeta)}$ where $\zeta <1/2$, we have the relationships:
\begin{equation}\label{eq:tau-zeta}
\begin{cases}
    1-\tau = \dfrac{1}{2(1-\zeta)},\\
    \sigma < \zeta < \dfrac{1}{2} \Longleftrightarrow 0 < \tau < \dfrac{1-2\sigma}{2(1-\sigma)} \quad \text{for } \sigma \in (0,1/2).
\end{cases}
\end{equation}
Combining \eqref{eq:reform-49}--\eqref{eq:tau-zeta} yields the equivalence between \eqref{eq:second-inex-2} and \eqref{eq:second-inex-2-Byrd}. 
\endproof

It is worth emphasizing that the additional restriction $\hat{x}\in \partial g^* (z_k)$ is intrinsic to our effective-set framework rather than being imposed merely to require a stronger model decrease than that in \eqref{eq:Byrd-4.8}. This restriction is essential for the affine identity \eqref{eq:g-affine}, which underlies our $f$-based line search; see Remark~\ref{re:LS} for further details.

The inner-product formulation \eqref{eq:second-inexact-condition} is particularly convenient in our analysis, since it directly yields the quantitative descent estimate used in Remark~\ref{re:descent} and in the unit-step acceptance result in Theorem~\ref{thm:full_step_acceptance_modESN}.

\item {\em For any nonzero step $d\in\mathcal{D}_k$ satisfying
\eqref{eq:second-inexact-condition}, the corresponding point $\hat{x}:= y_k - d$ also satisfies the inexact decrease condition in {\rm \cite[Equation~(14)]{mordukhovich2023globally}}}. Indeed, with $\widehat{x}=y_k-d\in\partial g^*(z_k)$, part~{\bf (i)}
gives some $\zeta\in(\sigma,1/2)$, where $\sigma \in (0,1/2)$, such that the inequality in \eqref{eq:second-inex-2-Byrd} holds. Since $H_k$ is positive definite and $d\neq0$ fulfills condition
\eqref{eq:second-inexact-condition}, it follows that
\[
\langle z_k+\nabla f(y_k),d\rangle
\geq(1-\tau)\langle H_kd,d\rangle>0.
\]
Therefore, by \eqref{eq:ellk-} we have
\[
\ell_k(\widehat{x})-\ell_k(y_k)
= \la z_k + \nabla f(y_k),\hat{x}-y_k\ra =
-\langle z_k+\nabla f(y_k),d\rangle<0,
\]
and hence \eqref{eq:second-inex-2-Byrd} yields $q_k(\widehat{x})-q_k(y_k)<0$. Thus 
\begin{equation}\label{eq:decrease}
q_k(\widehat{x})<q_k(y_k),
\end{equation}
which is the analogue of the inexact condition in
\cite[Equation~(14)]{mordukhovich2023globally}.
\endproof
\end{enumerate}

It is worth mentioning that condition \eqref{eq:decrease} also appeared in \cite[Equation (2.6)]{byrd2016inexact}, in which the authors showed that this condition can be used to prove the global convergence of their algorithm. However, to establish the unit-step acceptance result, see \cite[Theorem 4.6]{byrd2016inexact}, they instead use the strengthened condition \eqref{eq:Byrd-4.8}. An analogous role is played in our analysis by the condition \eqref{eq:second-inexact-condition}: through its equivalent face-restricted formulation established in part~{\bf (i)}, it provides the quantitative estimate used to prove eventual acceptance of the unit step by our Armijo line search in Theorem~\ref{thm:full_step_acceptance_modESN}.

By contrast, the unit-step acceptance result in \cite{mordukhovich2023globally} relies on a different mechanism. Their algorithm accepts a unit step through a separate test based on the proximal-gradient residual at the approximate subproblem solution and bypasses the Armijo line search whenever that test succeeds, see \cite[Algorithm 1]{mordukhovich2023globally} and the discussion therein. Thus, their analysis does not establish that \eqref{eq:decrease} alone guarantees eventual acceptance of the unit step by their Armijo rule.
}

\end{Remark}

We are now ready to present the globalized effective-set Newton method
in {\bf Algorithm~\ref{alg:Mod_regularized_ESN}}.

\begin{algorithm}[H]
\caption{Globalized Effective-set Newton method}
\label{alg:Mod_regularized_ESN}
\textbf{Input:} $x_0\in \dom g$, $\ell>0$, $\alpha\in(0,2/\ell)$, 
$\theta\in(0,1)$, $\sigma\in(0,1/2)$, $\bar\delta>0$, 
$\eta\in(0,1),\tau\in\left(0,\frac{1-2\sigma}{2(1-\sigma)}\right)$, and $c>0$.\\
\textbf{Repeat}
\begin{algorithmic}[1]

    \STATE Compute
    $y_k:=\prox_{\alpha g}\big(x_k-\alpha\nabla f(x_k)\big)$ and
    $\delta_k:=\min\left\{\bar\delta;\,\dfrac{\|y_k-x_k\|}{\al}\right\}$. If $y_k=x_k$, then stop; otherwise go to the next step.

    \STATE Compute
    \[
    z_k:=\frac{x_k-y_k}{\alpha}-\nabla f(x_k),\qquad
    \mathcal D_k:=y_k-\partial g^*(z_k),\qquad
    H_k:=\nabla^2 f(y_k)+\delta_k \mathbb{I}_n.
    \]

    \STATE If $z_k+\nabla f(y_k)\in N_{\mathcal D_k}(0)$, then set $d_k=0$;
    otherwise compute $d_k\in\mathcal D_k$ and
    $\xi_k\in N_{\mathcal D_k}(d_k)$ such that
    \begin{equation}\label{eq:inexact_subproblem_condition}
    \begin{cases}
    \big\|H_kd_k-(z_k+\nabla f(y_k))+\xi_k\big\|
    \le
    \min\big\{
        \eta\|z_k+\nabla f(y_k)\|;
        \, c\|z_k+\nabla f(y_k)\|^2
    \big\} \\ 
    \langle z_k + \nabla f(y_k),d_k\rangle
    \geq
    (1-\tau)\langle H_kd_k,d_k\rangle.
    \end{cases}
    \end{equation}

    \STATE Set $\beta_k=1$ and replace $\beta_k:=\theta\beta_k$ until
    \begin{equation}\label{eq:ls_f}
        f(y_k-\beta_kd_k)
        \le
        f(y_k)
        +\beta_k\langle z_k,d_k\rangle
        -\sigma\beta_k
        \langle z_k+\nabla f(y_k),d_k\rangle .
    \end{equation}

    \STATE Update $x_{k+1}:=y_k-\beta_kd_k$.

\end{algorithmic}
\end{algorithm}

In view of the relation in \eqref{eq:res}, the case $y_k = \prox_{\alpha g}\big(x_k-\alpha\nabla f(x_k)\big) = x_k$ is equivalent to
\begin{equation*}
    \dfrac{x_k - \alpha \nabla f(x_k) - x_k}{\alpha} \in \partial g (x_k),
\end{equation*}
i.e., $-\nabla f(x_k)\in \partial g(x_k)$. Hence, $x_k$ is a global minimizer of $\varphi=f+g$. We call such an iteration $k$ a \emph{stationary iteration}.

In the proposed {\bf Algorithm~\ref{alg:Mod_regularized_ESN}}, Step~3 computes an approximate solution of \eqref{eq:regularized_subproblem} satisfying the prescribed inexactness conditions. Before invoking the subproblem iterative solver, we first check whether
$z_k+\nabla f(y_k)\in N_{\mathcal D_k}(0).$ This test determines whether the zero direction is already the exact solution of the regularized subproblem. Indeed, since \(0\in\mathcal D_k\), the above inclusion is equivalent to
\[
0\in H_k(0)-\bigl(z_k+\nabla f(y_k)\bigr)+N_{\mathcal D_k}(0),
\]
which is precisely the first-order optimality condition for \eqref{eq:regularized_subproblem} at \(d=0\). Since \(H_k\) is positive definite, the regularized subproblem is strongly convex, and hence \(d=0\) is its unique optimal solution. Therefore, in this case no inner iteration is needed and we simply set \(d_k=0\). Otherwise, the exact solution \(d_k^*\) of \eqref{eq:regularized_subproblem} is necessarily nonzero, and an iterative solver is employed to compute an approximate direction satisfying the inexact conditions. This separation is also useful for the well-definedness of Step~3, because in the case that $z_k + \nabla f(y_k) \notin N_{\mathcal{D}_k}(0)$, we have  \(d_k^*\neq0\), which provides the strict inequality needed to show that the proposed inexactness conditions are satisfied after finitely many inner iterations.

Moreover, we also added a backtracking line search from $y_k$ along the direction
$-d_k$ in Step~4 of {\bf Algorithm~\ref{alg:Mod_regularized_ESN}}, which serves as a globalization strategy. A distinctive feature of this procedure is that the line-search
condition involves only the smooth component $f$. This is made possible by the
effective-set constraint $d_k\in\mathcal{D}_k$, which ensures that $g$ varies
affinely along the search segment, as shown in Remark~\ref{re:LS}. Consequently,
the contribution of $g$ to the objective change is known exactly, so each
backtracking trial requires only an evaluation of $f$, while the resulting
condition remains equivalent to an Armijo decrease condition for the original
composite objective $\varphi=f+g$. It is shown in Proposition~\ref{prop:well-defined}(ii) that such line search is well-defined.

\begin{Remark}[Armijo backtracking line-search on $f$ versus on $\varphi$]\label{re:LS}
\rm Although the backtracking line-search condition \eqref{eq:ls_f} is formulated solely in terms of the smooth component $f$, it is, in fact, equivalent to Armijo decrease conditions for the composite objective $\varphi = f+g$. Indeed, since $d_k\in \mathcal{D}_k$ and $y_k\in \partial g^* (z_k)$, recall from \eqref{eq:g-affine} that for any $\beta\in[0,1],$ 
\begin{equation*}
    g(y_k-\beta d_k)=g(y_k)-\beta\langle z_k,d_k\rangle.
\end{equation*}
Furthermore, the backtracking procedure terminates when there exists $\beta_k:= \theta^{m_k}$ $(m_k\ge 0)$ such that \eqref{eq:ls_f} is satisfied. Therefore, using the above identity and adding both sides of \eqref{eq:ls_f} with $g(y_k -\beta_k d_k)$ yields
\begin{equation}\label{eq:ls2_equiv}
\varphi(y_k-\beta_k d_k)
\le \varphi(y_k)-\sigma\beta_k\langle z_k+\nabla f(y_k),d_k\rangle.
\end{equation}
Thus, this step enforces descent of
$\varphi$ relative to $y_k$. The proposed formulations \eqref{eq:ls_f} are favorable compared to \eqref{eq:ls2_equiv} because, throughout the backtracking procedure, once the relevant inner products are computed before the loop, one only needs to evaluate
the smooth component $f$ and simple scalar multiplication on the right-hand side. 
\end{Remark}

Unlike Proposition~\ref{prop:local-descent}, where descent of the exact effective-set Newton direction
is established only within a neighborhood of the tilt-stable minimizer, the regularization $H_k = \nabla^2 f(y_k) + \delta_k \mathbb{I}_n$ and the second
inexact condition in {\bf Algorithm~\ref{alg:Mod_regularized_ESN}} guarantee descent at every non-stationary
iteration. Indeed, since $H_k \succeq \delta_k \mathbb{I}_n$ with $\delta_k>0$, this conclusion no longer relies on
tilt stability or on $x_k$ being sufficiently close to $\bar{x}$. The resulting global descent property is
summarized in the following remark.

\begin{Remark}[Globally descent direction under inexactness]\label{re:descent}
The descent property of the exact Newton direction is preserved under the inexact condition \eqref{eq:inexact_subproblem_condition}. Indeed, for each $k\in \N\cup \{0\}$, condition \eqref{eq:inexact_subproblem_condition} yields
\begin{equation}\label{eq:vkdk>0}
\la z_k+\nabla f(y_k),d_k \ra \geq (1-\tau)\la H_k d_k,d_k\ra \geq (1-\tau)\delta_k\|d_k\|^2 .
\end{equation}
Thus, recalling \eqref{eq:varphi_direct}, whenever $d_k\neq 0$ we have
\[
\varphi'(y_k;-d_k)
=-\la z_k+\nabla f(y_k),d_k\ra
\leq -(1-\tau)\delta_k\|d_k\|^2<0,
\]
so $-d_k$ remains a descent direction for $\varphi$ at $y_k$.
\end{Remark}

Before analyzing the convergence of {\bf Algorithm~\ref{alg:Mod_regularized_ESN}}, we need to verify that it is well-defined, i.e., by discussing the feasibility of the finite termination of the line-search \eqref{eq:ls_f} in Step 4. Although the feasibility of Step 3 has already been established at the beginning of this section, it remains to provide a constructive procedure for obtaining an approximate pair $(d_k,\xi_k)$ satisfying the required conditions. This issue is also addressed in the next proposition.

\begin{Proposition}\label{prop:well-defined} For each (non-stationary) iteration $k\in \N\cup\{0\}$ of {\bf Algorithm~\ref{alg:Mod_regularized_ESN}}, let $d_k^*$ be the exact solution of \eqref{eq:regularized_subproblem}. The following assertions hold:
\begin{enumerate}[\bf (i)]
\item For any pair $(\widehat{d}_k,\widehat{\xi}_k)\in \gph N_{\mathcal{D}_k}$, we have
\begin{equation}\label{eq:inexact_bound}
\|\hat d_k - d_k^*\| \leq \dfrac{\| H_k \hat d_k - (z_k+\nabla f(y_k)) + \hat\xi_k\|}{\delta_k}.
\end{equation}
Consequently, for any inner solver that generates $\{(d_k^j, \xi_k^j)\} \subset \gph N_{\mathcal{D}_k}$ such that 
\begin{equation}\label{eq:KKT_to0}
\| H_k d_k^j - (z_k+\nabla f(y_k)) + \xi_k^j\| \to 0 \quad \text{as}\quad j\to \infty,
\end{equation}
then the inexact conditions in \eqref{eq:inexact_subproblem_condition} are satisfied at $(d^j_k,\xi^j_k)$ for sufficiently large $j$.

\item The Armijo backtracking line search in Step 4 of {\bf Algorithm~\ref{alg:Mod_regularized_ESN}} terminates finitely. 
\end{enumerate}
\end{Proposition} 

\begin{proof}
Fix any iteration $k \in \N\cup\{0\}$, let us first show that {\bf (i)} holds. Indeed, since $d^*_k$ is the exact optimal solution of \eqref{eq:regularized_subproblem}, there exists $\xi_k^* \in N_{\mathcal{D}_k}(d^*_k)$ such that $H_kd^*_k - (z_k+\nabla f(y_k)) + \xi^*_k = 0$. Hence, we obtain
\begin{equation}
\begin{aligned}
\la H_k \hat d_k - (z_k+\nabla f(y_k)) + \hat\xi_k, \hat d_k - d^*_k\ra
&= \la H_k (\hat d_k-d^*_k) + \hat\xi_k  - \xi^*_k, \hat d_k - d^*_k\ra \\
&= \la H_k (\hat d_k-d^*_k),\hat d_k - d^*_k\ra + \la \hat\xi_k  - \xi^*_k, \hat d_k - d^*_k\ra \label{eq:in1}.
\end{aligned}
\end{equation}
The monotonicity of $N_{\mathcal{D}_k} = \partial \delta_{\mathcal{D}_k}$ gives $\la \hat\xi_k  - \xi^*_k, \hat d_k - d^*_k\ra \geq 0.$ Since $H_k \succeq \delta_k\mathbb{I}_n$, it follows from \eqref{eq:in1} that
\[
\| H_k \hat d_k - (z_k+\nabla f(y_k)) + \hat\xi_k\|\|\hat d_k - d^*_k\| \geq \la H_k \hat d_k - (z_k+\nabla f(y_k)) + \hat\xi_k, \hat d_k - d^*_k\ra \geq \delta_k \|\hat d_k-d^*_k\|^2.
\]
If $\hat d_k = d^*_k$, then \eqref{eq:inexact_bound} automatically holds true. Otherwise, since non-stationarity implies that $\delta_k >0$, dividing both sides of the above inequality by $\delta_k\|\hat d_k - d^*_k\|$ gives \eqref{eq:inexact_bound}. 

It remains to show that, for any inner solver that generates $(d_k^j, \xi_k^j) \in \gph N_{\mathcal{D}_k}$ such that \eqref{eq:KKT_to0} holds, then \eqref{eq:inexact_subproblem_condition} holds at $(d^j_k,\xi^j_k)$ for sufficiently large $j$. First, applying \eqref{eq:inexact_bound} with $\hat{d_k} = d_k^j$ and employing \eqref{eq:KKT_to0}, we have 
$$
d_k^j \to d_k^* \quad \text{as}\quad j\to \infty.
$$
Furthermore, it is worth noting that an inner solver is considered in Step 3 of {\bf Algorithm~\ref{alg:Mod_regularized_ESN}} only when $z_k + \nabla f(y_k) \notin N_{\mathcal{D}_k}(0)$. This implies that $z_k + \nabla f(y_k) \ne 0$ and thus
\[
\min\left\{
\eta\|z_k+\nabla f(y_k)\|;\,
\ c\|z_k+\nabla f(y_k)\|^2
\right\} >0.
\]
Since $\| H_k d_k^j - (z_k+\nabla f(y_k)) + \xi_k^j\| \to 0$ as $j\to\infty$, the first condition in \eqref{eq:inexact_subproblem_condition} is satisfied for sufficiently large $j$. Moreover, notice that if $d_k^* = 0$, we immediately deduce from the first-order optimality condition that $z_k + \nabla f(y_k) \in N_{\mathcal{D}_k}(0)$, which is a contradiction. This asserts that $d^*_k\ne0$. The admissible range of $\tau$ implies that $0 < 1-2\tau < 1-\tau < 1$. Hence $(1-\tau)d_k^* \in \mathcal{D}_k$, and optimality of problem \eqref{eq:regularized_subproblem} at $d_k^*$ yields
\begin{equation*}
    \dfrac{1}{2}\la H_k d_k^*,d_k^*\ra - \la z_k + \nabla f(y_k),d_k^*\ra \le \dfrac{1}{2}(1-\tau)^2\la H_k d_k^*,d_k^*\ra - (1-\tau)\la z_k + \nabla f(y_k),d_k^*\ra.
\end{equation*}
From this, it is obvious that 
\begin{equation*}
    \dfrac{1}{2}\tau (2-\tau)\la H_k d_k^*,d_k^*\ra \le \tau \la z_k + \nabla f(y_k),d_k^*\ra,
\end{equation*}
and hence
\begin{equation*}
    (1-\tau)\la H_k d_k^*,d_k^*\ra < \left(1-\dfrac{\tau}{2}\right)\la H_k d_k^*,d_k^*\ra \le \la z_k + \nabla f(y_k),d_k^*\ra.
\end{equation*}
Since the estimate is strict and we have $d_k^j \to d_k^*$ as $j\to \infty$, there exists $j\in \N$ large enough so that 
\begin{equation*}
    (1-\tau)\la H_k d^j_k,d^j_k\ra < \la z_k + \nabla f(y_k),d_k^j\ra,
\end{equation*}
and thereby the second condition in \eqref{eq:inexact_subproblem_condition} is also satisfied for sufficiently large $j$.

Next, we show that {\bf (ii)} holds. If $d_k = 0$, then the line search condition is immediately satisfied and hence the line search step terminates. Consider the case when $d_k \ne 0$, it follows from \eqref{eq:varphi_direct} and \eqref{eq:vkdk>0} that
\begin{equation*}
    \lim_{\beta\to 0^+} \left(\dfrac{\varphi (y_k - \beta d_k) - \varphi (y_k)}{\beta} + \la z_k + \nabla f(y_k),d_k\ra \right) = 0 < (1-\sigma) \la z_k + \nabla f(y_k),d_k\ra,
\end{equation*}
which thus yields
\[
\varphi (y_k - \beta d_k) - \varphi (y_k) < -\sigma \beta \la z_k + \nabla f(y_k),d_k\ra \quad \text{for all}\quad \beta>0 \text{ sufficiently small}.
\]
In particular, condition \eqref{eq:ls2_equiv} holds true for some $\beta_k:= \theta^{m_k}$, where $m_k$ is large enough. Thus by the equivalence between \eqref{eq:ls2_equiv} and \eqref{eq:ls_f}, the Armijo backtracking line search must terminate after finitely many steps.
The proof is complete.
\end{proof}

In the next result, we present several convergence properties of {\bf Algorithm~\ref{alg:Mod_regularized_ESN}}.

\begin{Proposition}\label{prop:conv_modESN} 
Let $\{x_k\}$ be the sequence generated by {\bf Algorithm~\ref{alg:Mod_regularized_ESN}} and suppose that each $x_k$ is not a stationary point of $\varphi$. Then, the following holds:
\begin{enumerate}[\bf(i)]

\item The objective value sequence $\{ \varphi(x_k)\}$ satisfies the following sufficient decrease property:
\begin{equation}\label{eq:suff_descent}
\varphi(x_{k+1})\le \varphi(x_k)-C\|y_k-x_k\|^2 -\sigma\beta_k\langle z_k+\nabla f(y_k),d_k\rangle, 
\end{equation}
for all $C\in\left(0,\frac{1}{\alpha} - \frac{\ell}{2}\right)$. Consequently, it holds that $$\sum_{k=0}^{\infty}\|y_k-x_k\|^2<\infty \quad \text{and}\quad y_k-x_k \to 0 \,\,\text{ as }\,\,k\to \infty.
$$
\item Every accumulation point of $\{x_k\}$ is a minimizer of $\varphi$. As a consequence, if $\ox$ is the unique minimizer of $\varphi$, then $x_k \to \ox$ as $k\to \infty$.

\end{enumerate}
\end{Proposition}

\begin{proof}
Fix a non-stationary iteration $k\in\N\cup \{0\}$. Let us prove {\bf (i)}. Since $g$ is convex and $z_k\in\partial g(y_k)$, we derive the estimate
\[
g(y_k)-g(x_k) \le \langle z_k,y_k - x_k\ra = \Big\langle \dfrac{x_k - y_k}{\alpha} - \nabla f(x_k), y_k - x_k\Big\rangle =-\langle\nabla f(x_k),y_k-x_k\ra
- \frac1\alpha\|y_k-x_k\|^2.
\]
On the other hand, since $\nabla f$ is $\ell$-Lipschitz continuous, we have
\[
f(y_k) \le f(x_k) +\langle\nabla f(x_k),y_k-x_k\ra+ \frac {\ell}2\|y_k-x_k\|^2.
\]
Adding the last two inequalities yields 
\[
\varphi(y_k) \le \varphi(x_k)-C\|y_k-x_k\|^2 \quad \text{for any } C\in\left(0,\dfrac{1}{\alpha} - \dfrac{\ell}{2}\right).
\]
Using the above estimate, together with Proposition~\ref{prop:well-defined}(ii) and Remark~\ref{re:LS}, one deduces that
\begin{align*}
\varphi(x_{k+1})
=\varphi(y_k-\beta_k d_k)
&\le \varphi(y_k)-\sigma\beta_k\langle z_k+\nabla f(y_k),d_k\rangle
\leq
\varphi(x_k)-C\|y_k-x_k\|^2-\sigma\beta_k\langle z_k+\nabla f(y_k),d_k\rangle,
\end{align*}
in which $\la z_k + \nabla f(y_k),d_k\ra \ge 0$ by Remark~\ref{re:descent}. Thus, we conclude that the sufficient decrease estimate \eqref{eq:suff_descent} is satisfied and $\{\varphi(x_k)\}$ is a decreasing sequence.  Moreover, since $\inf\varphi>-\infty$, summing \eqref{eq:suff_descent} over $k=0,1,\ldots,N$ $(N\in \N)$ gives
\[
C\sum_{k=0}^{N}\|y_k -x_k\|^2
\le
\varphi(x_0)-\varphi(x_{N+1})
\le
\varphi(x_0)-\inf\varphi.
\]
Letting $N\to\infty$, we obtain $\sum_{k=0}^{\infty}\|y_k-x_k\|^2<\infty$, which also implies that $y_k-x_k \to 0$.

To show {\bf (ii)}, note that since $\varphi(x_k)\le \varphi(x_0)$ for all $k\in \N$ and $\mathcal L_0$ is bounded, the sequence $\{x_k\}$ is also bounded. Take $\{x_{k_j}\}$ to be a convergent subsequence whose limit is $x^*$. Since $y_k - x_k\to0$, we have
\[
y_{k_j} \to x^*  \quad \text{and} \quad  z_{k_j}=\dfrac{x_{k_j}-y_{k_j}}{\alpha}-\nabla f(x_{k_j})\to -\nabla f(x^*).
\]
Since $z_{k_j}\in\partial g(y_{k_j})$ and $\operatorname{gph}\partial g$ is a closed set, it follows from the above that 
\begin{equation*}
    -\nabla f(x^*) \in \partial g(x^*),
\end{equation*}
i.e., $x^*\in \arg\min \varphi$. Hence, every accumulation point of $\{x_k\}$ is a minimizer to $\varphi$. The proof is complete.
\end{proof}

\begin{Remark}[Role of the globalization mechanism]\rm 
The global convergence of {\bf Algorithm~\ref{alg:Mod_regularized_ESN}}, as established in Proposition~\ref{prop:conv_modESN}(ii), relies essentially on the
combination of the proximal-gradient step (or, ISTA) and the damped Newton correction. Indeed, the global
Lipschitz continuity of $\nabla f$, together with $\alpha\in(0,2/\ell)$, yields the descent estimate
\[
\varphi(y_k)\leq \varphi(x_k)-C\|y_k-x_k\|^2,
\qquad
C\in\left(0,\frac1\alpha-\frac{\ell}{2}\right),
\]
for every iterate $x_k$, without requiring $x_k$ to be close to the solution. The Armijo line search condition \eqref{eq:ls_f}
then guarantees that the Newton correction does not destroy this decrease, leading to the sufficient decrease property \eqref{eq:suff_descent}:
\begin{equation*}
\varphi(x_{k+1})\le \varphi(x_k)-C\|y_k-x_k\|^2 -\sigma\beta_k\langle z_k+\nabla f(y_k),d_k\rangle. 
\end{equation*}
Consequently, we have $\|y_k-x_k\|\to0$, which is the key estimate used in
Proposition~\ref{prop:conv_modESN}(ii) to show that every accumulation point is a minimizer.

This is the main distinction from the local {\bf Algorithm~\ref{algo:modified-Dk}}, whose convergence
requires the initial point to lie sufficiently close to the tilt-stable minimizer so that the local
Newton estimates remain valid. {\bf Algorithm~\ref{alg:Mod_regularized_ESN}} removes this localization
requirement through the above globalization mechanism; once the iterates are sufficiently close to
the tilt-stable minimizer, Theorem~\ref{thm:full_step_acceptance_modESN} below shows that $\beta_k=1$
eventually, and the local fast Newton behavior is recovered.
\end{Remark}

Having clarified the role of our globalization mechanism, it is useful to compare it with the
widely used globalization strategy based on the Forward--Backward Envelope (FBE).

\begin{Remark}[Comparisons with FBE-based globalization]\label{rem:FBE_comparison}
{\rm
In the development of globally convergent nonsmooth Newton methods for composite setting, one of the most popular globalization strategies is based on the {\em Forward–Backward Envelope} (FBE) of $\varphi=f+g$, which is defined, for $\gamma>0$, as
\[
    \varphi_\gamma(x)
    :=
    f(x)-\frac{\gamma}{2}\|\nabla f(x)\|^2
    +e_{\gamma g}\big(x-\gamma\nabla f(x)\big),
\]
via the {\em Moreau envelope}
\begin{equation*}
    e_{\gamma g}(x):= \inf_{y\in \XX}\left\{g(y) + \frac{1}{2\gamma}\|y-x\|^2 \right\},\qquad x\in \XX.
\end{equation*}
The FBE was introduced in \cite{P13} under the name composite Moreau envelope and was subsequently developed as a smooth reformulation for designing globally convergent Newton-type methods. In particular, it has been employed in truncated Newton methods \cite{P14}, forward--backward quasi-Newton methods \cite{S17, T18}. More recently, FBE-based globalization has also been incorporated into coderivative-based generalized Newton methods \cite{KhanhMordukhovichPhat24, KhanhMordukhovichPhat25} and SCD semismooth\(^*\) Newton methods \cite{GfrererCOA25,Gfrerer25}. In these approaches, the FBE serves as a merit function, to globalize locally fast Newton-type directions. Our globalization differs fundamentally from these FBE-based strategies in the following aspects.
\begin{enumerate}[(i)]
    \item When $g\equiv0$, the Moreau envelope term vanishes, i.e. $e_{\gamma g}(x-\gamma \nabla f(x))=0$. Moreover, the FBE and its gradient are
    \[
        \varphi_\gamma(x)
        =f(x)-\frac{\gamma}{2}\|\nabla f(x)\|^2 \quad \text{and} \quad \nabla \varphi_\gamma(x) = \left(\mathbb{I} - \gamma \nabla^2 f(x)\right) \nabla f(x),
    \]
    respectively. Consider now a standard Armijo line search applied to the FBE from a point $y$ along a direction $-d$.  Its acceptance criterion is
    \[
    \varphi_\gamma(y - \beta d) \leq \varphi_\gamma(y) - \sigma\beta \left\la  \left(\mathbb{I} - \gamma \nabla^2 f(y)\right) \nabla f(y), d\right\ra,
    \]
    which is equivalent to 
    \begin{align}
    f(y-\beta d)
    -\frac{\gamma}{2}\bigl\|\nabla f(y-\beta d)\bigr\|^2
    \nonumber\leq
    f(y)
    -\frac{\gamma}{2}\|\nabla f(y)\|^2
    -\sigma\beta\langle \nabla f(y),d\rangle
    +\sigma\beta\gamma
    \left\langle
    \nabla^2 f(y)\nabla f(y),d
    \right\rangle .
    \end{align}
    
    Hence, even when $g\equiv0$, an FBE-based globalization does not recover the classical Armijo line-search criterion used in damped Newton's methods for
    smooth optimization, since the gradient-norm and Hessian-dependent
    terms remain present.
    
    Our globalization, in contrast, reduces exactly to the classical Armijo line-search condition for the smooth objective $f$. Indeed, in case $g \equiv 0$, for any fixed iteration $k$, observe that 
    \[
    y_k = x_k - \alpha \nabla f(x_k) \quad \text{and} \quad z_k =0.
    \]
    Using the equivalent formulation \eqref{eq:ls2_equiv} of \eqref{eq:ls_f}, as established in Remark~\ref{re:LS}, our backtracking procedure
    terminates once
    \[
    f(y_k - \beta_kd_k) \leq f(y_k) -\sigma \beta_k\la\nabla f(y_k),d_k\ra,
    \]
    which has exactly the Armijo acceptance form used in the classical globalization of Newton's method, with $y_k$ as the base point and $-d_k$ as the search direction. Thus, in the smooth case, our globalization retains the classical Armijo line-search structure, whereas an FBE-based globalization generally does not.
    
    \item In the FBE-based globalization
    \cite[Algorithm~4]{KhanhMordukhovichPhat25}, the parameter $\lambda$
    determines both the proximal-gradient point and the envelope
    $\varphi_\lambda$, requiring $\lambda\in(0,1/\ell)$ in the convex setting of $g$ considered here. In contrast, the constant
    $\alpha$ in our method is used only to generate $(y_k,z_k)$, while
    globalization is imposed directly through the $f$-based line search.
    This permits the larger range $\alpha\in(0,2/\ell)$.

    \item Each trial evaluation of the FBE requires evaluations of
    $\nabla f$, $f$, $g$, and $\prox_{\lambda g}$, whereas our backtracking
    only evaluates $f$ at the trial point, apart from scalar operations.
    Thus, our line search has a lower per-trial oracle cost, while
    Remark~\ref{re:LS} shows that it still guarantees descent of the
    composite objective.
    \item A limitation of our current framework is its stronger structural requirement on the nonsmooth term. In particular, our local quadratic convergence rate analysis in Theorem~\ref{theo:quad-poly2} is developed for polyhedral convex regularizers $g$, whereas FBE-based globalization has been developed for substantially broader classes of nonsmooth functions, including nonconvex, prox-bounded, prox-regular regularizers; see, e.g., \cite{GfrererCOA25,Gfrerer25,KhanhMordukhovichPhat25}. Thus, the lower per-trial cost of our $f$-based line search comes at the expense of a more specialized framework in the present analysis.
\end{enumerate}
}
\end{Remark}

The following result establishes that the unit step-size must be accepted when the iterates are sufficiently close to the tilt-stable minimizer $\ox$. Furthermore, this full step is always accepted when the smooth component $f$ is quadratic, without requiring the iterates to lie in a small neighborhood of $\ox$. These results rely on a uniform positive-curvature bound for $f$ at $y_k$ along directions in the effective set $\mathcal{D}_k = y_k- \partial g^* (z_k)$, i.e.,
\begin{equation*}
    \exists\, \kappa >0:\quad \la \nabla^2 f(y_k)w,w\ra \ge \kappa\|w\|^2 \quad \text{for all}\quad w\in \mathcal{D}_k.
\end{equation*}
As see later in the proof below, this relation follows from the tilt stability at $\ox$ and the inclusion $\mathcal{D}_k \subset \mathrm{Ker}\, \partial^2 g(y_k|\,z_k)$ established in Proposition~\ref{prop:D_k}(ii).

\begin{Theorem}[Acceptance of unit step size]
\label{thm:full_step_acceptance_modESN}
Let $\{x_k\}$ be the sequence generated by {\bf Algorithm~\ref{alg:Mod_regularized_ESN}} and $\bar{x}$ is the tilt-stable minimizer of the function $\varphi$. Then, the following holds:
\begin{enumerate}[\bf(i)]

\item There exists $k_0\equiv k_0(\ox)$ such that, at every non-stationary iteration $k\ge k_0$, the Armijo backtracking line search in Step 4 accepts $\beta_k =1$, i.e.,
\[
f(y_k- d_k)
\leq
f(y_k)
+\langle z_k,d_k\rangle
-\sigma\langle z_k+\nabla f(y_k),d_k\rangle.
\]
\item  In case $f(x) = \dfrac{1}{2}\la Ax,x\ra + \la b,x\ra + c$, where $A \in \R^{n \times n}$ is a symmetric positive semidefinite matrix, $b\in \R^n$ and $c\in \R$, then the backtracking stepsize $\beta_k = 1$ is accepted at every non-stationary iteration. 
\end{enumerate}
\end{Theorem}

\begin{proof} 
First, let us show that {\bf (i)} holds. Indeed, by the characterization of tilt stability in \eqref{eq:tilt} and the local Lipschitz continuity of $\nabla^2 f$ around $\ox$, we find positive constants $\kappa,r>0$ such that 
\begin{equation}\label{eq:unit}
\begin{aligned}
    &\la z,w\ra \ge \kappa \|w\|^2 \quad \forall\, z\in \partial^2 \varphi (x|\, v)(w),\ \forall\,
    (x,v)\in \gph \partial \varphi\cap \mathbb{B}_{r}(\ox,0),\\
    &\|\nabla^2 f(x)-\nabla^2 f(y)\| \le \kappa \left( (1-\sigma)(1-\tau)-\dfrac{1}{2}\right), \quad \forall\, x,y\in \mathbb{B}_{r}(\ox),
\end{aligned}
\end{equation}
where $(1-\sigma)(1-\tau) > \frac{1}{2}$ by the choice of $\tau$ and $\sigma$ in {\bf Algorithm~\ref{alg:Mod_regularized_ESN}}.

Since $\ox$ is the unique minimizer to problem \eqref{p:CP}, Proposition~\ref{prop:conv_modESN}(ii) yields the convergence $(x_k,y_k) \to (\ox,\ox)$ as $k\to \infty$. Further, we have 
\begin{equation}\label{eq:vkto0}
    v_k := z_k + \nabla f(y_k) =  \dfrac{x_k - y_k}{\alpha} - \nabla f(x_k) + \nabla f(y_k) \to 0 \quad \text{as}\quad k\to \infty.
\end{equation}
In particular, the pair $(y_k,v_k)$ belongs to $\gph \partial \varphi\cap \mathbb{B}_r (\ox,0)$ for sufficiently large $k\in \N$. Analogously to \eqref{eq:Noz_modESN}, but using instead the inclusion $\mathcal{D}_k\subset \mathrm{Ker}\, \partial^2 g(y_k|\,z_k)$ from Proposition~\ref{prop:D_k}, we deduce from \eqref{eq:unit} the relation
\begin{equation}\label{eq:unit3}
\la \nabla^2 f(y_k)w,w\ra\ge \kk\|w\|^2\quad \mbox{for all}\quad w\in \mathcal{D}_k,\,\, k\in \N \text{ large}.
\end{equation}
It follows from \eqref{eq:inexact_subproblem_condition} and \eqref{eq:unit3} that
\begin{equation}
\la z_k+\nabla f(y_k),d_k \ra
\geq (1-\tau)\la H_k d_k,d_k\ra
\geq (1-\tau)\la\nabla^2 f(y_k)d_k,d_k\ra\geq (1-\tau)\kappa\|d_k\|^2 \,\, \text{for}\,\, k\text{ large}. \label{eq:estimate2}
\end{equation}
If $d_k = 0$, the unit step acceptance holds trivially. Suppose that $d_k\ne0$, applying the Cauchy-Schwarz inequality for \eqref{eq:estimate2} and then employing \eqref{eq:vkto0} give
\[
\|d_k\| \leq  \dfrac{1}{\kappa(1-\tau)}\|z_k + \nabla f(y_k)\| = \dfrac{1}{\kappa (1-\tau)}\|v_k\| \to 0 \quad \text{as}\quad k\to \infty.
\]
The Taylor expansion of $f$ at $y_k$ asserts that
\begin{equation}\label{eq:2Taylor}
f(y_k-d_k) = f(y_k) -\la\nabla f(y_k),d_k\ra + \dfrac{1}{2}\la\nabla^2f(y_k-t_kd_k)d_k,d_k\ra \quad \text{for some} \quad t_k\in(0,1).
\end{equation}
Since $\{t_k\}\subset (0,1)$ and $d_k\to 0$, we have $y_k - t_k d_k\to \ox$ as $k\to \infty$. Because $(y_k - t_k d_k,y_k) \to (\ox,\ox)$, it follows from \eqref{eq:unit} and \eqref{eq:2Taylor} that
\begin{align*}
&f(y_k-d_k) - f(y_k) +\la\nabla f(y_k),d_k\ra - \dfrac{1}{2}\la\nabla^2f(y_k)d_k,d_k\ra\\
&= \dfrac{1}{2}\left\la\left(\nabla^2f(y_k-t_kd_k)-\nabla^2f(y_k)\right)d_k,d_k\right\ra\\
&\leq \dfrac{1}{2}\| \nabla^2f(y_k-t_kd_k)-\nabla^2f(y_k)\|\|d_k\|^2 \leq \dfrac{\kappa}{2}\left((1-\sigma)(1-\tau)-\dfrac{1}{2}\right)\|d_k\|^2.
\end{align*}
Hence, 
\begin{equation}
\begin{aligned}
&f(y_k-d_k) -f(y_k) -\la z_k,d_k\ra +\sigma\la z_k+\nabla f(y_k),d_k\ra \\
&\leq -\la\nabla f(y_k),d_k\ra -\la z_k,d_k\ra + \sigma\la z_k+\nabla f(y_k),d_k\ra+\dfrac{1}{2}\la\nabla^2f(y_k
)d_k,d_k\ra +\dfrac{\kappa}{2}\left((1-\sigma)(1-\tau)-\dfrac{1}{2}\right)\|d_k\|^2 \\
&= -(1-\sigma)\la z_k+\nabla f(y_k),d_k\ra+\dfrac{1}{2}\la\nabla^2f(y_k)d_k,d_k\ra +\dfrac{\kappa}{2}\left((1-\sigma)(1-\tau)-\dfrac{1}{2}\right)\|d_k\|^2\label{eq:estimate1}.
\end{aligned}
\end{equation}
Combining \eqref{eq:estimate2}, \eqref{eq:estimate1} and $\tau\in\left(0,\frac{1-2\sigma}{2(1-\sigma)}\right)$ where $0<\sigma <1/2$, we arrive at
\begin{equation}
\begin{aligned}
&f(y_k-d_k) -f(y_k) -\la z_k,d_k\ra +\sigma\la z_k+\nabla f(y_k),d_k\ra \\
&\le -(1-\sigma)\la z_k+\nabla f(y_k),d_k\ra+\dfrac{1}{2}\la\nabla^2f(y_k)d_k,d_k\ra +\dfrac{\kappa}{2}\left((1-\sigma)(1-\tau)-\dfrac{1}{2}\right)\|d_k\|^2 \\
&\leq \left(\dfrac{1}{2} -\left(1-\sigma\right)(1-\tau)\right)\la \nabla^2 f(y_k)d_k,d_k\ra +\dfrac{\kappa}{2}\left((1-\sigma)(1-\tau)-\dfrac{1}{2}\right)\|d_k\|^2\\
&\leq \kappa\left(\dfrac{1}{2} -\left(1-\sigma\right)(1-\tau)\right)\|d_k\|^2 +\dfrac{\kappa}{2}\left((1-\sigma)(1-\tau)-\dfrac{1}{2}\right)\|d_k\|^2  \\
&\leq-\dfrac{\kappa}{2}\left((1-\sigma)(1-\tau)-\dfrac{1}{2}\right)\|d_k\|^2 \leq0\label{eq:estimate3}. 
\end{aligned}
\end{equation}
Hence, we may conclude from \eqref{eq:estimate3} that the backtracking line search accepts $\beta_k = 1$ for all sufficiently
large $k\in \N$. 

Finally, we show that the assertion {\bf (ii)} holds. Fix any non-stationary iteration $k$ and suppose that $f(x) = \dfrac{1}{2}\la Ax,x\ra + \la b,x\ra + c.$ Since $\nabla f(x)=Ax+b$ and $\nabla^2 f(y_k)=A$, one may write
\[
f(y_k-d_k)
=
f(y_k)-\langle \nabla f(y_k),d_k\rangle
+\frac12\langle A d_k,d_k\rangle.
\]
Given that $H_k = A + \delta_k  \mathbb{I}_n$ and $\tau\in\left(0,\frac{1-2\sigma}{2(1-\sigma)}\right)$, we have
\[
\begin{aligned}
&f(y_k-d_k) -f(y_k)-\langle z_k,d_k\rangle
+\sigma\langle z_k+\nabla f(y_k),d_k\rangle\\
&= -\la \nabla f(y_k),d_k\ra + \dfrac{1}{2}\la Ad_k,d_k\ra -\la z_k,d_k\ra + \sigma \la z_k + \nabla f(y_k),d_k\ra \\
& = -(1-\sigma)\langle z_k+\nabla f(y_k),d_k\rangle + \frac12\langle A d_k,d_k\rangle\\
& \leq -(1-\sigma)(1-\tau)\langle H_kd_k,d_k\rangle + \frac12\langle A d_k,d_k\rangle\\
&\leq \left(\dfrac{1}{2} -\left(1-\sigma\right)(1-\tau)\right)\langle A d_k,d_k\rangle -(1-\sigma)(1-\tau)\delta_k\|d_k\|^2 \leq 0.
\end{aligned}
\]
This is exactly the line search acceptance condition with $\beta_k=1$. Hence, no
backtracking is performed and $\beta_k=1$ is accepted. The proof is complete.
\end{proof}

\begin{Proposition}[Backtracking complexity of Armijo line search]
\label{prop:backtracking-complexity}
For every non-stationary iteration $k \in \N\cup \{0\}$, if $d_k \ne 0$, then every stepsize $\beta$ satisfying 
\begin{equation}\label{eq:step_accept}
0<\beta \le \Bar{\beta}_k:= \min\left\{1;\,  \frac{2(1-\sigma)\langle z_k+\nabla f(y_k),d_k\rangle}{{\ell}\|d_k\|^2}\right\} 
\end{equation}
is accepted by the Armijo backtracking line search. Consequently, the number of iterations for line search procedures termination $N_k$ satisfies
\begin{equation}\label{eq:complexity_accept}
N_k
\le \overline{N}_k :=
\max\left\{
0;\,
\left\lceil
\frac{\log  \Bar{\beta}_k}{\log \theta}
\right\rceil
\right\}.
\end{equation}
\end{Proposition}
\begin{proof}
Since $\nabla f$ is $\ell$-Lipschitz continuous, for any $\beta >0$, we have
\[
f(y_k-\beta d_k) - f(y_k)+\beta\langle\nabla f(y_k),d_k\rangle
\le \frac{{\ell}}{2}\beta^2\|d_k\|^2.
\]
Hence, if $\beta$ satisfies \eqref{eq:step_accept}, then it follows that 
\begin{align*}
f(y_k-\beta d_k) - f(y_k)+\beta\langle\nabla f(y_k),d_k\rangle
\le \dfrac{\ell}{2}\beta^2 \|d_k\|^2 \le \beta (1-\sigma) \langle z_k+\nabla f(y_k),d_k\rangle.
\end{align*}
The above gives
\begin{align*}
f(y_k-\beta d_k) 
&\le f(y_k)-\beta\langle\nabla f(y_k),d_k\rangle + \beta (1-\sigma)\langle z_k+\nabla f(y_k),d_k\rangle\\
&= f(y_k)+\beta\langle z_k,d_k\rangle
-\sigma\beta\langle z_k+\nabla f(y_k),d_k\rangle,
\end{align*}
thereby proving the first statement. The second statement follows from the first one. Indeed, let $N_k$ be the number of backtracking steps, so $\beta_k = \theta^{N_k}$. After $N_k\le \overline{N}_k$ backtracking iterations, where $\overline{N}_k$ is  defined in \eqref{eq:complexity_accept}, the stepsize satisfies $\beta_{k} \leq \Bar{\beta}_k.$ 
\end{proof}

The next result establishes the convergence and linear convergence rate of the sequences generated by {\bf Algorithm~\ref{alg:Mod_regularized_ESN}}, in the presence of a tilt-stable minimizer.

\begin{Theorem}
\label{thm:linear_rate}
Let $\{x_k\}$ be the sequence generated by {\bf Algorithm~\ref{alg:Mod_regularized_ESN}}. Suppose that each $x_k$ is non-stationary, and let $\bar{x}$ be the tilt-stable minimizer of $\varphi$.
We have the following assertions:
\begin{enumerate}[\bf (i)]
    \item The sequence $\{\varphi(x_k)\}$ converges to $\varphi(\ox)$ with convergence rate at least Q-linear.

    \item The sequence $\{x_k\}$ converges to $\bar x$ with convergence rate at least R-linear.
\end{enumerate}   
\end{Theorem}

\begin{proof}
By Proposition~\ref{prop:conv_modESN}(i), we have that $\{\varphi(x_k)\}$ is a decreasing sequence and is bounded from below; hence, $\{\varphi(x_k)\}$ converges to some finite value $\varphi^*$. But since $x_k\to \ox$ by Proposition~\ref{prop:conv_modESN}(ii) and $\varphi$ is continuous relative to $\dom g$, we deduce that
\[
\varphi(\ox) = \lim_{k\to \infty}\varphi (x_k) = \varphi^*.
\]
In case $\varphi(x_k) = \varphi (\ox)$ for some $k \in \N\in \{0\}$, then $\varphi(x_{k+p}) = \varphi (\ox)$ for every $p \geq 0$ since the sequence $\{\varphi(x_k)\}$ is decreasing to $\varphi (\ox)$. Moreover, the relations in \eqref{eq:vkdk>0} and \eqref{eq:suff_descent} together imply that for any $C \in (0, 1/\alpha - \ell/2)$, we have
\[
0\leq C\|y_{k+p}-x_{k+p}\|^2 \le  C\|y_{k+p}-x_{k+p}\|^2 +\sigma\beta_{k+p}\langle z_{k+p}+\nabla f(y_{k+p}),d_{k+p}\rangle \leq \varphi(x_{k+p})-\varphi(x_{k+p+1}) =  0.
\]
This yields $y_{k+p} = x_{k+p} = \ox$ for every $p \geq 0$, which contradicts our assumption that each $x_k$ is non-stationary. From now on, let us suppose that $\varphi (x_k) > \varphi (\ox)$ for all $k$.

Since $\ox$ is the tilt-stable minimizer of $\varphi$ and $\{x_{k}\}$ converges to $\ox$, by \eqref{eq:suff_descent} and \eqref{eq:estimate2} and \cite[Theorems~3.2~and~3.5]{BorisNghia15} there exists $k_0 \in \N$ such that whenever $k\geq k_0$, we have the estimates
\begin{equation}
\begin{aligned}
\big(\varphi(x_k) - \varphi(\ox)\big) - \big(\varphi(x_{k+1}) - \varphi(\ox)\big)
&=\varphi(x_k) - \varphi(x_{k+1})\\
&\geq C\|y_k-x_k\|^2 + \sigma\beta_k \la z_k + \nabla f(y_k), d_k\ra\\
& \geq C\|y_k-x_k\|^2 + \sigma\beta_k(1-\tau)\la H_kd_k,d_k\ra\\
&\geq C\|y_k-x_k\|^2 + \sigma\kappa(1-\tau)\beta_k^2\|d_k\|^2 \\
&\geq C\|y_k-x_k\|^2 + \sigma\kappa(1-\tau)\|x_{k+1} - y_k\|^2\\
&\geq C_1(\|y_k-x_k\| + \|x_{k+1} - y_k\|)^2 \label{eq:phik_k1},
\end{aligned}
\end{equation}
where $C_1:= \dfrac{1}{2}\min\{C;\,\sigma\kappa(1-\tau)\}$, and 
\begin{equation}\label{eq:quad_estimate}
\varphi(x_k)\ge \varphi(\bar x)+\frac{\kappa}{2}\|x_k-\bar x\|^2.
\end{equation}
Using \eqref{eq:quad_estimate}, we have, for any $v \in \partial \varphi(x_k)$, 
\[
\varphi(x_k) - \varphi(\ox) \leq \la v, x_k- \ox\ra\leq \|v\|\|x_k-\ox\|\leq M\|v\|\big(\varphi(x_k)-\varphi(\ox)\big)^{1/2},
\]
where $M:= \sqrt{2/\kappa}$. Taking the infimum over $v \in \partial \varphi(x_k)$, the former yields
\begin{equation}\label{eq:KL1}
\big(\varphi(x_k)-\varphi(\ox)\big)^{1/2} \leq M\,\d(0,\partial \varphi(x_k)),\qquad \forall\, k\ge k_0.
\end{equation}

Furthermore, observe from {\bf Algorithm~\ref{alg:Mod_regularized_ESN}} that $y_k \in \partial g^*(z_k)$ and $y_k - d_k \in \partial g^*(z_k)$, and thus we have 
\[
x_{k+1} = y_k - \beta_kd_k \in \partial g^*(z_k),\quad \text{by convexity}.
\]
This yields $z_k\in \partial g(x_{k+1})$, and hence $z_k + \nabla f(x_{k+1}) \in \partial \varphi(x_{k+1})$. Consequently,
\begin{equation}\label{eq:dist_est}
\begin{aligned}
\d(0,\partial \varphi(x_{k+1})) 
\leq \|z_k + \nabla f(x_{k+1})\| &\le \dfrac{1}{\alpha}\|y_k - x_k\| + \|\nabla f(x_{k+1}) - \nabla f(x_k)\|\\
&\leq \dfrac{1}{\alpha}\|y_k -x_k\| + \ell \|x_{k+1} - x_k\|
\leq C_2(\|y_k -x_k\| +\|x_{k+1} - y_k\|),
\end{aligned}
\end{equation}
where $C_2:= 1/\alpha + \ell$.

For each $k\ge k_0$, let $r_k:=\varphi(x_k)-\varphi(\ox)>0$. Combining
\eqref{eq:KL1}--\eqref{eq:dist_est} gives
\[
r_{k+1}^{1/2}
\leq M\operatorname{dist}(0,\partial\varphi(x_{k+1}))
\leq
MC_2\bigl(
    \|y_k-x_k\|+\|x_{k+1}-y_k\|
\bigr).
\]
Together with \eqref{eq:phik_k1}, this yields $r_{k+1}\leq
\frac{M^2C_2^2}{C_1}(r_k-r_{k+1}).$ Hence for $k\ge k_0$,
\[
r_{k+1}
\leq
\rho
r_k \qquad \text{for}\,\,\, \rho:=\dfrac{\frac{M^2C_2^2}{C_1}}{1+\frac{M^2C_2^2}{C_1}}=\frac{M^2C_2^2}{C_1+ M^2C_2^2}\in (0,1),
\]
which gives the Q-linear convergence rate of $\{r_k\}$ to zero. Thus,
$\{\varphi(x_k)\}$ converges Q-linearly to
$\varphi(\bar{x})$. Finally, since $r_{k+1}\leq \rho r_k$ for all $k\geq k_0$, it follows from \eqref{eq:quad_estimate} that
\[
\|x_k-\bar{x}\|
\leq
M\big(\varphi(x_k)-\varphi(\ox)\big)^{1/2}
=
M\,r_k^{1/2} \le M\,\rho^{(k-k_0)/2} r_{k_0}^{1/2}
= \big(Mr_{k_0}^{1/2}\big)\rho^{(k-k_0)/2},
\]
in which $\rho \in (0,1)$. Therefore, the sequence $\{x_k\}$ converges R-linearly to $\bar{x}$, which completes
the proof.

\end{proof}

In the next result, under the essenial polyhedrality of the convex function $g$, we establish the local quadratic convergence of {\bf Algorithm~\ref{alg:Mod_regularized_ESN}}.

\begin{Theorem}[Quadratic convergence of globalized Newton method with effective set] \label{theo:quad-poly2}
Suppose that $\bar{x}$ is the tilt-stable minimizer of the function $\varphi$ in \eqref{p:CP}, i.e., there exists $\kappa>0$ for which \eqref{eq:tilt} holds. Assume that $g$ is polyhedral and that the Hessian $\nabla^2 f(x)$ is Lipschitz continuous around $\ox$. Suppose further that {\bf Algorithm~\ref{alg:Mod_regularized_ESN}} generates an infinite sequence $\{x_k\}$ of non-stationary iterates. Then the iterate sequence $\{x_k\}$ and the objective value sequence $\{\varphi(x_k)\}$ converge quadratically to $\ox$ and $\varphi(\ox)$, respectively.

\end{Theorem}

\begin{proof}
Since $\xi_k\in N_{\mathcal{D}_k}(d_k)$, we have $\la \xi_k ,d-d_k\ra \leq 0$ for any $d \in \mathcal{D}_k$. It then follows from \eqref{eq:inexact_subproblem_condition} and the Cauchy-Schwarz inequality that
\begin{equation}
\begin{aligned}
\forall\, d\in \mathcal{D}_k,\quad\la H_kd_k -\left(z_k+\nabla f(y_k)\right),d - d_k\ra 
&=  \la H_kd_k -\left(z_k+\nabla f(y_k)\right) +\xi_k ,d-d_k\ra - \la \xi_k ,d-d_k\ra \\
&\geq \la H_kd_k -\left(z_k+\nabla f(y_k)\right) +\xi_k,d-d_k\ra\\
&\geq
-\| H_kd_k -\left(z_k+\nabla f(y_k)\right) +\xi_k\|\|d-d_k\|\\
&\geq -c\|z_k + \nabla f(y_k)\|^2\|d-d_k\|\label{eq:quad1}.
\end{aligned}
\end{equation}

It follows from the proof of Theorem~\ref{thm:full_step_acceptance_modESN} that $(y_k,z_k+\nabla f(y_k))\xrightarrow{{\rm gph}\, \partial \varphi}(\ox,0)$. Similarly to \eqref{eq:Noz_modESN}, we find a uniform constant $\kappa>0$ such that
\begin{equation}\label{eq:uniform-Lk}
    \la \nabla^2 f(y_k)w,w\ra \ge \kappa \|w\|^2 \quad \text{for all}\quad w\in \mathrm{span}\, \mathcal{D}_k.
\end{equation}
Since $y_k\xrightarrow{{\rm dom}\,g} \ox$ and $g$ is a polyhedral convex function, we may employ \cite[Proposition~3.3(a)]{MordukhovichSarabi16} to get the inclusions $z_k \in \partial g(y_k)\subset \partial g(\ox)$ for sufficiently large $k$, which further yield $\ox\in \partial g^* (z_k)$ for all such $k$. This implies, in particular, 
\begin{equation*}
    y_k -\ox \in \mathcal{D}_k\qquad \text{and}\qquad d_k - (y_k-\ox) \in \mathcal{D}_k - \mathcal{D}_k \subset \mathrm{span}\, \mathcal{D}_k.
\end{equation*}
Hence, it follows from \eqref{eq:quad1}--\eqref{eq:uniform-Lk} that 
\begin{equation}
\begin{aligned}
&\la H_k(y_k -\ox) - (z_k + \nabla f(y_k)),y_k - \ox-d_k\ra \\
&= \la H_k(d_k - (y_k -\ox)), d_k - (y_k -\ox)\ra + \la H_kd_k -\left(z_k+\nabla f(y_k)\right),y_k-\ox - d_k\ra\\
&\geq \la H_k(d_k - (y_k -\ox)), d_k - (y_k -\ox)\ra -c\|z_k + \nabla f(y_k)\|^2\,\|d_k - (y_k -\ox)\| \\
&\geq \la \nabla^2 f(y_k)(d_k - (y_k -\ox)), d_k - (y_k -\ox)\ra -c\|z_k + \nabla f(y_k)\|^2\,\|d_k - (y_k -\ox)\| \\
&\geq \kappa \|d_k - (y_k -\ox)\|^2 -c\|z_k + \nabla f(y_k)\|^2\,\|d_k - (y_k -\ox)\|.
\label{eq:quad2}
\end{aligned}
\end{equation}
Note that, similarly to \eqref{eq:connect_equiv} and \eqref{eq:projzkzbar-2_modESN}, we can derive
\begin{align}\label{eq:use-monotone}
\la H_k(y_k -\ox) - (z_k + \nabla f(y_k)),y_k - \ox-d_k\ra 
= \la H_k(y_k - \ox) + \nabla f(\ox) -\nabla f(y_k),y_k - \ox-d_k\ra.
\end{align}
Hence, using Cauchy-Schwarz inequality together with \eqref{eq:quad2}--\eqref{eq:use-monotone} gives
\begin{equation}\label{eq:quad2v1}
\|H_k(y_k - \ox) + \nabla f(\ox) -\nabla f(y_k)\|\, \|y_k - \ox-d_k\|\geq \kappa \|d_k - (y_k -\ox)\|^2 -c\|z_k + \nabla f(y_k)\|^2\,\|d_k - (y_k -\ox)\|.
\end{equation}
Since $H_k = \nabla^2f(y_k) + \delta_k\mathbb{I}_n$ and $\delta_k = \min\left\{\bar \delta;\,\frac{\|y_k-x_k\|}{\alpha}\right\}$, we derive from \eqref{eq:1-di-2_modESN} and \eqref{eq:y0xbar-rem-1} that
\begin{equation}
\begin{aligned}
\left\|H_k(y_k-\ox) + \nabla f(\ox)-\nabla f(y_k)\right\| 
&\leq \left\|\nabla^2f(y_k)(y_k-\ox) + \nabla f(\ox)-\nabla f(y_k)\right\| +\delta_k\|y_k-\ox\| \\
&\leq \dfrac{\ell}{2}\|y_k-\ox\|^2 + \dfrac{1}{\alpha}\big(\|y_k-\ox\|+\|x_k-\ox\|\big)\|y_k-\ox\|  \\
&\leq \left(\dfrac{\ell}{2}+\dfrac{1}{\alpha}\right)(1+\ell \alpha)^2\|x_k-\ox\|^2 + \dfrac{1+\ell \alpha}{\alpha}\|x_k-\ox\|^2\\
&\le \dfrac{1+\ell \alpha}{\alpha}\left(\dfrac{(1+\ell \alpha)(2+\ell\alpha)}{2}+1\right)\|x_k-\ox\|^2.
\label{eq:quad4}
\end{aligned}
\end{equation}
Dividing both sides of \eqref{eq:quad2v1} by $\kappa\|d_k - (y_k -\ox)\|$ and combining with \eqref{eq:quad4}, we arrive at
\begin{equation}\label{eq:quad5}
\begin{aligned}
\|d_k - (y_k -\ox)\| &\le \dfrac{1}{\kappa} \|H_k(y_k - \ox) + \nabla f(\ox) -\nabla f(y_k)\| + \dfrac{c}{\kappa}\|z_k + \nabla f(y_k)\|^2 \\
&\leq \dfrac{1+\ell \alpha}{\kappa\alpha}\left(\dfrac{(1+\ell \alpha)(2+\ell\alpha)}{2}+1\right)\|x_k-\ox\|^2+ \dfrac{c}{\kappa}\|z_k + \nabla f(y_k)\|^2.
\end{aligned}
\end{equation}
Note by \eqref{eq:y0xbar-rem-1}--\eqref{eq:z0zbar} that
\begin{equation}\label{eq:quad6}
\begin{aligned}
\|z_k + \nabla f(y_k)\| &\leq \|z_k-\bar z\|+ \|\bar z + \nabla f(y_k)\|\\
&\leq (1+\ell)(\|z_k-\bar z\|+\|y_k-\ox\|) \leq  (1+\ell)\eta\|x_k-\ox\|.
\end{aligned}
\end{equation}
where $\eta:=\frac{(1+\ell\al)(2+\al)}{\al}$. Since the step size $\beta_k = 1$ is eventually accepted by Theorem~\ref{thm:full_step_acceptance_modESN}(i), we derive from \eqref{eq:quad5}--\eqref{eq:quad6} that
\begin{align*}
\|x_{k+1} - \ox\|=\|d_k - (y_k -\ox)\| 
&\leq C\|x_k-\ox\|^2,
\end{align*}
where the constant $C$ is 
\[
C :=  \dfrac{1+\ell \alpha}{\kappa\alpha}\left(\dfrac{(1+\ell \alpha)(2+\ell\alpha)}{2}+1\right) + \dfrac{c}{\kappa}(1+\ell)^2 \eta^2.
\]
Thus, the local quadratic convergence of the iterate sequence $\{x_k\}$ generated {\bf Algorithm~\ref{alg:Mod_regularized_ESN}} is proved. Following the same arguments as Theorem~\ref{theo:quad-poly1}, the local quadratic convergence of the objective value sequence $\{\varphi(x_k)\}$ can also be shown. The proof is complete.
\end{proof}

\section{Solving subproblem for some models with polyhedral regularizers}
\label{sec:efficient_implementation}

In this section, we describe an implementable subproblem solver for {\bf Algorithm~\ref{alg:Mod_regularized_ESN}} for the problem
\begin{equation}\label{eq:leastsquare}
\min_{x\in\mathbb R^n}\quad \frac12\|Ax-b\|^2+ g(x),
\end{equation}
where $A\in\mathbb R^{m\times n}$ $(m\ll n)$, $b\in\mathbb R^m$, and $g$ is a sparsity-inducing polyhedral regularizer. Such models arise widely in machine learning, signal processing, and statistics, where different regularizers encode different forms of structured sparsity. We focus on three representative classes, namely the Lasso, the fused Lasso and the nonnegative least squares problem.

Fix a nonstationary outer iteration $k$.  For the least-squares loss,
$\nabla^2 f(y_k)=A^\top A$ and the regularized subproblem
\eqref{eq:regularized_subproblem} becomes
\begin{equation}
    \min_{d\in\mathcal{D}_k}
    \quad
        \frac12\big\langle (A^\top A + \delta_k \mathbb{I}_n) d,d\big\rangle
        -\langle z_k+A^\top(Ay_k-b),d\rangle,
        \label{eq:subprob_ls}
\end{equation}
where $\mathcal{D}_k=y_k - \partial g^* (z_k)$. With the change of variables $d=y_k-u$, where $u\in \partial g^* (z_k)$, problem
\eqref{eq:subprob_ls} is, up to a constant independent of $u$, equivalent to
\begin{equation}
\min_{u\in\partial g^*(z_k)}
\quad \frac12\|Au\|^2+\frac{\delta_k}{2}\|u\|^2-\langle A^\top b+\delta_k y_k-z_k,u\rangle.
\label{eq:u_subproblem}
\end{equation}
Solving the primal problem \eqref{eq:u_subproblem} directly is possible using standard variants of projected gradient methods. However, since the primal problem is constrained and has dimension $n$, with $n\gg m$, it is computationally preferable to work with its $m$-dimensional dual formulation. We therefore first derive the corresponding dual problem. Define $c_k:=A^\top b+\delta_k y_k-z_k.$
Then, by Fenchel--Rockafellar duality, the dual of \eqref{eq:u_subproblem} is

\begin{equation}\label{eq:inner_dual}
\min_{v\in\mathbb R^m}\quad \Psi_k (v):=\frac12\|v\|^2
-\frac{\delta_k}{2}\left\|P_{\partial g^*(z_k)}
\left(\frac{c_k-A^\top v}{\delta_k}\right)\right\|^2
+\left\la c_k - A^\top v, P_{\partial g^*(z_k)}
\left(\frac{c_k-A^\top v}{\delta_k}\right) \right\ra.
\end{equation}
Since the primal objective function is strongly convex on the feasible set $\partial g^* (z_k)$, the dual objective $\Psi_k$ is continuously differentiable; see, e.g., \cite[Theorem~26.3]{Rockafellar70}. Moreover, observe that since $\Psi_k$ is also strongly convex,
\eqref{eq:inner_dual} admits a unique minimizer $v^*$, characterized by the equation
\begin{equation}
\nabla\Psi_k(v^*)=v^*-AP_{\partial g^*(z_k)}
\left(\frac{c_k-A^\top v^*}{\delta_k}\right) =0.
\label{eq:inner_system}
\end{equation}

Since the objective function in \eqref{eq:inner_dual} is continuously
differentiable,
{\em conjugate gradient method} is a plausible choice for solving \eqref{eq:inner_dual}. Here, we employ the conjugate gradient method in \cite{HagerZhang2013}.














\begin{algorithm}[H]
\caption{Nonlinear Conjugate Gradient for \eqref{eq:inner_dual}}
\label{alg:hz_inner}
\textbf{Input:} $v^0\in\mathbb{R}^m$, $0<c_1<c_2<1$, 
$u^0:=P_{\partial g^*(z_k)}\big((c_k-A^\top v^0)/\delta_k\big)$,
$q^0:=v^0-Au^0$, $p^0:=-q^0$, and $j:=0$.\\
\textbf{Repeat}
\begin{algorithmic}[1]

    \STATE Set
    $d^j:=y_k-u^j$ and
    $\xi^j:=A^\top v^j+\delta_k u^j-c_k$.

    \STATE If \eqref{eq:dual_inexact_cond} holds with
    $(v,d)=(v^j,d^j)$, stop and return
    $(d_k,\xi_k):=(d^j,\xi^j)$.

    \STATE Choose $\alpha_j>0$ satisfying
    \[
        \Psi_k(v^j+\alpha_jp^j)
        \le
        \Psi_k(v^j)+c_1\alpha_j\langle q^j,p^j\rangle,
        \qquad
        \big\langle\nabla\Psi_k(v^j+\alpha_jp^j),p^j\big\rangle
        \ge
        c_2\langle q^j,p^j\rangle.
    \]

    \STATE Set
    $v^{j+1}:=v^j+\alpha_jp^j$ and
    $u^{j+1}:=P_{\partial g^*(z_k)}
    \big((c_k-A^\top v^{j+1})/\delta_k\big)$.

    \STATE Set
    $q^{j+1}:=v^{j+1}-Au^{j+1}$ and
    $w^j:=q^{j+1}-q^j$.

    \STATE Set
    \[
        \beta_j
        :=
        \max\left\{
        \frac{\langle w^j,q^{j+1}\rangle}
             {\langle p^j,w^j\rangle}
        -
        \frac{\|w^j\|^2\langle p^j,q^{j+1}\rangle}
             {\langle p^j,w^j\rangle^2};
        \,
        0.4\frac{\langle p^j,q^j\rangle}{\|p^j\|^2}
        \right\}.
    \]

    \STATE Set
    $p^{j+1}:=-q^{j+1}+\beta_jp^j$ and
    $j\leftarrow j+1$.

\end{algorithmic}
\end{algorithm}

The following proposition connects the dual residual with the inexactness
condition \eqref{eq:inexact_subproblem_condition} required by the outer method in {\bf Algorithm~\ref{alg:Mod_regularized_ESN}}.
\begin{Proposition}[Dual inexact certificate]
\label{prop:dual_certificate}
Fix $k \in \N\cup \{0\}$, for any $v_j\in\mathbb R^m$ with some $j \in \N$, set
\begin{equation}
u_k^j:=P_{\partial g^*(z_k)}
\left(\frac{c_k-A^\top v_j}{\delta_k}\right),\,\,\, d_k^j:=y_k-u_k^j \quad \text{and}\quad\xi_k^j:=A^\top v_j+\delta_ku_k^j-c_k.
\label{eq:recovery_definitions}
\end{equation}
Then $(d_k^j,\xi_k^j)\in\operatorname{gph}N_{{\mathcal{D}}_k}$ and
\begin{equation}
    H_kd_k^j-(z_k+\nabla f(y_k))+\xi_k^j
    =A^\top\nabla\Psi_k(v_j).
    \label{eq:certificate_identity}
\end{equation}
Consequently, the pair $(d_k^j,\xi_k^j)$ satisfies \eqref{eq:inexact_subproblem_condition} whenever
\begin{equation}\label{eq:dual_inexact_cond}
\begin{cases}
\|A^\top\nabla\Psi_k(v_j)\|
\le
\min\left\{
   \eta\|z_k+\nabla f(y_k)\|,
   c\|z_k+\nabla f(y_k)\|^2
\right\}\\
\la A^\top \nabla\Psi_k(v_j),d_k^j\ra\leq \la \tau H_kd_k^j+\xi_k^j,d_k^j\ra.
\end{cases}
\end{equation}
\end{Proposition}
\begin{proof}
The first-order optimality condition of \eqref{eq:u_subproblem} gives $c_k-A^\top v-\delta_ku
\in N_{\partial g^*(z_k)}(u)$.
Moreover, $\mathcal{D}_k=y_k-\partial g^*(z_k)$, which implies that
\begin{equation}
N_{\mathcal{D}_k}(d) = N_{y_k-\partial g^*(z_k)}(y_k -u)=-N_{\partial g^*(z_k)}(u).
\label{eq:translated_normal}
\end{equation}
The definition of $\xi$ therefore yields $\xi\in N_{\mathcal{D}_k}(d)$. Finally, using $\nabla f(y_k)=A^\top(Ay_k-b)$ and $c_k = A^\top b + \delta_ky_k -z_k$,
we obtain
\begin{align*}
H_kd-(z_k+\nabla f(y_k))+\xi
&=(A^\top A+\delta_k \mathbb{I}_n)(y_k-u)
  -z_k-A^\top(Ay_k-b)+\xi\\
&=-A^\top Au+c_k-\delta_ku+\xi=A^\top\nabla\Psi_k(v).
\end{align*}
This proves \eqref{eq:certificate_identity}. Substituting \eqref{eq:certificate_identity} into the first condition of \eqref{eq:inexact_subproblem_condition} gives that of \eqref{eq:dual_inexact_cond}. Moreover, notice that the second condition of \eqref{eq:dual_inexact_cond} is also equivalent to that of \eqref{eq:inexact_subproblem_condition}. Indeed, using  \eqref{eq:certificate_identity}, we deduce that
\begin{align*}
&\la A^\top \nabla\Psi_k(v_j),d_k^j\ra\leq \la \tau H_kd_k^j+\xi_k^j,d_k^j\ra\\
\Longleftrightarrow\,\, &\la H_kd_k^j-(z_k+\nabla f(y_k))+\xi_k^j,d_k^j\ra - \la \tau H_kd_k^j+\xi_k^j,d_k^j\ra \leq 0\\
\Longleftrightarrow\,\, &\la z_k+\nabla f(y_k),d_k^j\ra \geq (1-\tau) \la H_kd_k^j,d_k^j\ra,
\end{align*}
which establish the equivalence of \eqref{eq:inexact_subproblem_condition} and \eqref{eq:dual_inexact_cond}. The proof is complete.
\end{proof}

To evaluate the gradient as in \eqref{eq:inner_system}, one need to be able to evaluate the projection $P_{\partial g^*(z_k)}$. In the subsequent analysis, we show that
for some classes of problems with popularly used in machine learning and statistics, this projection can be computed efficiently.

{\bf (i) Nonnegative least squares problem.} In this problem, the regularizer is the indicator of the nonnegative orthant $\R^n_+$, i.e., $g(x):=\delta_{\R^n_+}(x)$  for $x\in \R^n$. We have
\begin{equation}\label{eq:g*delta}
    g^* =(\delta_{\R^n_+})^* = \delta_{(\R^n_+)^\circ} = \delta_{\R^n_-},
\end{equation}
where $\R^n_{-}$ is the nonpositive orthant. For any $z\in \R^n_{-}$, we derive 
\begin{equation}
\partial g^* (z) = N_{\R^n_{-}}(z) = \left\{x\in \R^n \;\Bigg|\; 
    \begin{cases}
    x_i =0, &z_i<0 \\
    x_i \ge 0, &z_i=0
    \end{cases}
    \right\}.
\end{equation}
Thus, the projection onto $\partial g^*(z)$ is defined elementwise as
\[
\left(P_{\partial g^*(z)}(w))\right)_i:=
\begin{cases}
0, & z_i <0,\\
\max\{w_i,0\}, & z_i=0.
\end{cases}
\]

{\bf (ii) Lasso problem.} In this problem, the regularizer is the $\ell_1$ norm, i.e., $$g(x):=\lambda\sum_{i=1}^n |x_i| \quad x\in \R^n,$$ 
where $\lambda>0$ is the regularization parameter. We have
\begin{equation}\label{eq:g*delta}
    g^* =(\lambda\|\cdot\|_{1})^* = (\sigma_{\lambda\mathbb{B}_{\infty}})^* = \delta_{\lambda\mathbb{B}_{\infty}},
\end{equation}
where $\B_\infty=\{v\in \R^n \,|\, |v_i|\le 1,\, i=1,\ldots, n\}$ is the $\ell_\infty$ unit ball in $\R^n$. For any $z\in \lambda\B_\infty$, it is well-known that 
\begin{equation}\label{eq:normal_Binfinity}
\partial g^* (z) = N_{\lambda\mathbb{B}_{\infty}}(z) = \left\{x\in \R^n \;\Bigg|\; 
    \begin{cases}
    x_i =0, &|z_i|<\lambda \\
    x_i z_i, \ge 0 &|z_i|=\lambda
    \end{cases}
    \right\}.
\end{equation}
Using the above formula, we derive the projection onto $\partial g^*(z)$ elementwise as
\[
\left(P_{\partial g^*(z)}(w)\right)_i:=
\begin{cases}
0, & |z_i| <\lambda,\\
\max\{w_i,0\}, & z_i=\lambda,\\
\min\{w_i,0\}, & z_i=-\lambda.
\end{cases}
\]
Let $J_k:= \{i \in \{1,\ldots,n\} |\, |(z_k)_i| <\lambda\}$. Observe that instead of implementing {\bf Algorithm~\ref{alg:hz_inner}} using naive matrix-free matrix-vector  products $A^{\top}v$ with $\mathcal{O}(mn)$, one  can exploit the sparsity of the projection formula and use $A^{\top}_{J_k^c}v$ with reduced complexity $\mathcal{O}(m|J_k^c|).$

{\bf (iii) Fused Lasso problem.}
In this problem, $g$ is the addition of the $\ell_1$-norm with the 1D total variation semi-norm, i.e., for $\lambda_1,\lambda_2>0$ let $$g(x):= \lambda_1\|x\|_1+\lambda_2\|Dx\|_1,$$ 
where $D \in \R^{(n-1)\times n}$ is the {\em difference matrix} such that
\[
(Dx)_i:=x_i-x_{i+1}, \quad i=1,\ldots,n-1.
\]
Denote by $h:=\|\cdot\|_1$ the $\ell_1$ norm on $\R^{2n-1}$. 
Then the fused Lasso regularizer can be written as the composition $g(x)=h(\widetilde{D}x)$ where $\widetilde D
:=
\begin{bmatrix}
\lambda_1 I_n & \lambda_2 D^\top
\end{bmatrix}^{\!\top}
\in\R^{(2n-1)\times n}$. For any $z\in \R^n$, the Fenchel-Rockafellar conjugate formula \cite[Theorem~16.3]{Rockafellar70} gives us that  
\begin{equation}\label{eq:FR}
    g^* (z) = \inf \left\{\delta_{\B^{2n -1}_{\infty}}(y)\mid \Tilde{D}^*y=z\right\}= \delta_{\Tilde{D}^*(\B^{2n-1}_{\infty})}(z).
\end{equation}
where $\B^{2n-1}_\infty=\{v\in \R^{2n-1} \,|\, |v_i|\le 1,\, i=1,\ldots, 2n-1\}= \B^{n}_\infty\times \B^{n-1}_\infty$ is the $\ell_\infty$ ball in $\R^{2n-1}$. Hence, for $z\in\dom g^*=\Tilde{D}^*(\B^{2n-1}_{\infty})$, 
there exists $q=(\begin{matrix}
a^\top & b^\top
\end{matrix})^{\!\top} \in \B^{2n-1}_{\infty}$ with $a \in \R^n, b \in \R^{n-1}$ such that 
\[
z = \Tilde{D}^*q = \lambda_1a + \lambda_2D^\top b.
\]
Moreover, \cite[Theorem 6.43]{Rockafellar98} gives
$$
x\in\partial g^*(z)
\iff x \in N_{\Tilde{D}^*(\B^{2n-1}_{\infty})}(\Tilde{D}^*q) \iff
\Tilde{D}x\in N_{\B^{2n-1}_{\infty}}(q).
$$
Let $u:=\lambda_1 a, v := \lambda_2b$. Because $N_{\B^{2n-1}_{\infty}}(q) = N_{\B^{n}_{\infty}}(a) \times N_{\B^{n-1}_{\infty}}(b)$, we deduce from the above that 
\begin{equation*}
\begin{aligned}
\partial g^*(z) 
&= \left\{x \in \R^n |\, \lambda_1x\in N_{\B^{n}_{\infty}}(a) \text{ and } \lambda_2 Dx\in N_{\B^{n-1}_{\infty}}(b)  \right\}\\
&= \left\{x \in \R^n |\, x\in N_{\lambda_1\B^{n}_{\infty}}(u) \text{ and }Dx\in N_{\lambda_2\B^{n-1}_{\infty}}(v)  \right\}.
\end{aligned}
\end{equation*}
Since $(Dx)_i=x_i-x_{i+1}$, using \eqref{eq:normal_Binfinity} gives
\begin{equation}
\partial g^*(z)
=
\left\{
x\in\R^n
\ \middle|\
\begin{aligned}
&\begin{cases}
x_i=0, & |u_i|<\lambda_1,\\
u_i x_i\ge 0, & |u_i|=\lambda_1,
\end{cases}
&& i=1,\ldots,n,\\[1mm]
&\begin{cases}
x_i=x_{i+1}, & |v_i|<\lambda_2,\\
v_i(x_i-x_{i+1})\ge 0, & |v_i|=\lambda_2,
\end{cases}
&& i=1,\ldots,n-1
\end{aligned}
\right\}.
\end{equation}
Thus, for any $z= u + D^\top v$ with $\|u\|_{\infty}\leq \lambda_1, \|v\|_{\infty}\leq \lambda_2$, the projection onto $\partial g^*(z)$ can be computed by solving
\begin{equation}\label{eq:proj-fused-lasso}
\begin{aligned}
P_{\partial g^*(z)}(w)\in\argmin_{x\in \partial g^* (z)}\,\, \frac{1}{2}\sum_{i=1}^n (x_i-w_i)^2.
\end{aligned}
\end{equation}

Define
$\mathcal I(z):=\{i\in \{1,\ldots,n-1\}|\,|v_i|=\lambda_2\}$. Set $p:=|\mathcal I(z)|$ and, when
$p>0$, sort its elements in increasing order as
$t_1<\ldots<t_p$. Let $t_0:=0$, $t_{p+1}:=n$, and the index set
\[
G_r:=\{t_{r-1}+1,\ldots,t_r\},
\qquad r=1,\ldots,p+1.
\]
Notice that for every for $ t_{r-1}<i<t_r$ and $r = 1,\ldots,p+1$, $|v_i| < \lambda_2$, which further implies $x_i = x_{i+1}$. Therefore, $x$ has common constant value on every group belongs to $G_r$. Denote this value to be $\theta_r$, then we obtain
\begin{equation}\label{eq:decompose}
\begin{aligned}
\frac12 \sum_{i=1}^n (x_i-w_i)^2 
= \frac12 \sum_{r=1}^{p+1} \sum_{i=t_{r-1}+1}^{t_r} (\theta_r-w_i)^2 
= \frac12 \sum_{r=1}^{p+1} \left[|G_r|(\theta_r - \Bar{w}_r)^2 + \sum_{i \in G_r} (w_i-\Bar{w}_r)^2\right],
\end{aligned}
\end{equation}
where $\Bar w_r:=|G_r|^{-1}\sum_{i\in G_r}w_i$. Since the second sum on the right hand side of the last equation does not depend on $\theta$, we can drop it from the objective function. Hence, the projection in \eqref{eq:proj-fused-lasso} can be found by solving for $\theta^*$ such that
\begin{equation}\label{eq:proj-fused-reduced}
\begin{aligned}
\theta^*\in\argmin_{\theta\,\in\,\R^{p+1}}\quad
&\frac12\sum_{r=1}^{p+1}|G_r|(\theta_r-\bar w_r)^2\\
\text{subject to}\quad
&
\begin{cases}
\theta_r=0,
& |u_i|<\lambda_1,\quad i\in G_r,\\
u_i\theta_r\ge0,
& |u_i|=\lambda_1,\quad i\in G_r,\\
v_{t_r}(\theta_r-\theta_{r+1})\ge0,
& r=1,\ldots,p,
\end{cases}
\qquad r=1,\ldots,p+1.
\end{aligned}
\end{equation}
The projection is then recovered from
$[P_{\partial g^*(z)}(w)]_i=\theta_r^*$ for
$i\in G_r$. The
dynamic programming isotonic regression method in
\cite[Sections~3.3--3.4]{weiss2019contrast} solves this problem exactly
in $\mathcal{O}(p+1)$ operations. In case $p=0$, let $\Bar{w}:= n^{-1}\sum_{i=1}^n w_i$ ,the projection $P_{\partial g^*(z)}(w)$ is explicitly computed as
\[
P_{\partial g^*(z)}(w)
:=
\begin{cases}
\max\{\Bar{w},0\}\mathbf{1}_n,
& u_i=\lambda_1, \forall i\\
\min\{\Bar{w},0\}\mathbf{1}_n,
& u_i=-\lambda_1, \forall i\\
0_n,
& \text{otherwise}.
\end{cases}
\]

\begin{Remark}[Recovery of the fused-Lasso subgradient components]
{\rm For a fixed iteration $k\in \N\cup \{0\}$, the components $u$ and $v$ in the representation $z_k = u + D^\top v$ can be recovered directly from quantities computed when evaluating the proximal-gradient step from the outer loop. Indeed, observe that in Step 1 of {\bf Algorithm~\ref{alg:Mod_regularized_ESN}}, one has
\begin{equation}\label{eq:fused_prox}
y_k=\prox_{\alpha(\lambda_1\|\cdot\|_1+\lambda_2\|D \cdot\|_1)}(x_k-\alpha\nabla f(x_k)) = \prox_{\alpha \lambda_1\|\cdot\|_1}\left(\prox_{\alpha\lambda_2\|D \cdot\|_1}(x_k-\alpha\nabla f(x_k))\right)
\end{equation}
Let $\Tilde{y}_k := \prox_{\alpha\lambda_2\|D \cdot\|_1}(x_k-\alpha\nabla f(x_k))$ and $y_k =\prox_{\alpha\lambda_1\|\cdot\|_1}(\Tilde{y}_k)$, then \eqref{eq:res} implies that
\[
\dfrac{ \Tilde{y}_k - y_k}{\alpha} = \lambda_1 a\quad \text{and} \quad     \dfrac{x_k-\alpha\nabla f(x_k) - \Tilde{y}_k}{\alpha} = \lambda_2 D^\top b.
\]
where $a \in \partial \|y_k\|_1$ and $b\in \partial \| D\Tilde{y}\|_1$, respectively. Let $u:= \lambda_1a, v:= \lambda_2 b$, note that since $D$ is full row rank, $DD^\top$ is invertible. One can derive from the above that 
\[
u =\dfrac{ \Tilde{y}_k - y_k}{\alpha} \quad \text{and} \quad   v=(DD^\top)^{-1}D \left( \dfrac{x_k-\alpha\nabla f(x_k) - \Tilde{y}_k}{\alpha} \right).\]
It follows from Step 2 of {\bf Algorithm~\ref{alg:Mod_regularized_ESN}} that
\[
z_k = \dfrac{x_k- y_k}{\alpha} - \nabla f(x_k) = \dfrac{x_k-\alpha \nabla f(x_k)- \Tilde{y}_k + \Tilde{y}_k -y_k}{\alpha} = u +D^\top v.
\]
Thus, the required pair $(u,v)$ is available from the proximal-gradient step of the outer procedure.}
\end{Remark}

\section{Numerical experiments}\label{sec:numerical}
In this section, we evaluate the numerical performance of
Algorithm~\ref{alg:Mod_regularized_ESN} on the regularized least-squares
problem~\eqref{eq:leastsquare}, where the regularizer $g$ is either the Lasso
or the fused-Lasso regularizer studied in Section~\ref{sec:efficient_implementation}.
We conduct two complementary sets of experiments, corresponding to the two
typical settings in which the proposed algorithm may be used: synthetic data
and real-world data. In the first set, we use randomly generated instances of
varying dimensions to assess the efficiency and scalability of the proposed
method against several well-established first-order and second-order
algorithms. In the second set, we evaluate the algorithm on real datasets and
summarize the results using the performance profiles of~\cite{DolanMore2002}.
These datasets range from medium-scale biomedical datasets, originally from the Kent Ridge
Biomedical Data Set Repository~\cite{LiLiu2002KentRidge}, to large-scale datasets from the LIBSVM
library~\cite{chang2011libsvm}.

\subsection{Implementation details}

For all of our numerical experiments, we choose the initial iterate $x_0 := 0$ and use the following \emph{relative KKT residual} stopping criterion, suggested in~\cite{LST18,LST18Fused}, to determine the quality of an approximate optimal solution $x_k$. Specifically, we terminate the algorithms when one of the following condition satisfied: either
\begin{equation}\label{eq:eta_residual}
\eta
:=\dfrac{\|x_k - \prox_{g} (x_k - A^T(Ax_k - b))\|}{1+\|x_k\| + \|Ax_k - b \|} \leq 10^{-6},
\end{equation}
or the maximum number of iterations reaches $10000$, or when the maximum runtime permitted is reached. All experiments are run in MATLAB R2026a on an AMD Ryzen 9 5900X with 32 GB of RAM.




\subsection{Synthetic dataset regression experiments}

In this section, to examine the performance of the tested methods under a controlled setting, we consider randomly generated instances of the regularized least-squares problem. We use three increasingly large dimensions
\[
(m,n)\in\{(1000,3000),(2000,6000),(4000,12000)\}.
\]
This configuration allows the underdetermination ratio $n/m=3$ remains fixed while both dimensions double successively. For each dimension, the entries of $A\in\mathbb{R}^{m\times n}$ are generated independently from the standard normal distribution $\mathcal{N}(0,1)$, after which every column is normalized to have unit Euclidean norm. For each experiment, the reference signals $x^\dagger$, the observed signal $b$ and regularization parameters are generated as follows:

\begin{enumerate}[\bf (i)]

\item \textbf{Lasso regularizer.}
The reference signal $x^\dagger\in\mathbb{R}^n$ has $s=0.05n\in\{150,300,600\}$ nonzero components for $n\in\{3000,6000,12000\}$, respectively. Its support is selected uniformly at random among all subsets of $\{1,\ldots,n\}$ having cardinality $s$. For every selected index $i$, we set $x_i^\dagger=\xi_i a_i$, where $\xi_i\in\{-1,1\}$ is chosen with equal probability and $a_i$ is drawn uniformly from $[0.5,1.5]$; all remaining components are set to zero. Thus, for example, when $n=3000$, the signal $x^\dagger$ has exactly $150$ randomly located nonzero components. We choose
\[
\lambda=10^{-3}\|A^\top b\|_\infty
\]
and generate $b=Ax^\dagger+\varepsilon$, where $\varepsilon\sim\mathcal{N}(0,10^{-3}\mathbb{I}_m)$.

\item \textbf{Fused-Lasso regularizer.}
The coordinates of $x^\dagger\in\mathbb{R}^n$ are partitioned into $q=0.02n\in\{60,120,240\}$ consecutive blocks of length $50$. Half of the block levels are set to zero, while each remaining block is assigned a constant value $\xi_j a_j$, where $\xi_j\in\{-1,1\}$ is chosen with equal probability and $a_j$ is drawn uniformly from $[0.5,1.5]$. For example, the first few blocks may have the form
\[
x^\dagger=
\big(
\underbrace{0,\ldots,0}_{50},
\underbrace{1.2,\ldots,1.2}_{50},
\underbrace{0,\ldots,0}_{50},
\underbrace{-0.8,\ldots,-0.8}_{50},
\ldots
\big)^\top,
\]
with the remaining blocks generated in the same manner. This construction produces a sparse piecewise-constant signal consistent with the structure promoted by the fused-Lasso regularizer. We choose
\[
\lambda_1=10^{-3}\|A^\top b\|_\infty,
\qquad
\lambda_2=\lambda_1,
\]
and again generate $b=Ax^\dagger+\varepsilon$, where $\varepsilon\sim\mathcal{N}(0,10^{-3}I_m)$.

\end{enumerate}

\begin{figure}[h]
    \centering
    \includegraphics[width=0.99\linewidth]{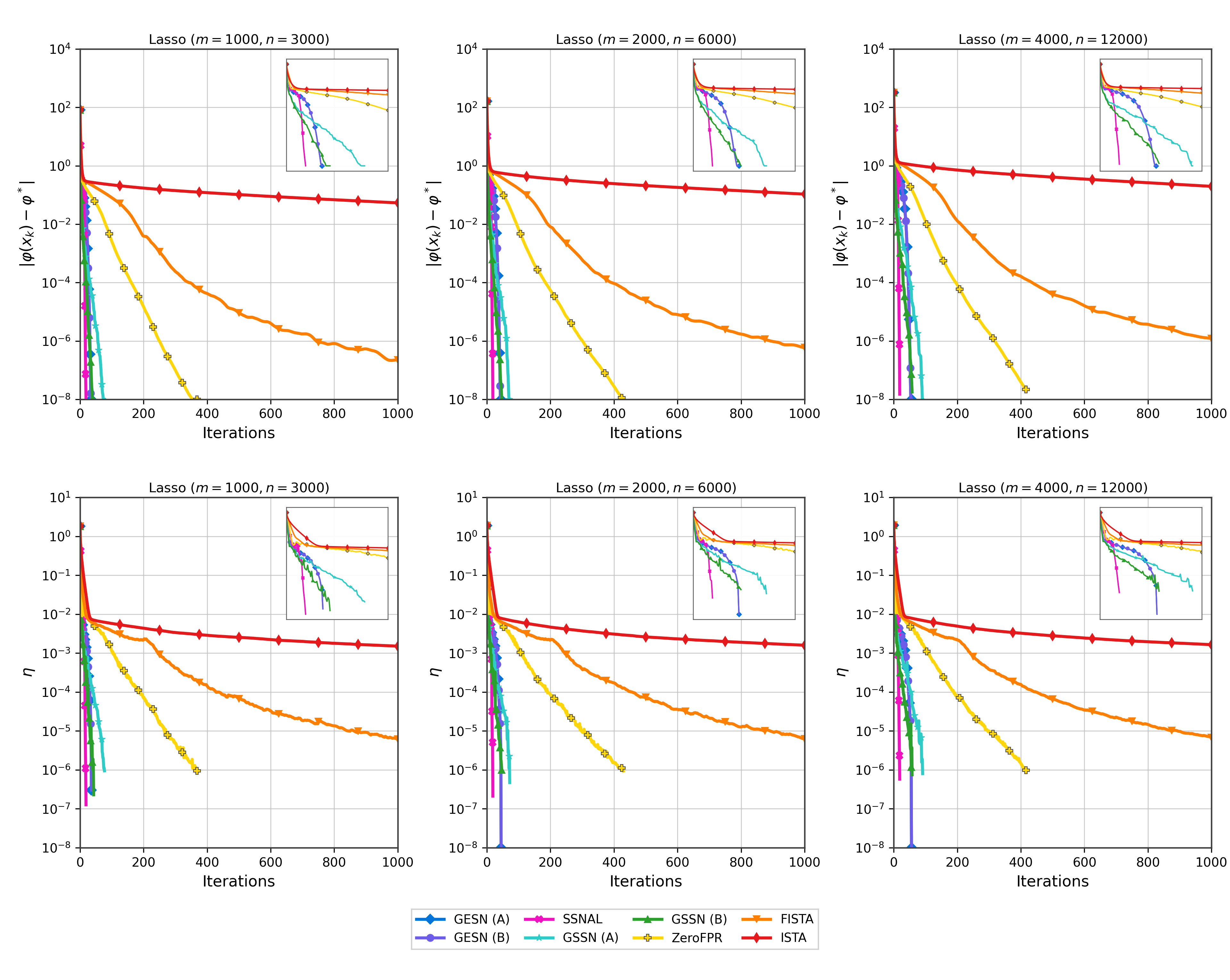}
    \caption{Plots for Lasso problem}
    \label{fig:random_lasso}
\end{figure}
Regarding the hyperparameter setting of {\bf Algorithm~\ref{alg:Mod_regularized_ESN}}, we use $\alpha = 1.9/\|A\|^2$, $\bar\delta = 0.5$,
and the inexactness parameters $\eta = 0.99$, $c = 10^{6}$, and
$\tau = 0.25$. The line search uses $\sigma = 0.25$
and the backtracking factor $\theta = 0.5$. The
subproblem is solved by the nonlinear preconditioned CG method in \cite{HagerZhang2013} with a Nyström-scaled spectral transform with rank at most 256. We tested two variants of GESN, GESN(A) and GESN(B), in which the A variant uses a larger maximum number of inner iterations, $100000$ versus $800$ for the other variant. For this experiment, we consider the following algorithms:

\begin{enumerate}
\item The Iterative Shrinkage-Thresholding Algorithm (ISTA) in \cite{BeckTeboulle09} with fixed stepsize $1.9/\|A\|^2$.
\item The Fast Iterative Shrinkage-Thresholding Algorithm (FISTA) in \cite{BeckTeboulle09} with fixed stepsize $1/\|A\|^2$.
\item The Semismooth Newton Augmented Lagrangian method (SSNAL) in \cite{LST18,LST18Fused}.
\item The ZeroFPR method with a FBE linesearch in \cite{T18}.
\item The Globalized SCD Semismooth$^{*}$ Newton method (GSSN) in \cite{GfrererCOA25}.
\end{enumerate}

For the settings of the compared algorithms, we run SSNAL and ZeroFPR with
their default configurations of the released code, which are presented in the Appendix. For GSSN,
we find that the publicly released implementation~\cite{GfrererCOA25} solves
the Newton subproblem differently from Algorithm~2 of~\cite{GfrererCOA25}, so
we report two versions that differ largely in how the inner solver is stopped. The first version,
which we name GSSN(A), terminates the inner solver once the relative residual
criterion of~\cite[Algorithm~2, Step~2]{GfrererCOA25} is satisfied. The maximum number of iteration for the inner solver of GSSN(A) is set to
$100000$ iterations. The other version, GSSN(B), uses the default settings of the released code that can also be found in the Appendix. It stops the subproblem solver once the relative residual of the reduced Newton system falls
below the fixed tolerance $0.1$, and it allows at most $30$ inner iterations per
outer iteration. In this experiment, all methods are terminated when the specified stopping criterion \eqref{eq:eta_residual} is satisfied or the prescribed time limit of $600$ seconds.

To evaluate convergence, we consider two type of plots. The first plots the objective gap $|\varphi(x_k)-\varphi^*|$
with respect to the number of iterations, where $\varphi^*$ is the minimum objective value reached by all methods in the considered experiment. The second plots the relative residual defined in \eqref{eq:eta_residual} with respect to the number of iterations. The convergence results for Lasso and fused Lasso are displayed in Figure~\ref{fig:random_lasso} and Figure~\ref{fig:random_fused_lasso}, respectively. The top row shows the objective gap versus the iteration count, whereas the bottom row shows the relative KKT residual $\eta$ versus the iteration count. In each figure, the dimensions of the input data is included at the top. Meanwhile, the runtime of each algorithm, along with the number of iterations and the KKT residual $\eta$, are detailed in Table~\ref{tab:random_lasso} and Table~\ref{tab:random_fused_lasso}. In the tables, the fastest and second-fastest runtimes are indicated in bold and
underlined, respectively.

\begin{figure}[h]
    \centering
    \includegraphics[width=0.99\linewidth]{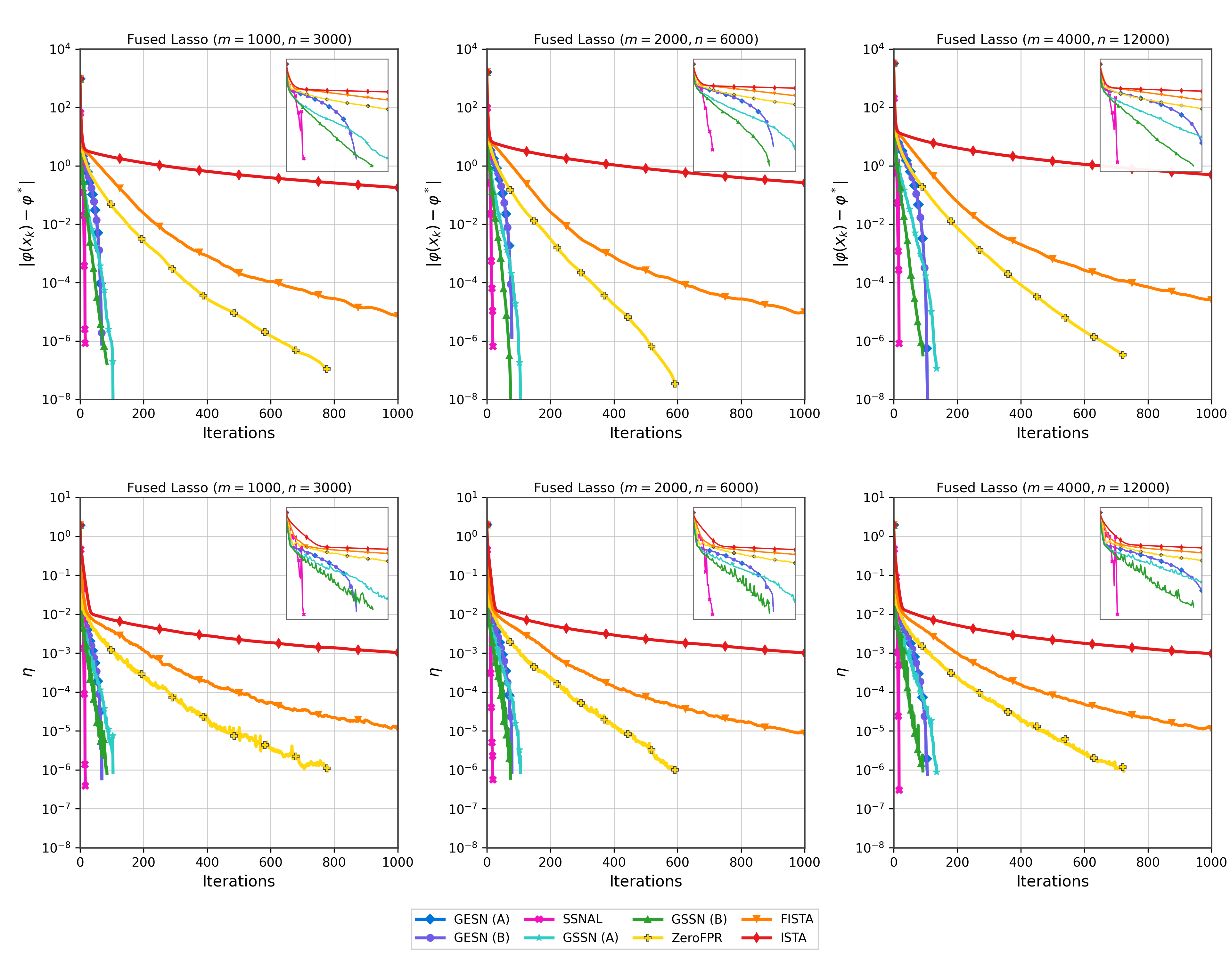}
    \caption{Plots for fused Lasso problem}
    \label{fig:random_fused_lasso}
\end{figure}

\begin{table}
\centering
\footnotesize
\setlength{\tabcolsep}{3.2pt}
\renewcommand{\arraystretch}{1.12}
\setlength{\arrayrulewidth}{0.4pt}

\begin{adjustbox}{max width=\textwidth}
\begin{tabular}{|l|c|c|c|c|c|c|c|c|c|}
\hline
\multirow{2}{*}{\textbf{Solver}}
& \multicolumn{3}{c|}{\textbf{Runtime (s)}}
& \multicolumn{3}{c|}{\textbf{Iterations}}
& \multicolumn{3}{c|}{\textbf{Residual $\eta$}}
\\
\cline{2-10}
&
$1000{\times}3000$
& $2000{\times}6000$
& $4000{\times}12000$
&
$1000{\times}3000$
& $2000{\times}6000$
& $4000{\times}12000$
&
$1000{\times}3000$
& $2000{\times}6000$
& $4000{\times}12000$
\\
\hline

GESN (A)
& \underline{1.10} & \textbf{5.77} & \underline{34.98}
& 36 & 45 & 56
& $3.02{\times}10^{-7}$
& $1.14{\times}10^{-9}$
& $4.81{\times}10^{-10}$
\\ \hline

GESN (B)
& \textbf{1.08} & \underline{5.89} & \textbf{34.91}
& 36 & 45 & 56
& $3.02{\times}10^{-7}$
& $1.14{\times}10^{-9}$
& $4.81{\times}10^{-10}$
\\ \hline

SSNAL
& 3.22 & 11.97 & 68.18
& 19 & 19 & 19
& $1.26{\times}10^{-7}$
& $2.12{\times}10^{-7}$
& $5.77{\times}10^{-7}$
\\ \hline

GSSN (A)
& 2.98 & 14.00 & 79.43
& 77 & 72 & 91
& $9.47{\times}10^{-7}$
& $4.70{\times}10^{-7}$
& $7.99{\times}10^{-7}$
\\ \hline

GSSN (B)
& 1.44 & 6.56 & 35.27
& 43 & 47 & 58
& $2.29{\times}10^{-7}$
& $9.99{\times}10^{-7}$
& $7.47{\times}10^{-7}$
\\ \hline

ZeroFPR
& 2.25 & 11.22 & 46.33
& 368 & 430 & 416
& $9.36{\times}10^{-7}$
& $9.34{\times}10^{-7}$
& $9.87{\times}10^{-7}$
\\ \hline

FISTA
& 5.11 & 31.30 & 130.80
& 1,821 & 1,913 & 1,926
& $9.99{\times}10^{-7}$
& $9.99{\times}10^{-7}$
& $1.00{\times}10^{-6}$
\\ \hline

ISTA
& 11.87 & 81.22 & 350.06
& 10,000 & 10,000 & 10,000
& $1.50{\times}10^{-6}$
& $3.38{\times}10^{-6}$
& $3.07{\times}10^{-6}$
\\ \hline

\end{tabular}
\end{adjustbox}
\caption{Three-run mean computational results for the Lasso problems.}
\label{tab:random_lasso}
\end{table}

Figures~\ref{fig:random_lasso} and~\ref{fig:random_fused_lasso} show that the proposed GESN variants reduce both the objective gap and the relative KKT residual rapidly for all problem. In particular, their convergence curves exhibit a quick downward trajectory, consistent with the fast local convergence result. Moreover, it can be observed from Table~\ref{tab:random_lasso} and Table~\ref{tab:random_fused_lasso} that the number of outer iterations increases only moderately as $(m,n)$ grows: GESN requires between $36$ and $56$ iterations for the Lasso problems and between $69$ and $106$ iterations for the fused-Lasso problems. By contrast, ZeroFPR and FISTA require several hundred and several thousand iterations, respectively, while ISTA reaches the maximum of $10{,}000$ iterations without satisfying the prescribed tolerance. The trajectories produced by GESN(A) and GESN(B) are nearly indistinguishable, indicating that the two implementations preserve essentially the same convergence behavior.

For the Lasso problems, Table~\ref{tab:random_lasso} shows that the two
variants of GESN attain the shortest and second-shortest runtimes for all
three dimensions. The best GESN runtimes are $1.08$, $5.77$, and $34.91$
seconds, respectively, which are $2.0$--$3.0$ times faster than SSNAL,
$2.3$--$2.8$ times faster than GSSN~(A), $1.3$--$2.1$ times faster than
ZeroFPR, and $3.7$--$5.4$ times faster than FISTA. The closest competitor is
GSSN~(B), whose runtime approaches that of GESN as the dimension grows and
is within $1\%$ of it for $(m,n)=(4000,12000)$. Although SSNAL requires
fewer outer iterations, the lower cost of each GESN iteration leads to
shorter total runtimes. Moreover, on the two larger instances GESN
terminates with residuals of order $10^{-9}$ and $10^{-10}$, two to three
orders of magnitude below the prescribed tolerance.

A similar pattern is observed for the fused-Lasso problems in
Table~\ref{tab:random_fused_lasso}. Again, the two GESN variants attain the
shortest and second-shortest runtimes for all three dimensions, with best
runtimes of $2.22$, $8.40$, and $75.04$ seconds, while requiring at most
$106$ outer iterations. GESN is $1.9$--$2.9$ times faster than SSNAL,
$1.4$--$2.6$ times faster than GSSN~(B), and $4.8$--$8.4$ times faster than
GSSN~(A). The margin over ZeroFPR decreases with the dimension, from a
factor of $2.8$ for $(m,n)=(1000,3000)$ to about $5\%$ for
$(m,n)=(4000,12000)$. ZeroFPR requires between $592$ and $779$ iterations,
which indicates a much lower cost per iteration, whereas each GESN
iteration achieves a substantially larger reduction of the residual.

In both tables, all methods except ISTA
attain the prescribed residual tolerance in \eqref{eq:eta_residual}. ISTA terminates at the iteration
limit with residuals of order $10^{-6}$ for the Lasso problems and
$10^{-5}$ for the fused-Lasso problems.

\begin{table}[H]
\centering
\footnotesize
\setlength{\tabcolsep}{3.2pt}
\renewcommand{\arraystretch}{1.12}
\setlength{\arrayrulewidth}{0.4pt}

\begin{adjustbox}{max width=\textwidth}
\begin{tabular}{|l|c|c|c|c|c|c|c|c|c|}
\hline
\multirow{2}{*}{\textbf{Solver}}
& \multicolumn{3}{c|}{\textbf{Runtime (s)}}
& \multicolumn{3}{c|}{\textbf{Iterations}}
& \multicolumn{3}{c|}{\textbf{Residual $\eta$}}
\\
\cline{2-10}
&
$1000{\times}3000$
& $2000{\times}6000$
& $4000{\times}12000$
&
$1000{\times}3000$
& $2000{\times}6000$
& $4000{\times}12000$
&
$1000{\times}3000$
& $2000{\times}6000$
& $4000{\times}12000$
\\
\hline

GESN (A)
& \textbf{2.22} & \underline{8.81} & \underline{77.03}
& 69 & 79 & 106
& $5.91{\times}10^{-7}$
& $8.66{\times}10^{-7}$
& $7.48{\times}10^{-7}$
\\ \hline

GESN (B)
& \underline{2.24} & \textbf{8.40} & \textbf{75.04}
& 69 & 79 & 106
& $5.91{\times}10^{-7}$
& $8.66{\times}10^{-7}$
& $7.48{\times}10^{-7}$
\\ \hline

SSNAL
& 5.24 & 24.24 & 144.30
& 17 & 19 & 17
& $3.86{\times}10^{-7}$
& $5.50{\times}10^{-7}$
& $3.03{\times}10^{-7}$
\\ \hline

GSSN (A)
& 13.20 & 70.71 & 363.44
& 104 & 106 & 135
& $8.54{\times}10^{-7}$
& $8.50{\times}10^{-7}$
& $8.66{\times}10^{-7}$
\\ \hline

GSSN (B)
& 4.82 & 21.86 & 108.00
& 85 & 75 & 92
& $8.09{\times}10^{-7}$
& $6.07{\times}10^{-7}$
& $9.34{\times}10^{-7}$
\\ \hline

ZeroFPR
& 6.11 & 16.77 & 78.97
& 779 & 592 & 724
& $9.35{\times}10^{-7}$
& $9.90{\times}10^{-7}$
& $9.25{\times}10^{-7}$
\\ \hline

FISTA
& 9.87 & 39.66 & 167.80
& 2,357 & 2,241 & 2,356
& $9.96{\times}10^{-7}$
& $9.99{\times}10^{-7}$
& $1.00{\times}10^{-6}$
\\ \hline

ISTA
& 28.21 & 94.32 & 362.57
& 10,000 & 10,000 & 10,000
& $2.46{\times}10^{-5}$
& $1.95{\times}10^{-5}$
& $2.19{\times}10^{-5}$
\\ \hline

\end{tabular}
\end{adjustbox}
\caption{Three-run mean computational results for the fused-Lasso problems.}
\label{tab:random_fused_lasso}
\end{table}

\subsection{Real datasets performance profile regression experiments}
\label{sec:real_data_experiments}

In this subsection, we illustrate the performace of the proposed algorithm on ten medium-scale biomedical datasets obtained from Kent Ridge Biomedical Data Set Repository \cite{LiLiu2002KentRidge} and ten large-scale datasets from the LIBSVM collection \cite{chang2011libsvm}.

\begin{table}[ht]
\centering
\footnotesize
\setlength{\tabcolsep}{3.5pt}
\renewcommand{\arraystretch}{0.95}

\begin{minipage}[t]{0.49\textwidth}
\vspace{0pt}
\centering
\caption{Statistics of the biomedical test instances.}
\label{tab:biomedical_datasets}
\vspace{2pt}
\begin{tabular}{|c|r|r|c|}
\hline
probname & $m$ & $n$ & $\lambda_{\max}(AA^*)$ \\ \hline
DLBCLH        & 160 & 7399   & $8.38\mathrm{e}{+02}$ \\ \hline
DLBCLN        & 160 & 7399   & $8.38\mathrm{e}{+02}$ \\ \hline
DLBCLS        & 47  & 4026   & $6.23\mathrm{e}{+02}$ \\ \hline
lungH1        & 203 & 12600  & $6.64\mathrm{e}{+03}$ \\ \hline
lungH2        & 149 & 12533  & $6.96\mathrm{e}{+03}$ \\ \hline
lungM         & 96  & 7129   & $4.46\mathrm{e}{+03}$ \\ \hline
lungO         & 39  & 2880   & $2.30\mathrm{e}{+03}$ \\ \hline
NervousSystem & 60  & 7129   & $3.34\mathrm{e}{+03}$ \\ \hline
ovarianP      & 253 & 15153  & $1.30\mathrm{e}{+04}$ \\ \hline
ovarianS      & 216 & 373401 & $2.90\mathrm{e}{+05}$ \\ \hline
\end{tabular}
\end{minipage}
\hfill
\begin{minipage}[t]{0.49\textwidth}
\vspace{0pt}
\centering
\caption{Statistics of the UCI test instances.}
\label{tab:libsvm_datasets}
\vspace{2pt}
\begin{tabular}{|c|r|r|c|}
\hline
probname & $m$ & $n$ & $\lambda_{\max}(AA^*)$ \\ \hline
E2006.train      & 16087 & 150360  & $1.91\mathrm{e}{+05}$ \\ \hline
E2006.test       & 3308  & 150358  & $4.79\mathrm{e}{+04}$ \\ \hline
log1p.E2006.test & 3308  & 4272226 & $1.46\mathrm{e}{+07}$ \\ \hline
pyrim5           & 74    & 201376  & $1.22\mathrm{e}{+06}$ \\ \hline
triazines4       & 186   & 635376  & $2.07\mathrm{e}{+07}$ \\ \hline
abalone7         & 4177  & 6435    & $5.21\mathrm{e}{+05}$ \\ \hline
bodyfat7         & 252   & 116280  & $5.29\mathrm{e}{+04}$ \\ \hline
housing7         & 506   & 77520   & $3.28\mathrm{e}{+05}$ \\ \hline
mpg7             & 392   & 3432    & $1.28\mathrm{e}{+04}$ \\ \hline
space\_ga9       & 3107  & 5005    & $4.01\mathrm{e}{+03}$ \\ \hline
\end{tabular}
\end{minipage}
\end{table}

For the biomedical datasets, we extract 10 instances, namely \texttt{DLBCLH}, \texttt{DLBCLN},
\texttt{DLBCLS}, \texttt{lungH1}, \texttt{lungH2}, \texttt{lungM},
\texttt{lungO}, \texttt{NervousSystem}, \texttt{ovarianP}, and
\texttt{ovarianS}. During the data preprocessing stage, we normalize the columns of the matrix $A$ to have norm at most one for these datasets. For the LIBSVM collection, we also extract 10 instances, namely \texttt{E2006.train}, \texttt{E2006.test}, \texttt{log1p.E2006.test}, \texttt{pyrim}, \texttt{triazines}, \texttt{abalone}, \texttt{bodyfat}, \texttt{housing},
\texttt{mpg}, and \texttt{space\_ga}. Moreover, for the last seven of the aforementioned datasets, their original features are expanded by polynomial basis
functions, as suggested in \cite{LST18}. For example, the last digits in \texttt{space\_ga9} indicates that an order 9
polynomial is used to generate the basis functions. This naming convention is also adopted for the rest
of the expanded data sets. The dataset dimensions along with the largest eigenvalues of $AA^\top$ is summarized in Table~\ref{tab:biomedical_datasets} and Table~\ref{tab:libsvm_datasets}.

To further illustrate the performance of the proposed algorithm, we use the performance profile
in \cite{DolanMore2002}. Specifically, let $P$ denote the set of test
problems for one model and $S$ the set of solvers. For $p\in P$ and
$s\in S$, let $T_{p,s}$ be the total runtime of solver $s$ on problem $p$. If a method fails to attain \eqref{eq:eta_residual} before the prescribed time limit, we set
$T_{p,s}=+\infty$. Specifically, for a method $s$, we plot the graph of the function
\begin{equation}
    \rho_s(\tau)
    :=\frac{1}{|P|}
    \left|\left\{p\in\mathcal{P}\Bigg|\,   \frac{T_{p,s}}{\min_{s'\in\mathcal{S}}T_{p,s'}}\le\tau\right\}\right|
    \quad \text{for } \tau\ge1,
    \label{eq:num_performance_profile}
\end{equation}
where $|P|$ is the size of $P$. Thus, $\rho_s(\tau)$ is the fraction of test problems for which method $s$
requires at most $\tau$ times the runtime of the fastest method. For example, $\rho_s(1)$ is the fraction of problems on which the method $s$ performs the best. 

The results in the previous subsection suggest that the second-order methods
clearly outperform the first-order ones. We therefore compare the performance of our proposed
\textbf{Algorithm~\ref{alg:Mod_regularized_ESN}} against only methods of second-order type. The test settings are
summarized in Table~\ref{tab:benchmark_setup_realdata}. For the Lasso problem, we choose the regularization hyperparameter $\lambda = a\|A^\top b\|_\infty$ with $a \in \{10^{-2},10^{-3},10^{-4}\}$. For the fused Lasso problem, we set $\lambda_1 = \alpha_1\|A^\top b\|_\infty$ and
$\lambda_2 = \alpha_2\lambda_1$, where $\alpha_1 \in \{10^{-3},10^{-4}\}$ and
$\alpha_2 \in \{10,2,0.02,0.01\}$ for the biomedical datasets and
$\alpha_2 \in \{1,0.5,10^{-1},10^{-2}\}$ for the UCI/LIBSVM datasets. In
total, each benchmark contains 110 test instances. The same parameter grid is
used for every dataset in a benchmark, and every solver runs with a single
fixed configuration as describe in the previous subsection. The maximum prescribed time limit for every tested methods in this experiment is $1800$ seconds.

\begin{table}[H]
\centering
\resizebox{\textwidth}{!}{%
\begin{tabular}{llllrlr}
\toprule
&&&
\multicolumn{2}{c}{Biomedical}
&
\multicolumn{2}{c}{UCI/LIBSVM}
\\
\cmidrule(lr){4-5}\cmidrule(lr){6-7}
Model
& Methods compared with GESN
& Parameterization
& Hyperparameter grid
& No.
& Hyperparameter grid
& No.
\\
\midrule

Lasso
& SSNAL~\cite{LST18},
  GSSN~\cite{GfrererCOA25},
  ZeroFPR~\cite{T18}
& $\lambda=a\|A^\top b\|_\infty$
& $a\in\{10^{-2},10^{-3},10^{-4}\}$
& 30
& $a\in\{10^{-2},10^{-3},10^{-4}\}$
& 30
\\[1mm]

Fused Lasso
& SSNAL~\cite{LST18Fused},
  GSSN~\cite{GfrererCOA25},
  ZeroFPR~\cite{T18}
& $\lambda_1=\alpha_1\|A^\top b\|_\infty,\;
   \lambda_2=\alpha_2\lambda_1$
& $\begin{aligned}
   \alpha_1&\in\{10^{-3},10^{-4}\},\\
   \alpha_2&\in\{10,2,0.02,0.01\}
   \end{aligned}$
& 80
& $\begin{aligned}
   \alpha_1&\in\{10^{-3},10^{-4}\},\\
   \alpha_2&\in\{1,0.5,10^{-1},10^{-2}\}
   \end{aligned}$
& 80
\\[1mm]

\midrule
\multicolumn{3}{r}{Total}
& \multicolumn{1}{c}{}
& 110
& \multicolumn{1}{c}{}
& 110
\\
\bottomrule
\end{tabular}%
}
\caption{Methods and hyperparameter settings used in the biomedical and UCI/LIBSVM benchmarks.}
\label{tab:benchmark_setup_realdata}
\end{table}

Figures~\ref{fig:biomedical_profiles} and~\ref{fig:libsvm_profiles} report
the performance profiles of the tested methods on the biomedical datasets
(Table~\ref{tab:biomedical_datasets}) and the UCI/LIBSVM datasets
(Table~\ref{tab:libsvm_datasets}).

\begin{figure}[H]
    \centering
    \includegraphics[width=0.94\linewidth]{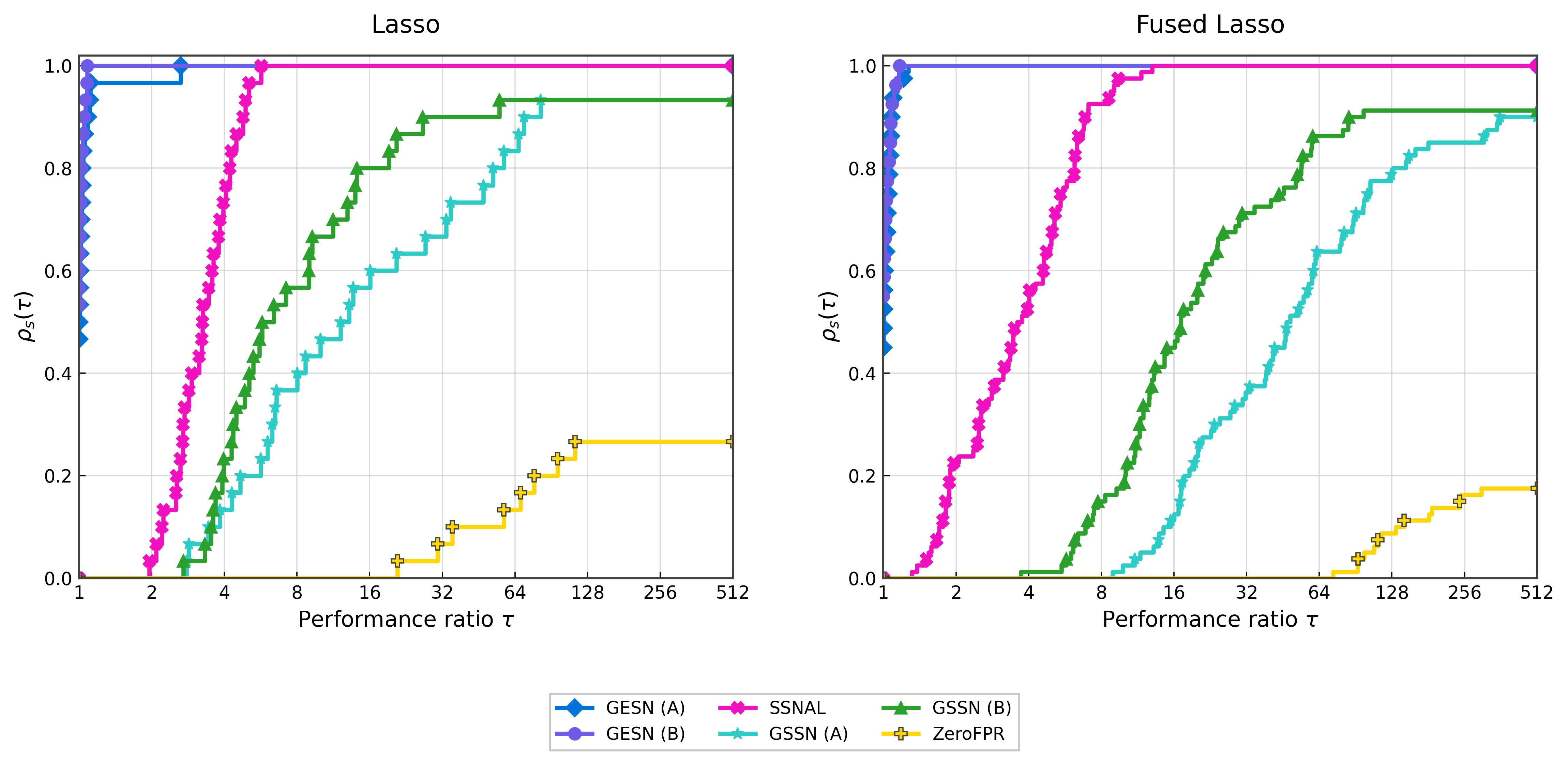}
    \caption{Performance profile for biomedical datasets}
    \label{fig:biomedical_profiles}
\end{figure}

\begin{figure}[H]
    \centering    \includegraphics[width=0.94\linewidth]{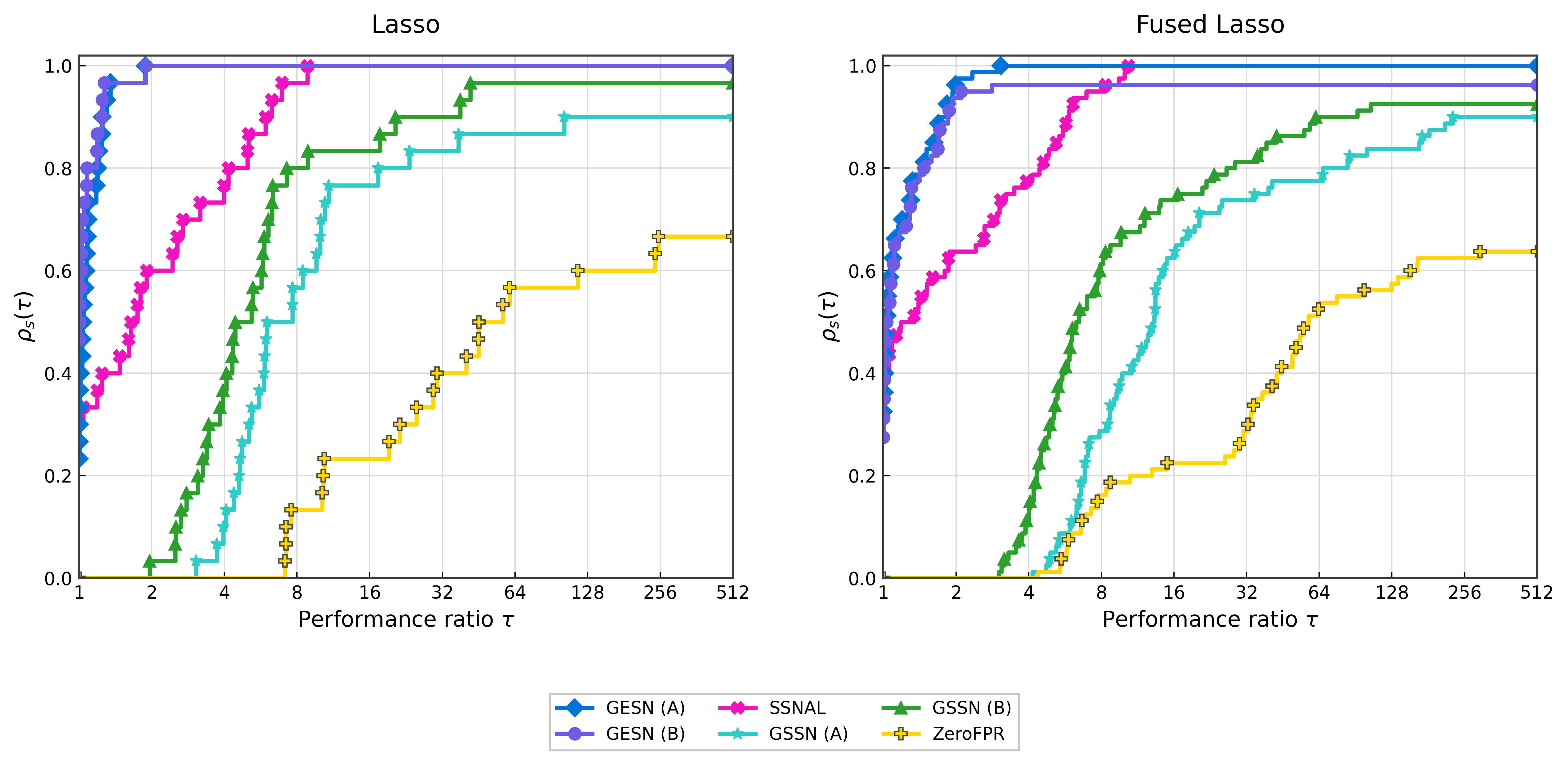}
    \caption{Performance profile for LIBSVM datasets}
    \label{fig:libsvm_profiles}
\end{figure}

On the biomedical datasets the two GESN variants achieve the best runtimes for both regularizers. For the Lasso problems, GESN(A) and GESN(B) attain the shortest runtime on $47\%$ and $53\%$ of the instances, respectively. GESN(B) stays within a factor of $1.08$ of the best runtime on all 30 instances, and GESN(A) is within a factor of two on all but one. For the fused-Lasso problems, both variants solve all 80 tested problems, each within a factor of $1.27$ of the best runtime. SSNAL also solves all tested instances; it is within a factor of four of the best runtime on $73\%$ of the Lasso and $54\%$ of the fused-Lasso instances. For GSSN, the B variant has consistently better performance than the A one, it solves 28 of the Lasso instances and 73 of the fused-Lasso instances. ZeroFPR reaches the prescribed tolerance on 8 Lasso and 14 fused-Lasso instances.

For the run results on LIBSVM instances, SSNAL attains the shortest runtime on $30\%$ of the Lasso and $40\%$ of the fused-Lasso instances whereas GESN(A) or (B) attains it on the remaining ones. For the Lasso problems, both GESN variants solve every instance within a factor of two of the best runtime. SSNAL also solves every instance, and on $60\%$ of them it is within a factor of two of the best. For the fused-Lasso problems, GESN(A) solves all 80 instances and is within a factor of two of the best runtime on $96\%$ of them while SSNAL solves all instances and is within a factor of two on $64\%$. GESN(B) behaves similarly to GESN(A), except on three instances where an inner solve requires more than its prescribed 800 CG iterations. GSSN(A), GSSN(B) and ZeroFPR solve most instances: 27, 29 and 20 of the 30 Lasso instances, and 72, 74 and 51 of the 80 fused-Lasso instances. Overall, the results indicate that GESN achieves runtimes close to the best on most instances.

For detailed numerical results of each algorithm on every dataset and hyperparameter setting, we refer to Tables~\ref{tab:biomedical_lasso_detail} and~\ref{tab:biomedical_fused_lasso_detail} (biomedical datasets) and Tables~\ref{tab:libsvm_lasso_detail} and~\ref{tab:libsvm_fused_lasso_detail} (UCI/LIBSVM datasets) in the Appendix.

\section{Conclusion}\label{sec:conclude}

In this paper, we proposed GESN, a globally convergent inexact Newton
method for convex composite optimization with polyhedral regularizers. The
method combines a proximal-gradient step with a second-order direction belongs to the effective set, which can be computed approximately under inexact
conditions that ensure descent. Unlike the recently developed
effective-subspace Newton method~\cite{NVV25}, which is a special instance of the coderivative-based Newton method \cite[Algorithm~1 and Algorithm~2]{KhanhMordukhovichPhat25} and the SCD semismooth$^*$ method {\rm \cite[Algorithm~4.2]{Gfrerer25}} when the regularizer is
the indicator or support function of a polyhedral set \cite[Corollary 3.7]{NVV25}, GESN is not a
coderivative-based method, as illustrated in Example~\ref{ex:not_coderivative_baed}. 

Another
distinctive feature of GESN is that its globalization does not rely on the
forward--backward envelope, which sets it apart from many other globally
convergent Newton-type methods for composite problems; see,
e.g.,~\cite{P13,P14,S17,T18,KhanhMordukhovichPhat24,KhanhMordukhovichPhat25,GfrererCOA25,Gfrerer25}.
Instead, its line search is a classical Armijo rule that requires only
evaluations of the smooth part, while still guaranteeing descent of the
composite objective. We established global convergence and, near
tilt-stable minimizers, eventual acceptance of the unit step and local
quadratic convergence. Numerical experiments on Lasso and fused-Lasso
problems with synthetic and real data show that GESN is competitive with,
and often faster than, state-of-the-art first-order and second-order
methods.

Several directions remain open. First, extending the convergence analysis
beyond tilt stability would broaden the applicability of the theory. Second,
the current global convergence result assumes that the gradient of the smooth
component is globally Lipschitz continuous. Relaxing this assumption to local
Lipschitz continuity by incorporating adaptive stepsize mechanisms, such as
those developed in~\cite{latafat2025}, is a natural next step. This would
require a new Lyapunov-type analysis to guarantee global convergence of an adaptive variant of GESN that does not need the Lipschitz constant. Third, it would be
of interest to extend the effective-set framework beyond polyhedral
regularizers, for example to group Lasso or spectral regularizers, whose
nonpolyhedral structure makes the local convergence analysis more delicate.

\section*{Acknowledgments}
We are deeply grateful to the author of \cite{GfrererCOA25} for kindly sharing the code implementation of their paper during the development of this paper. {We are also thankful to both anonymous referees for their careful readings, constructive suggestions, as well as deep remarks on other works that helped improve the original presentation significantly.}

\section{Appendix}

In this appendix, we present additional results for the experiment conducted in Section~\ref{sec:real_data_experiments}. The configuration set up for each tested algorithms is the following:
\begin{itemize}
\item SSNAL: The released code \footnote{https://github.com/MatOpt/SuiteLasso} of \cite{LST18,LST18Fused} with its default
settings and the Lipschitz constant $L=\|A\|^2$. The initial penalty parameter is
$\sigma_0=\max\{1/\sqrt{L},\min\{1,1/L,\lambda\}\}$ for Lasso and
$\sigma_0=\min\{1,1/\sqrt{L}\}$ for fused Lasso, and it is adjusted after each outer
iteration by the released rule within $[10^{-4},10^{7}]$ and $[10^{-4},10^{8}]$,
respectively. Each augmented Lagrangian subproblem is solved by the semismooth Newton method
with at most $10$ to $30$ iterations, depending on the progress of the outer iteration, and
the problem is rescaled during the iterations. For fused Lasso, the Newton systems are solved
with the released preconditioner with the set up precond$\,=2$.
\item GSSN(A): The released code \footnote{https://github.com/HGfrerer/GSSN\_Lasso} of \cite{GfrererCOA25}, with the Newton subproblem solved by the
trust-region CG method of \cite[Algorithm~2]{GfrererCOA25} with the stopping parameter
$\chi(t)=0.1/(1-\log(t/(t+1)))$, see \cite[Section 6]{GfrererCOA25}, or $100000$ CG iterations, and without the direct solver.
The line-search parameters are $\alpha=0.8$, $\beta=0.2$, $\sigma=0.1$, $\tau=0.25$.
\item GSSN(B): The same source code in GSSN(A) with its default settings: the
Newton subproblem is solved directly if the reduced system has dimension at most $200$, and
otherwise by preconditioned CG with relative tolerance $0.1$ and at most $30$ iterations;
$\alpha=0.8$, $\beta=0.2$, $\sigma=0.5$, $\tau=0.25$.
\item ZeroFPR: The implementation of \cite{T18} in the ForBES toolbox \footnote{https://github.com/kul-optec/ForBES} with its default
settings. The search directions are computed by L-BFGS with memory $10$, and the step along
them is chosen by a monotone backtracking line search on the forward--backward envelope with
the sufficient-decrease constant $\beta/(4\gamma)$, where $\beta=0.05$. The proximal-gradient
stepsize is fixed to $\gamma=(1-\beta)/L=0.95/L$, where $L=\lambda_{\max}(A^\top A)$ is
computed by ForBES at the start of the run.
\end{itemize}
\input{appendix_detail_tables}

\small

\bibliographystyle{abbrv}
\bibliography{bio}
\end{document}

%% file: appendix_detail_tables.tex

\begingroup
\scriptsize
\setlength{\tabcolsep}{1.5pt}
\setlength{\aboverulesep}{0pt}\setlength{\belowrulesep}{0pt}
\renewcommand{\arraystretch}{1.0}
\setlength{\LTcapwidth}{\linewidth}
\begin{longtable}{@{}ll*{6}{c}|*{6}{c}@{}}
\caption{Results of the tested methods on the Lasso problems with the biomedical datasets, where $\lambda=a\|A^\top b\|_\infty$. Each run is limited to 600 seconds and 10000 iterations. Red entries did not reach $\eta\le10^{-6}$: a run that reached the iteration limit within the time limit is reported with its final $\eta$ and runtime, and a run that reached the time limit is reported as ``--'' and ``$>$600''. The fastest successful time in each row is in bold. In the $\eta$ columns, $x\pm y$ stands for $x\times10^{\pm y}$.}\label{tab:biomedical_lasso_detail}\\
\toprule
 &  & \multicolumn{6}{c|}{$\eta$} & \multicolumn{6}{c}{Time (s)} \\
\cmidrule(r){3-8}\cmidrule(l){9-14}
Problem & $a$ & GESN(A) & GESN(B) & GSSN(A) & GSSN(B) & SSNAL & ZeroFPR & GESN(A) & GESN(B) & GSSN(A) & GSSN(B) & SSNAL & ZeroFPR \\
\midrule
\endfirsthead
\multicolumn{14}{@{}l}{\normalsize\textit{Table~\thetable\ (continued)}}\\[2pt]
\toprule
 &  & \multicolumn{6}{c|}{$\eta$} & \multicolumn{6}{c}{Time (s)} \\
\cmidrule(r){3-8}\cmidrule(l){9-14}
Problem & $a$ & GESN(A) & GESN(B) & GSSN(A) & GSSN(B) & SSNAL & ZeroFPR & GESN(A) & GESN(B) & GSSN(A) & GSSN(B) & SSNAL & ZeroFPR \\
\midrule
\endhead
\multicolumn{14}{r@{}}{\textit{continued on next page}}\\
\endfoot
\bottomrule
\endlastfoot
\texttt{DLBCLH} & $10^{-2}$ & 8.5-7 & 8.5-7 & 7.6-8 & 7.7-8 & 8.4-7 & 9.6-7 & 0.23 & \textbf{0.09} & 0.50 & 0.45 & 0.50 & 6.79 \\*
 & $10^{-3}$ & 6.7-9 & 6.7-9 & 4.9-8 & 2.9-7 & 3.2-7 & \textcolor{red}{1.7-4} & \textbf{0.17} & \textbf{0.17} & 1.45 & 0.66 & 0.58 & \textcolor{red}{10.7} \\*
 & $10^{-4}$ & 9.6-7 & 9.6-7 & 9.8-7 & 5.3-7 & 3.0-7 & \textcolor{red}{4.4-4} & \textbf{0.13} & \textbf{0.13} & 7.36 & 1.15 & 0.49 & \textcolor{red}{10.2} \\
\midrule
\texttt{DLBCLN} & $10^{-2}$ & 4.6-7 & 4.6-7 & 2.8-8 & 2.5-7 & 3.0-7 & 9.8-7 & \textbf{0.20} & \textbf{0.20} & 0.94 & 0.67 & 0.66 & 23.0 \\*
 & $10^{-3}$ & 1.4-8 & 1.4-8 & 3.9-7 & 7.0-7 & 5.9-7 & \textcolor{red}{2.1-5} & \textbf{0.35} & \textbf{0.35} & 4.20 & 1.27 & 0.72 & \textcolor{red}{44.3} \\*
 & $10^{-4}$ & 3.6-7 & 3.6-7 & 4.5-7 & 9.4-7 & 4.5-7 & \textcolor{red}{7.6-5} & \textbf{0.29} & \textbf{0.29} & 19.9 & 3.69 & 0.77 & \textcolor{red}{19.5} \\
\midrule
\texttt{DLBCLS} & $10^{-2}$ & 9.8-7 & 9.8-7 & 3.6-8 & 1.1-8 & 5.7-7 & 8.5-7 & \textbf{0.11} & \textbf{0.11} & 0.32 & 0.39 & 0.45 & 2.33 \\*
 & $10^{-3}$ & 1.5-7 & 1.5-7 & 3.4-7 & 4.1-7 & 9.1-7 & \textcolor{red}{3.4-6} & \textbf{0.11} & \textbf{0.11} & 0.70 & 0.54 & 0.47 & \textcolor{red}{5.44} \\*
 & $10^{-4}$ & 6.3-7 & 6.3-7 & 9.3-7 & 9.4-7 & 7.9-7 & \textcolor{red}{1.1-4} & \textbf{0.10} & \textbf{0.10} & 1.59 & 0.57 & 0.47 & \textcolor{red}{5.35} \\
\midrule
\texttt{lungH1} & $10^{-2}$ & 1.3-8 & 1.3-8 & 1.2-7 & 2.4-8 & 9.9-7 & 8.4-7 & \textbf{0.24} & 0.25 & 1.46 & 1.28 & 0.86 & 23.3 \\*
 & $10^{-3}$ & 9.0-7 & 9.0-7 & 1.6-8 & 8.6-10 & 2.9-7 & \textcolor{red}{3.2-4} & \textbf{0.47} & 0.48 & 4.65 & 4.17 & 1.17 & \textcolor{red}{47.6} \\*
 & $10^{-4}$ & 9.0-8 & 9.0-8 & 8.6-7 & 9.1-7 & 5.8-7 & \textcolor{red}{3.4-4} & 0.79 & \textbf{0.73} & 34.6 & 14.1 & 1.43 & \textcolor{red}{54.4} \\
\midrule
\texttt{lungH2} & $10^{-2}$ & 2.1-7 & 2.1-7 & 8.0-7 & 6.6-7 & 1.9-7 & 7.7-7 & 0.20 & \textbf{0.19} & 0.67 & 0.83 & 0.63 & 13.1 \\*
 & $10^{-3}$ & 1.3-7 & 1.3-7 & 1.8-7 & 6.1-7 & 7.6-7 & \textcolor{red}{2.0-4} & \textbf{0.26} & 0.28 & 3.37 & 3.62 & 0.97 & \textcolor{red}{28.6} \\*
 & $10^{-4}$ & 6.7-7 & 6.7-7 & 9.0-7 & 8.7-7 & 8.1-7 & \textcolor{red}{1.5-4} & 0.33 & \textbf{0.30} & 24.7 & 16.7 & 1.53 & \textcolor{red}{26.5} \\
\midrule
\texttt{lungM} & $10^{-2}$ & 1.4-11 & 1.4-11 & 4.9-8 & 9.0-7 & 4.3-7 & 9.9-7 & \textbf{0.17} & \textbf{0.17} & 0.46 & 0.45 & 0.54 & 5.86 \\*
 & $10^{-3}$ & 3.6-8 & 3.6-8 & 3.4-8 & 2.4-10 & 1.6-7 & \textcolor{red}{3.6-4} & \textbf{0.22} & 0.23 & 1.77 & 1.23 & 0.59 & \textcolor{red}{12.1} \\*
 & $10^{-4}$ & 8.2-7 & 8.2-7 & 9.7-7 & 1.0-6 & 7.1-7 & \textcolor{red}{1.3-4} & \textbf{0.25} & 0.26 & 8.77 & 2.85 & 0.69 & \textcolor{red}{13.8} \\
\midrule
\texttt{lungO} & $10^{-2}$ & 5.8-7 & 5.8-7 & 7.3-7 & 3.1-12 & 5.5-7 & 8.9-7 & 0.10 & \textbf{0.09} & 0.36 & 0.37 & 0.42 & 2.87 \\*
 & $10^{-3}$ & 4.3-7 & 4.3-7 & 1.2-7 & 1.7-11 & 3.4-7 & \textcolor{red}{2.6-4} & \textbf{0.11} & \textbf{0.11} & 0.72 & 0.50 & 0.44 & \textcolor{red}{3.81} \\*
 & $10^{-4}$ & 9.9-7 & 9.9-7 & 9.1-9 & 9.4-7 & 2.6-7 & \textcolor{red}{1.1-4} & \textbf{0.12} & \textbf{0.12} & 3.20 & 0.85 & 0.42 & \textcolor{red}{3.78} \\
\midrule
\texttt{Nervous} & $10^{-2}$ & 2.9-7 & 2.9-7 & 2.0-7 & 3.5-8 & 2.5-7 & 9.7-7 & \textbf{0.09} & \textbf{0.09} & 0.39 & 0.39 & 0.44 & 5.19 \\*
\texttt{System} & $10^{-3}$ & 6.8-7 & 6.8-7 & 4.1-7 & 8.2-7 & 1.0-6 & \textcolor{red}{1.8-4} & \textbf{0.18} & 0.20 & 1.21 & 0.66 & 0.54 & \textcolor{red}{20.7} \\*
 & $10^{-4}$ & 9.4-7 & 9.4-7 & 5.3-7 & 9.2-7 & 8.5-7 & \textcolor{red}{3.0-4} & \textbf{0.18} & 0.20 & 6.07 & 1.69 & 0.58 & \textcolor{red}{20.4} \\
\midrule
\texttt{ovarianP} & $10^{-2}$ & 8.7-7 & 8.7-7 & 2.1-7 & 6.3-7 & 8.7-7 & \textcolor{red}{2.0-4} & 0.34 & \textbf{0.33} & 4.53 & 2.12 & 1.40 & \textcolor{red}{134} \\*
 & $10^{-3}$ & 4.0-7 & 4.0-7 & 9.9-7 & 7.2-7 & 2.1-7 & \textcolor{red}{7.8-4} & \textbf{0.59} & 0.60 & 12.2 & 8.24 & 1.56 & \textcolor{red}{107} \\*
 & $10^{-4}$ & 8.2-7 & 8.2-7 & 6.3-7 & 5.9-7 & 9.1-7 & \textcolor{red}{3.7-4} & 1.06 & \textbf{1.03} & 67.8 & 27.2 & 2.29 & \textcolor{red}{105} \\
\midrule
\texttt{ovarianS} & $10^{-2}$ & 1.8-8 & 1.8-8 & 5.5-7 & 2.0-8 & 3.6-7 & \textcolor{red}{--} & 5.50 & \textbf{5.44} & 283 & 112 & 15.5 & \textcolor{red}{$>$600} \\*
 & $10^{-3}$ & 3.9-8 & 3.9-8 & \textcolor{red}{--} & \textcolor{red}{--} & 8.9-7 & \textcolor{red}{--} & 8.36 & \textbf{7.94} & \textcolor{red}{$>$600} & \textcolor{red}{$>$600} & 20.2 & \textcolor{red}{$>$600} \\*
 & $10^{-4}$ & 2.5-7 & 2.5-7 & \textcolor{red}{--} & \textcolor{red}{--} & 8.5-7 & \textcolor{red}{--} & \textbf{10.5} & 10.6 & \textcolor{red}{$>$600} & \textcolor{red}{$>$600} & 23.1 & \textcolor{red}{$>$600} \\
\end{longtable}
\endgroup

\begingroup
\scriptsize
\setlength{\tabcolsep}{1.5pt}
\setlength{\aboverulesep}{0pt}\setlength{\belowrulesep}{0pt}
\renewcommand{\arraystretch}{1.0}
\setlength{\LTcapwidth}{\linewidth}
\begin{longtable}{@{}llr*{6}{c}|*{6}{c}@{}}
\caption{Same as Table~\ref{tab:biomedical_lasso_detail}, but for the fused Lasso problems with the biomedical datasets, where $\lambda_1=\alpha_1\|A^\top b\|_\infty$ and $\lambda_2=\alpha_2\lambda_1$.}\label{tab:biomedical_fused_lasso_detail}\\
\toprule
 &  &  & \multicolumn{6}{c|}{$\eta$} & \multicolumn{6}{c}{Time (s)} \\
\cmidrule(r){4-9}\cmidrule(l){10-15}
Problem & $\alpha_1$ & $\alpha_2$ & GESN(A) & GESN(B) & GSSN(A) & GSSN(B) & SSNAL & ZeroFPR & GESN(A) & GESN(B) & GSSN(A) & GSSN(B) & SSNAL & ZeroFPR \\
\midrule
\endfirsthead
\multicolumn{15}{@{}l}{\normalsize\textit{Table~\thetable\ (continued)}}\\[2pt]
\toprule
 &  &  & \multicolumn{6}{c|}{$\eta$} & \multicolumn{6}{c}{Time (s)} \\
\cmidrule(r){4-9}\cmidrule(l){10-15}
Problem & $\alpha_1$ & $\alpha_2$ & GESN(A) & GESN(B) & GSSN(A) & GSSN(B) & SSNAL & ZeroFPR & GESN(A) & GESN(B) & GSSN(A) & GSSN(B) & SSNAL & ZeroFPR \\
\midrule
\endhead
\multicolumn{15}{r@{}}{\textit{continued on next page}}\\
\endfoot
\bottomrule
\endlastfoot
\texttt{DLBCLH} & $10^{-3}$ & 10 & 8.3-9 & 8.3-9 & 3.0-7 & 1.8-9 & 3.8-7 & 9.2-7 & 0.12 & \textbf{0.10} & 1.38 & 1.70 & 0.86 & 24.4 \\*
 &  & 2 & 2.3-10 & 2.3-10 & 7.5-7 & 2.5-12 & 6.1-7 & \textcolor{red}{3.6-5} & \textbf{0.12} & \textbf{0.12} & 2.43 & 2.10 & 0.87 & \textcolor{red}{25.6} \\*
 &  & 0.02 & 9.4-7 & 9.4-7 & 3.7-8 & 9.5-7 & 3.6-7 & \textcolor{red}{1.6-4} & \textbf{0.19} & \textbf{0.19} & 2.58 & 1.86 & 0.92 & \textcolor{red}{23.1} \\*
 &  & 0.01 & 1.7-8 & 1.7-8 & 9.0-7 & 2.5-8 & 3.6-7 & \textcolor{red}{3.6-4} & 0.28 & \textbf{0.27} & 2.44 & 1.75 & 0.87 & \textcolor{red}{22.6} \\
\texttt{DLBCLH} & $10^{-4}$ & 10 & 9.8-9 & 9.8-9 & 7.6-7 & 8.4-7 & 5.7-9 & \textcolor{red}{1.6-4} & \textbf{0.14} & \textbf{0.14} & 5.40 & 4.03 & 0.95 & \textcolor{red}{24.5} \\*
 &  & 2 & 3.4-7 & 3.4-7 & 7.9-7 & 9.8-7 & 6.4-7 & \textcolor{red}{2.8-4} & 0.14 & \textbf{0.13} & 5.40 & 3.25 & 0.92 & \textcolor{red}{25.7} \\*
 &  & 0.02 & 9.0-7 & 9.0-7 & 8.6-7 & 8.3-7 & 4.8-7 & \textcolor{red}{1.4-4} & \textbf{0.25} & 0.26 & 9.42 & 2.95 & 1.00 & \textcolor{red}{23.0} \\*
 &  & 0.01 & 5.1-7 & 5.1-7 & 7.0-7 & 7.1-7 & 4.9-7 & \textcolor{red}{1.9-4} & \textbf{0.18} & \textbf{0.18} & 9.32 & 3.57 & 0.92 & \textcolor{red}{22.8} \\
\midrule
\texttt{DLBCLN} & $10^{-3}$ & 10 & 8.0-8 & 8.0-8 & 1.1-8 & 6.1-7 & 1.4-7 & 8.8-7 & \textbf{0.20} & 0.21 & 8.04 & 4.41 & 1.44 & 18.8 \\*
 &  & 2 & 1.9-8 & 1.9-8 & 2.4-7 & 1.0-8 & 3.8-7 & \textcolor{red}{2.7-6} & \textbf{0.22} & 0.23 & 10.8 & 4.71 & 1.37 & \textcolor{red}{32.4} \\*
 &  & 0.02 & 2.8-7 & 2.8-7 & 1.0-6 & 4.6-7 & 2.5-7 & \textcolor{red}{5.9-5} & \textbf{0.25} & \textbf{0.25} & 10.5 & 1.54 & 1.26 & \textcolor{red}{29.5} \\*
 &  & 0.01 & 3.3-7 & 3.3-7 & 8.8-7 & 9.8-7 & 4.9-7 & \textcolor{red}{3.4-5} & \textbf{0.35} & \textbf{0.35} & 10.9 & 4.57 & 1.21 & \textcolor{red}{30.0} \\
\texttt{DLBCLN} & $10^{-4}$ & 10 & 8.0-7 & 8.0-7 & 9.4-7 & 7.2-7 & 7.4-7 & \textcolor{red}{2.5-5} & 0.29 & \textbf{0.28} & 21.9 & 15.2 & 1.44 & \textcolor{red}{32.2} \\*
 &  & 2 & 9.6-7 & 9.6-7 & 8.8-7 & 8.9-7 & 3.6-7 & \textcolor{red}{4.3-5} & \textbf{0.25} & \textbf{0.25} & 40.2 & 12.8 & 1.44 & \textcolor{red}{32.6} \\*
 &  & 0.02 & 8.0-7 & 8.0-7 & 9.8-7 & 8.2-7 & 5.4-7 & \textcolor{red}{1.2-4} & \textbf{0.29} & 0.30 & 90.0 & 17.6 & 1.34 & \textcolor{red}{29.6} \\*
 &  & 0.01 & 8.5-7 & 8.5-7 & 8.9-7 & 6.5-7 & 6.6-7 & \textcolor{red}{5.4-5} & \textbf{0.31} & \textbf{0.31} & 31.8 & 16.5 & 1.40 & \textcolor{red}{32.8} \\
\midrule
\texttt{DLBCLS} & $10^{-3}$ & 10 & 6.6-7 & 6.6-7 & 4.6-8 & 4.8-12 & 2.6-7 & 9.4-7 & \textbf{0.06} & \textbf{0.06} & 1.02 & 0.89 & 0.80 & 5.68 \\*
 &  & 2 & 3.4-7 & 3.4-7 & 6.9-7 & 7.9-7 & 9.3-7 & 9.4-7 & \textbf{0.06} & \textbf{0.06} & 1.26 & 1.25 & 0.73 & 6.14 \\*
 &  & 0.02 & 6.1-7 & 6.1-7 & 7.5-7 & 7.6-7 & 6.3-7 & \textcolor{red}{4.7-6} & \textbf{0.13} & \textbf{0.13} & 2.13 & 1.29 & 0.78 & \textcolor{red}{17.2} \\*
 &  & 0.01 & 4.6-8 & 4.6-8 & 5.4-7 & 1.9-11 & 4.3-7 & 9.5-7 & 0.14 & \textbf{0.13} & 1.27 & 1.41 & 0.82 & 15.0 \\
\texttt{DLBCLS} & $10^{-4}$ & 10 & 1.5-7 & 1.5-7 & 4.4-7 & 9.0-7 & 3.4-7 & 8.2-7 & \textbf{0.09} & \textbf{0.09} & 2.16 & 1.64 & 0.82 & 15.9 \\*
 &  & 2 & 2.2-8 & 2.2-8 & 3.7-7 & 2.3-7 & 5.8-7 & \textcolor{red}{4.0-5} & \textbf{0.08} & \textbf{0.08} & 3.87 & 2.00 & 0.75 & \textcolor{red}{16.9} \\*
 &  & 0.02 & 7.6-7 & 7.6-7 & 7.7-7 & 9.2-7 & 6.6-7 & \textcolor{red}{5.1-5} & \textbf{0.10} & 0.11 & 4.18 & 1.79 & 0.89 & \textcolor{red}{18.8} \\*
 &  & 0.01 & 6.4-7 & 6.4-7 & 1.8-7 & 2.4-10 & 7.6-7 & \textcolor{red}{5.1-5} & 0.18 & \textbf{0.16} & 5.26 & 1.91 & 0.80 & \textcolor{red}{15.7} \\
\midrule
\texttt{lungH1} & $10^{-3}$ & 10 & 7.0-7 & 7.0-7 & 7.9-7 & 7.4-7 & 1.2-7 & 9.6-7 & \textbf{0.35} & 0.36 & 3.87 & 2.18 & 1.62 & 50.6 \\*
 &  & 2 & 9.5-7 & 9.5-7 & 7.6-8 & 8.0-7 & 5.6-7 & \textcolor{red}{9.4-5} & 0.38 & \textbf{0.37} & 6.60 & 2.76 & 1.34 & \textcolor{red}{75.8} \\*
 &  & 0.02 & 7.0-7 & 7.0-7 & 6.1-7 & 1.1-8 & 3.1-7 & \textcolor{red}{2.3-4} & 0.50 & \textbf{0.49} & 9.61 & 8.19 & 1.23 & \textcolor{red}{74.6} \\*
 &  & 0.01 & 8.9-7 & 8.9-7 & 1.9-7 & 4.9-7 & 2.2-7 & \textcolor{red}{2.9-4} & \textbf{0.59} & 0.60 & 9.45 & 7.88 & 1.07 & \textcolor{red}{73.8} \\
\texttt{lungH1} & $10^{-4}$ & 10 & 4.6-7 & 4.6-7 & 1.7-8 & 3.0-8 & 3.6-7 & \textcolor{red}{2.5-4} & 0.66 & \textbf{0.61} & 16.8 & 18.8 & 1.44 & \textcolor{red}{77.4} \\*
 &  & 2 & 6.5-7 & 6.5-7 & 4.8-7 & 9.6-7 & 9.3-8 & \textcolor{red}{3.8-4} & 0.75 & \textbf{0.73} & 44.2 & 31.7 & 1.49 & \textcolor{red}{82.7} \\*
 &  & 0.02 & 7.3-7 & 7.3-7 & 8.7-7 & 9.1-7 & 2.6-7 & \textcolor{red}{5.0-4} & 0.77 & \textbf{0.73} & 64.7 & 40.0 & 1.38 & \textcolor{red}{74.3} \\*
 &  & 0.01 & 8.2-7 & 8.2-7 & 8.6-7 & 8.6-7 & 2.6-7 & \textcolor{red}{3.9-4} & \textbf{0.91} & 0.97 & 73.7 & 47.2 & 1.51 & \textcolor{red}{81.6} \\
\midrule
\texttt{lungH2} & $10^{-3}$ & 10 & 3.6-8 & 3.6-8 & 3.6-8 & 1.2-9 & 7.2-7 & 9.5-7 & \textbf{0.20} & \textbf{0.20} & 3.35 & 2.00 & 0.78 & 14.4 \\*
 &  & 2 & 2.4-9 & 2.4-9 & 5.2-8 & 3.5-7 & 9.8-8 & 8.6-7 & \textbf{0.28} & 0.29 & 8.54 & 1.95 & 0.96 & 31.1 \\*
 &  & 0.02 & 7.6-8 & 7.6-8 & 8.1-8 & 1.0-7 & 2.9-7 & \textcolor{red}{2.0-4} & \textbf{0.34} & 0.40 & 33.4 & 8.75 & 0.84 & \textcolor{red}{51.8} \\*
 &  & 0.01 & 5.0-9 & 5.0-9 & 6.5-8 & 8.1-7 & 5.7-7 & \textcolor{red}{5.4-5} & \textbf{0.29} & 0.33 & 26.5 & 6.71 & 0.92 & \textcolor{red}{64.3} \\
\texttt{lungH2} & $10^{-4}$ & 10 & 7.1-9 & 7.1-9 & 3.7-7 & 2.9-8 & 2.1-7 & \textcolor{red}{4.9-5} & \textbf{0.38} & 0.41 & 20.4 & 4.43 & 1.11 & \textcolor{red}{54.9} \\*
 &  & 2 & 9.8-7 & 9.8-7 & 9.1-7 & 9.8-7 & 2.7-7 & \textcolor{red}{1.5-4} & 0.48 & \textbf{0.38} & 49.1 & 6.46 & 1.25 & \textcolor{red}{53.6} \\*
 &  & 0.02 & 2.4-7 & 2.4-7 & 9.8-7 & 3.0-7 & 7.1-7 & \textcolor{red}{2.2-4} & 0.36 & \textbf{0.34} & 122 & 27.3 & 0.88 & \textcolor{red}{55.4} \\*
 &  & 0.01 & 9.7-7 & 9.7-7 & 7.2-7 & 7.8-7 & 9.0-7 & \textcolor{red}{2.7-4} & \textbf{0.36} & 0.40 & 125 & 21.3 & 0.98 & \textcolor{red}{47.5} \\
\midrule
\texttt{lungM} & $10^{-3}$ & 10 & 4.0-8 & 4.0-8 & 4.1-8 & 5.3-9 & 6.3-7 & 9.7-7 & 0.12 & \textbf{0.11} & 2.27 & 0.83 & 0.60 & 14.8 \\*
 &  & 2 & 2.2-7 & 2.2-7 & 2.7-7 & 7.0-9 & 1.5-7 & 7.6-7 & \textbf{0.11} & 0.12 & 2.27 & 1.22 & 0.60 & 20.3 \\*
 &  & 0.02 & 2.3-7 & 2.3-7 & 3.9-7 & 6.9-10 & 5.1-7 & \textcolor{red}{3.8-4} & \textbf{0.48} & 0.53 & 6.29 & 1.78 & 0.63 & \textcolor{red}{25.3} \\*
 &  & 0.01 & 9.6-7 & 9.6-7 & 1.7-8 & 2.4-7 & 3.2-7 & \textcolor{red}{1.8-4} & 0.34 & \textbf{0.32} & 5.89 & 1.90 & 0.62 & \textcolor{red}{25.5} \\
\texttt{lungM} & $10^{-4}$ & 10 & 4.0-7 & 4.0-7 & 6.7-7 & 9.6-7 & 2.5-7 & \textcolor{red}{1.1-4} & 0.20 & \textbf{0.18} & 10.2 & 2.66 & 0.84 & \textcolor{red}{27.5} \\*
 &  & 2 & 1.6-8 & 1.6-8 & 5.1-7 & 7.7-7 & 3.5-7 & \textcolor{red}{2.1-4} & \textbf{0.18} & \textbf{0.18} & 10.9 & 3.93 & 0.73 & \textcolor{red}{28.3} \\*
 &  & 0.02 & 9.0-7 & 9.0-7 & 2.8-7 & 8.1-7 & 2.8-7 & \textcolor{red}{1.9-4} & 0.29 & \textbf{0.27} & 21.4 & 5.44 & 0.67 & \textcolor{red}{25.7} \\*
 &  & 0.01 & 2.0-7 & 2.0-7 & 7.6-7 & 7.3-7 & 4.2-7 & \textcolor{red}{1.9-4} & \textbf{0.40} & 0.45 & 22.0 & 5.17 & 0.71 & \textcolor{red}{25.5} \\
\midrule
\texttt{lungO} & $10^{-3}$ & 10 & 2.4-7 & 2.4-7 & 5.5-8 & 3.6-12 & 2.1-7 & 9.3-7 & \textbf{0.07} & \textbf{0.07} & 0.78 & 0.52 & 0.45 & 7.31 \\*
 &  & 2 & 3.4-8 & 3.4-8 & 5.3-8 & 5.0-7 & 2.4-8 & \textcolor{red}{4.4-5} & \textbf{0.07} & \textbf{0.07} & 1.66 & 0.65 & 0.45 & \textcolor{red}{12.2} \\*
 &  & 0.02 & 3.4-7 & 3.4-7 & 5.8-7 & 4.0-10 & 5.8-7 & \textcolor{red}{3.2-4} & 0.14 & \textbf{0.13} & 2.08 & 0.74 & 0.50 & \textcolor{red}{11.3} \\*
 &  & 0.01 & 9.9-7 & 9.9-7 & 6.6-7 & 1.8-11 & 5.7-7 & \textcolor{red}{1.3-4} & 0.11 & \textbf{0.10} & 1.72 & 0.83 & 0.50 & \textcolor{red}{11.0} \\
\texttt{lungO} & $10^{-4}$ & 10 & 9.6-7 & 9.6-7 & 1.1-9 & 6.4-7 & 7.0-7 & \textcolor{red}{1.9-4} & \textbf{0.09} & 0.10 & 4.32 & 1.20 & 0.50 & \textcolor{red}{12.0} \\*
 &  & 2 & 1.4-7 & 1.4-7 & 6.5-7 & 9.4-7 & 1.9-7 & \textcolor{red}{1.5-4} & \textbf{0.08} & \textbf{0.08} & 7.07 & 1.30 & 0.51 & \textcolor{red}{12.1} \\*
 &  & 0.02 & 5.7-7 & 5.7-7 & 1.6-7 & 2.5-7 & 5.7-7 & \textcolor{red}{1.5-4} & \textbf{0.12} & 0.14 & 5.76 & 1.58 & 0.53 & \textcolor{red}{11.3} \\*
 &  & 0.01 & 6.3-7 & 6.3-7 & 8.1-7 & 3.6-7 & 5.9-7 & \textcolor{red}{1.6-4} & 0.14 & \textbf{0.13} & 7.31 & 1.39 & 0.50 & \textcolor{red}{11.2} \\
\midrule
\texttt{Nervous} & $10^{-3}$ & 10 & 4.2-8 & 4.2-8 & 9.4-9 & 2.1-8 & 1.8-7 & \textcolor{red}{2.1-6} & 0.10 & \textbf{0.09} & 2.14 & 0.94 & 0.59 & \textcolor{red}{24.9} \\*
\texttt{System} &  & 2 & 1.8-8 & 1.8-8 & 3.3-7 & 1.5-7 & 1.7-7 & 9.8-7 & \textbf{0.10} & 0.12 & 2.85 & 1.11 & 0.62 & 24.7 \\*
 &  & 0.02 & 8.7-12 & 8.7-12 & 3.6-8 & 3.0-8 & 3.5-8 & \textcolor{red}{7.9-5} & \textbf{0.21} & \textbf{0.21} & 3.01 & 1.19 & 0.58 & \textcolor{red}{22.8} \\*
 &  & 0.01 & 1.7-7 & 1.7-7 & 4.4-7 & 6.0-9 & 6.1-8 & \textcolor{red}{8.4-5} & 0.17 & \textbf{0.16} & 2.83 & 1.12 & 0.55 & \textcolor{red}{22.5} \\
\texttt{Nervous} & $10^{-4}$ & 10 & 3.2-8 & 3.2-8 & 1.5-8 & 4.0-9 & 5.9-7 & \textcolor{red}{2.5-4} & \textbf{0.12} & 0.13 & 5.42 & 1.72 & 0.62 & \textcolor{red}{25.2} \\*
\texttt{System} &  & 2 & 8.0-7 & 8.0-7 & 5.6-7 & 9.7-7 & 6.5-7 & \textcolor{red}{1.7-4} & \textbf{0.14} & 0.15 & 8.76 & 1.82 & 0.67 & \textcolor{red}{26.1} \\*
 &  & 0.02 & 8.5-7 & 8.5-7 & 8.9-7 & 9.2-7 & 1.4-7 & \textcolor{red}{1.6-4} & 0.26 & \textbf{0.25} & 15.3 & 2.85 & 0.62 & \textcolor{red}{23.0} \\*
 &  & 0.01 & 2.3-7 & 2.3-7 & 9.6-7 & 5.0-7 & 1.5-7 & \textcolor{red}{1.6-4} & \textbf{0.23} & 0.24 & 14.3 & 2.67 & 0.57 & \textcolor{red}{22.8} \\
\midrule
\texttt{ovarianP} & $10^{-3}$ & 10 & 3.6-8 & 3.6-8 & 8.4-7 & 4.3-7 & 6.6-7 & 9.1-7 & 0.45 & \textbf{0.37} & 35.7 & 8.91 & 1.24 & 110 \\*
 &  & 2 & 5.9-7 & 5.9-7 & 3.7-7 & 8.8-7 & 1.6-7 & \textcolor{red}{4.5-5} & \textbf{0.39} & 0.41 & 40.6 & 11.6 & 1.34 & \textcolor{red}{125} \\*
 &  & 0.02 & 4.4-8 & 4.4-8 & 6.0-8 & 2.8-7 & 4.7-7 & \textcolor{red}{4.6-4} & \textbf{0.67} & 0.71 & 84.4 & 26.8 & 1.24 & \textcolor{red}{120} \\*
 &  & 0.01 & 1.6-7 & 1.6-7 & 5.6-8 & 2.7-7 & 9.6-7 & \textcolor{red}{5.7-4} & 0.88 & \textbf{0.85} & 85.9 & 29.2 & 1.31 & \textcolor{red}{118} \\
\texttt{ovarianP} & $10^{-4}$ & 10 & 6.7-8 & 6.7-8 & 9.6-7 & 7.8-7 & 1.3-7 & \textcolor{red}{4.5-4} & \textbf{0.51} & 0.52 & 160 & 42.6 & 1.46 & \textcolor{red}{121} \\*
 &  & 2 & 1.3-7 & 1.3-7 & 8.9-7 & 6.1-7 & 2.7-7 & \textcolor{red}{3.6-4} & \textbf{0.96} & 0.97 & 144 & 57.2 & 1.63 & \textcolor{red}{122} \\*
 &  & 0.02 & 4.5-7 & 4.5-7 & 9.4-7 & 8.0-7 & 2.2-7 & \textcolor{red}{4.4-4} & \textbf{1.10} & 1.12 & 199 & 93.1 & 1.98 & \textcolor{red}{117} \\*
 &  & 0.01 & 6.1-7 & 6.1-7 & 5.8-7 & 9.8-7 & 5.6-7 & \textcolor{red}{4.6-4} & 1.35 & \textbf{1.23} & 179 & 120 & 1.94 & \textcolor{red}{115} \\
\midrule
\texttt{ovarianS} & $10^{-3}$ & 10 & 1.5-9 & 1.5-9 & \textcolor{red}{--} & 9.6-9 & 1.2-7 & \textcolor{red}{--} & \textbf{8.06} & 8.30 & \textcolor{red}{$>$600} & 368 & 15.1 & \textcolor{red}{$>$600} \\*
 &  & 2 & 2.4-7 & 2.4-7 & \textcolor{red}{--} & \textcolor{red}{--} & 8.9-8 & \textcolor{red}{--} & 13.0 & \textbf{12.7} & \textcolor{red}{$>$600} & \textcolor{red}{$>$600} & 19.2 & \textcolor{red}{$>$600} \\*
 &  & 0.02 & 8.1-7 & 8.1-7 & \textcolor{red}{--} & \textcolor{red}{--} & 7.1-7 & \textcolor{red}{--} & 13.6 & \textbf{13.5} & \textcolor{red}{$>$600} & \textcolor{red}{$>$600} & 23.9 & \textcolor{red}{$>$600} \\*
 &  & 0.01 & 5.0-8 & 5.0-8 & \textcolor{red}{--} & \textcolor{red}{--} & 8.2-7 & \textcolor{red}{--} & 17.2 & \textbf{16.5} & \textcolor{red}{$>$600} & \textcolor{red}{$>$600} & 22.8 & \textcolor{red}{$>$600} \\
\texttt{ovarianS} & $10^{-4}$ & 10 & 8.8-7 & 8.8-7 & \textcolor{red}{--} & \textcolor{red}{--} & 5.7-7 & \textcolor{red}{--} & 12.3 & \textbf{11.9} & \textcolor{red}{$>$600} & \textcolor{red}{$>$600} & 29.8 & \textcolor{red}{$>$600} \\*
 &  & 2 & 2.7-7 & 2.7-7 & \textcolor{red}{--} & \textcolor{red}{--} & 2.8-7 & \textcolor{red}{--} & \textbf{15.5} & 15.7 & \textcolor{red}{$>$600} & \textcolor{red}{$>$600} & 29.2 & \textcolor{red}{$>$600} \\*
 &  & 0.02 & 2.5-7 & 2.5-7 & \textcolor{red}{--} & \textcolor{red}{--} & 8.9-7 & \textcolor{red}{--} & 15.0 & \textbf{14.2} & \textcolor{red}{$>$600} & \textcolor{red}{$>$600} & 26.7 & \textcolor{red}{$>$600} \\*
 &  & 0.01 & 5.4-7 & 5.4-7 & \textcolor{red}{--} & \textcolor{red}{--} & 7.6-7 & \textcolor{red}{--} & \textbf{15.1} & \textbf{15.1} & \textcolor{red}{$>$600} & \textcolor{red}{$>$600} & 25.4 & \textcolor{red}{$>$600} \\
\end{longtable}
\endgroup

\begingroup
\scriptsize
\setlength{\tabcolsep}{1.5pt}
\setlength{\aboverulesep}{0pt}\setlength{\belowrulesep}{0pt}
\renewcommand{\arraystretch}{1.0}
\setlength{\LTcapwidth}{\linewidth}
\begin{longtable}{@{}ll*{6}{c}|*{6}{c}@{}}
\caption{Same as Table~\ref{tab:biomedical_lasso_detail}, but for the Lasso problems with the UCI/LIBSVM datasets, where $\lambda=a\|A^\top b\|_\infty$; the time limit is 1800 seconds, so ``$>$1800'' marks a run that reached it. $^\dagger$GSSN(A) stopped with a step-size failure and returned no iterate.}\label{tab:libsvm_lasso_detail}\\
\toprule
 &  & \multicolumn{6}{c|}{$\eta$} & \multicolumn{6}{c}{Time (s)} \\
\cmidrule(r){3-8}\cmidrule(l){9-14}
Problem & $a$ & GESN(A) & GESN(B) & GSSN(A) & GSSN(B) & SSNAL & ZeroFPR & GESN(A) & GESN(B) & GSSN(A) & GSSN(B) & SSNAL & ZeroFPR \\
\midrule
\endfirsthead
\multicolumn{14}{@{}l}{\normalsize\textit{Table~\thetable\ (continued)}}\\[2pt]
\toprule
 &  & \multicolumn{6}{c|}{$\eta$} & \multicolumn{6}{c}{Time (s)} \\
\cmidrule(r){3-8}\cmidrule(l){9-14}
Problem & $a$ & GESN(A) & GESN(B) & GSSN(A) & GSSN(B) & SSNAL & ZeroFPR & GESN(A) & GESN(B) & GSSN(A) & GSSN(B) & SSNAL & ZeroFPR \\
\midrule
\endhead
\multicolumn{14}{r@{}}{\textit{continued on next page}}\\
\endfoot
\bottomrule
\endlastfoot
\texttt{E2006.} & $10^{-2}$ & 8.5-8 & 8.5-8 & 2.3-10 & 2.3-10 & 4.2-9 & 9.9-14 & 0.24 & \textbf{0.19} & 1.08 & 1.07 & 0.92 & 1.32 \\*
\texttt{train} & $10^{-3}$ & 5.1-8 & 5.1-8 & 2.3-10 & 2.3-10 & 5.0-7 & 1.8-13 & \textbf{0.20} & \textbf{0.20} & 1.09 & 1.03 & 0.98 & 1.41 \\*
 & $10^{-4}$ & 3.5-10 & 3.5-10 & 7.8-13 & 3.3-12 & 1.6-7 & 5.2-8 & 0.29 & \textbf{0.25} & 1.29 & 1.58 & 1.04 & 1.79 \\
\midrule
\texttt{E2006.test} & $10^{-2}$ & 1.6-14 & 1.6-14 & 1.2-10 & 1.2-10 & 1.7-9 & 1.6-14 & 0.10 & \textbf{0.09} & 0.52 & 0.39 & 0.55 & 0.67 \\*
 & $10^{-3}$ & 1.6-7 & 1.6-7 & 1.2-10 & 1.2-10 & 3.4-9 & 8.0-14 & \textbf{0.07} & \textbf{0.07} & 0.51 & 0.39 & 0.59 & 0.69 \\*
 & $10^{-4}$ & 2.4-10 & 2.4-10 & 2.0-13 & 1.6-12 & 3.4-7 & 4.1-9 & \textbf{0.09} & \textbf{0.09} & 0.74 & 0.55 & 0.60 & 0.90 \\
\midrule
\texttt{log1p.} & $10^{-2}$ & 8.2-7 & 8.2-7 & 8.9-12 & 7.1-10 & 7.1-7 & 6.1-7 & 2.03 & \textbf{1.90} & 14.5 & 11.5 & 5.10 & 108 \\*
\texttt{E2006.test} & $10^{-3}$ & 2.9-11 & 3.0-11 & 6.9-10 & 1.1-12 & 1.7-7 & 9.4-7 & 4.47 & \textbf{4.45} & 22.5 & 17.5 & 7.29 & 517 \\*
 & $10^{-4}$ & 1.7-11 & 3.2-11 & 6.9-7 & 9.6-7 & 1.4-7 & \textcolor{red}{--} & 19.0 & 17.6 & 141 & 249 & \textbf{14.1} & \textcolor{red}{$>$1800} \\
\midrule
\texttt{pyrim5} & $10^{-2}$ & 1.9-10 & 1.9-10 & 1.4-7 & 8.7-11 & 9.4-8 & \textcolor{red}{4.9-3} & \textbf{2.31} & 2.36 & 13.8 & 7.16 & 2.40 & \textcolor{red}{514} \\*
 & $10^{-3}$ & 8.2-7 & 8.2-7 & 1.1-7 & 4.5-7 & 5.4-9 & \textcolor{red}{7.5-3} & 3.27 & 3.32 & 53.4 & 27.3 & \textbf{3.09} & \textcolor{red}{523} \\*
 & $10^{-4}$ & 2.4-8 & 2.4-8 & 8.1-7 & 6.2-7 & 1.6-7 & \textcolor{red}{1.2-2} & 3.90 & 3.66 & 366 & 135 & \textbf{3.57} & \textcolor{red}{502} \\
\midrule
\texttt{triazines4} & $10^{-2}$ & 8.7-11 & 8.7-11 & 3.2-7 & 4.3-7 & 1.0-7 & \textcolor{red}{--} & 25.0 & 25.2 & 134 & 69.0 & \textbf{13.3} & \textcolor{red}{$>$1800} \\*
 & $10^{-3}$ & 1.4-7 & 1.4-7 & 6.8-7 & 8.6-7 & 6.7-7 & \textcolor{red}{--} & 28.9 & 28.9 & 864 & 969 & \textbf{23.1} & \textcolor{red}{$>$1800} \\*
 & $10^{-4}$ & 4.5-8 & 4.5-8 & \textcolor{red}{--} & \textcolor{red}{--} & 7.4-7 & \textcolor{red}{--} & 25.1 & \textbf{24.9} & \textcolor{red}{$>$1800} & \textcolor{red}{$>$1800} & 47.5 & \textcolor{red}{$>$1800} \\
\midrule
\texttt{abalone7} & $10^{-2}$ & 4.6-10 & 4.6-10 & 7.4-14 & 2.3-7 & 1.9-7 & 5.9-7 & \textbf{0.72} & 0.74 & 2.91 & 2.42 & 1.06 & 18.0 \\*
 & $10^{-3}$ & 3.3-7 & 3.3-7 & 9.1-8 & 6.2-9 & 7.7-7 & 7.9-7 & 1.95 & 1.99 & 9.19 & 6.75 & \textbf{1.57} & 383 \\*
 & $10^{-4}$ & 5.6-8 & 5.6-8 & 6.2-7 & 7.7-9 & 3.6-7 & \textcolor{red}{1.7-3} & 3.19 & 3.14 & 28.6 & 19.2 & \textbf{2.64} & \textcolor{red}{500} \\
\midrule
\texttt{bodyfat7} & $10^{-2}$ & 4.0-9 & 4.0-9 & \textcolor{red}{$\dagger$} & 4.1-8 & 2.5-8 & 3.4-7 & \textbf{0.88} & 0.89 & \textcolor{red}{9.66} & 2.23 & 1.55 & 18.9 \\*
 & $10^{-3}$ & 3.5-12 & 3.5-12 & \textcolor{red}{$\dagger$} & 1.3-8 & 1.6-7 & 1.2-7 & 1.11 & \textbf{1.06} & \textcolor{red}{16.2} & 4.07 & 1.91 & 48.4 \\*
 & $10^{-4}$ & 6.2-7 & 6.2-7 & 1.8-7 & 1.4-15 & 8.5-8 & 6.9-7 & 2.03 & \textbf{2.00} & 20.9 & 11.4 & 3.22 & 122 \\
\midrule
\texttt{housing7} & $10^{-2}$ & 1.2-13 & 1.2-13 & 2.2-7 & 1.3-7 & 1.0-8 & 8.4-7 & 2.03 & \textbf{2.01} & 9.27 & 5.05 & 2.50 & 506 \\*
 & $10^{-3}$ & 1.4-9 & 1.4-9 & 2.6-7 & 3.7-11 & 9.8-7 & \textcolor{red}{9.2-3} & 4.11 & 4.09 & 33.2 & 14.1 & \textbf{3.45} & \textcolor{red}{941} \\*
 & $10^{-4}$ & 3.6-7 & 3.6-7 & 9.7-7 & 9.8-7 & 5.3-7 & \textcolor{red}{1.1-2} & 7.21 & 6.86 & 157 & 137 & \textbf{6.71} & \textcolor{red}{934} \\
\midrule
\texttt{mpg7} & $10^{-2}$ & 4.7-8 & 4.7-8 & 7.5-8 & 8.7-12 & 2.8-8 & 9.1-7 & \textbf{0.06} & \textbf{0.06} & 0.28 & 0.26 & 0.35 & 1.14 \\*
 & $10^{-3}$ & 1.6-9 & 1.6-9 & 1.3-7 & 3.2-7 & 6.6-8 & 9.8-7 & \textbf{0.11} & \textbf{0.11} & 0.50 & 0.32 & 0.45 & 3.46 \\*
 & $10^{-4}$ & 7.1-7 & 7.1-7 & 1.2-7 & 4.8-11 & 3.4-7 & 8.7-7 & \textbf{0.24} & 0.25 & 1.13 & 0.64 & 0.59 & 11.0 \\
\midrule
\texttt{space\_ga9} & $10^{-2}$ & 4.2-13 & 4.2-13 & 4.8-9 & 1.2-8 & 8.3-8 & 2.8-7 & \textbf{0.18} & \textbf{0.18} & 0.68 & 0.59 & 0.58 & 1.86 \\*
 & $10^{-3}$ & 3.4-8 & 3.4-8 & 5.9-11 & 2.4-11 & 9.0-7 & 9.7-7 & \textbf{0.29} & 0.31 & 1.14 & 0.99 & 0.74 & 8.47 \\*
 & $10^{-4}$ & 1.4-7 & 1.4-7 & 7.6-8 & 1.6-10 & 3.0-7 & 8.1-7 & 1.01 & \textbf{0.98} & 2.99 & 1.93 & 1.17 & 39.5 \\
\end{longtable}
\endgroup

\begingroup
\scriptsize
\setlength{\tabcolsep}{1.5pt}
\setlength{\aboverulesep}{0pt}\setlength{\belowrulesep}{0pt}
\renewcommand{\arraystretch}{1.0}
\setlength{\LTcapwidth}{\linewidth}
\begin{longtable}{@{}llr*{6}{c}|*{6}{c}@{}}
\caption{Same as Table~\ref{tab:biomedical_lasso_detail}, but for the fused Lasso problems with the UCI/LIBSVM datasets, where $\lambda_1=\alpha_1\|A^\top b\|_\infty$ and $\lambda_2=\alpha_2\lambda_1$; the time limit is 1800 seconds, so ``$>$1800'' marks a run that reached it. For GESN(B), the iteration limit includes its cap of 800 CG iterations per inner solve.}\label{tab:libsvm_fused_lasso_detail}\\
\toprule
 &  &  & \multicolumn{6}{c|}{$\eta$} & \multicolumn{6}{c}{Time (s)} \\
\cmidrule(r){4-9}\cmidrule(l){10-15}
Problem & $\alpha_1$ & $\alpha_2$ & GESN(A) & GESN(B) & GSSN(A) & GSSN(B) & SSNAL & ZeroFPR & GESN(A) & GESN(B) & GSSN(A) & GSSN(B) & SSNAL & ZeroFPR \\
\midrule
\endfirsthead
\multicolumn{15}{@{}l}{\normalsize\textit{Table~\thetable\ (continued)}}\\[2pt]
\toprule
 &  &  & \multicolumn{6}{c|}{$\eta$} & \multicolumn{6}{c}{Time (s)} \\
\cmidrule(r){4-9}\cmidrule(l){10-15}
Problem & $\alpha_1$ & $\alpha_2$ & GESN(A) & GESN(B) & GSSN(A) & GSSN(B) & SSNAL & ZeroFPR & GESN(A) & GESN(B) & GSSN(A) & GSSN(B) & SSNAL & ZeroFPR \\
\midrule
\endhead
\multicolumn{15}{r@{}}{\textit{continued on next page}}\\
\endfoot
\bottomrule
\endlastfoot
\texttt{E2006.} & $10^{-3}$ & 1 & 1.9-9 & 1.9-9 & 5.9-11 & 5.9-11 & 3.1-7 & 9.3-13 & 0.23 & \textbf{0.22} & 1.05 & 1.00 & 1.08 & 1.33 \\*
\texttt{train} &  & 0.5 & 5.9-9 & 5.9-9 & 6.0-11 & 6.0-11 & 9.7-8 & 1.3-13 & \textbf{0.21} & \textbf{0.21} & 1.12 & 0.91 & 0.99 & 1.23 \\*
 &  & 0.1 & 9.9-8 & 9.9-8 & 5.9-11 & 5.9-11 & 4.2-7 & 1.8-13 & \textbf{0.21} & 0.22 & 1.05 & 0.90 & 1.26 & 1.41 \\*
 &  & 0.01 & 6.6-8 & 6.6-8 & 5.9-11 & 5.9-11 & 4.9-7 & 4.0-13 & \textbf{0.21} & \textbf{0.21} & 1.03 & 0.90 & 1.21 & 1.41 \\
\texttt{E2006.} & $10^{-4}$ & 1 & 2.0-10 & 2.0-10 & 4.0-13 & 2.9-12 & 5.2-10 & 3.5-8 & 0.32 & \textbf{0.31} & 1.27 & 0.92 & 1.27 & 1.34 \\*
\texttt{train} &  & 0.5 & 2.7-10 & 2.7-10 & 1.9-13 & 2.7-12 & 1.2-8 & 4.3-8 & 0.30 & \textbf{0.29} & 1.54 & 1.02 & 1.50 & 1.66 \\*
 &  & 0.1 & 3.4-10 & 3.4-10 & 6.7-13 & 1.5-12 & 3.6-8 & 2.0-7 & \textbf{0.28} & 0.30 & 1.94 & 1.15 & 1.31 & 1.55 \\*
 &  & 0.01 & 3.5-10 & 3.5-10 & 1.0-12 & 7.7-12 & 3.3-8 & 2.8-9 & \textbf{0.29} & 0.30 & 1.73 & 1.14 & 1.30 & 1.68 \\
\midrule
\texttt{E2006.test} & $10^{-3}$ & 1 & 1.7-7 & 1.7-7 & 3.1-11 & 3.1-11 & 1.9-9 & 2.5-13 & \textbf{0.09} & \textbf{0.09} & 0.66 & 0.48 & 0.88 & 0.78 \\*
 &  & 0.5 & 1.7-7 & 1.7-7 & 3.1-11 & 3.1-11 & 2.7-9 & 3.1-13 & \textbf{0.09} & \textbf{0.09} & 0.59 & 0.48 & 0.76 & 0.71 \\*
 &  & 0.1 & 1.5-7 & 1.5-7 & 3.1-11 & 3.1-11 & 3.3-9 & 5.3-13 & \textbf{0.09} & \textbf{0.09} & 0.69 & 0.48 & 0.88 & 0.93 \\*
 &  & 0.01 & 1.7-7 & 1.7-7 & 3.1-11 & 3.1-11 & 3.4-9 & 2.4-13 & \textbf{0.08} & 0.09 & 0.72 & 0.47 & 0.86 & 0.72 \\
\texttt{E2006.test} & $10^{-4}$ & 1 & 2.7-10 & 2.7-10 & 8.2-14 & 3.9-12 & 1.3-7 & 1.7-8 & \textbf{0.13} & \textbf{0.13} & 0.83 & 0.63 & 0.80 & 0.70 \\*
 &  & 0.5 & 2.7-10 & 2.7-10 & 1.6-13 & 2.5-12 & 1.0-7 & 1.6-8 & \textbf{0.13} & \textbf{0.13} & 0.86 & 0.63 & 0.72 & 0.96 \\*
 &  & 0.1 & 2.5-10 & 2.5-10 & 4.7-13 & 2.6-12 & 7.4-8 & 1.1-8 & \textbf{0.12} & \textbf{0.12} & 0.82 & 0.70 & 0.83 & 0.86 \\*
 &  & 0.01 & 2.4-10 & 2.4-10 & 1.2-12 & 8.6-12 & 6.8-8 & 1.4-8 & \textbf{0.13} & 0.14 & 0.82 & 0.65 & 0.77 & 0.87 \\
\midrule
\texttt{log1p.} & $10^{-3}$ & 1 & 6.5-12 & 6.4-12 & 1.9-13 & 1.4-10 & 3.7-7 & 5.3-7 & 3.97 & \textbf{3.93} & 63.3 & 47.7 & 12.1 & 298 \\*
\texttt{E2006.test} &  & 0.5 & 1.8-10 & 1.8-10 & 1.2-7 & 1.8-13 & 3.8-7 & 9.7-7 & 4.70 & \textbf{4.16} & 55.7 & 36.3 & 11.9 & 267 \\*
 &  & 0.1 & 2.3-10 & 2.3-8 & 6.1-9 & 1.1-7 & 2.5-7 & 9.3-7 & 7.78 & \textbf{5.92} & 50.2 & 35.4 & 14.2 & 803 \\*
 &  & 0.01 & 5.9-9 & 5.9-9 & 2.0-7 & 1.6-9 & 2.3-7 & 8.7-7 & \textbf{10.6} & 11.6 & 70.0 & 41.2 & 14.9 & 1044 \\
\texttt{log1p.} & $10^{-4}$ & 1 & 2.0-8 & \textcolor{red}{5.2+0} & 7.6-8 & 7.5-7 & 4.4-7 & \textcolor{red}{--} & \textbf{16.1} & \textcolor{red}{11.6} & 183 & 194 & 30.0 & \textcolor{red}{$>$1800} \\*
\texttt{E2006.test} &  & 0.5 & 3.1-9 & \textcolor{red}{2.6+1} & 6.0-7 & 9.1-7 & 2.1-7 & \textcolor{red}{--} & \textbf{17.2} & \textcolor{red}{10.1} & 184 & 404 & 31.9 & \textcolor{red}{$>$1800} \\*
 &  & 0.1 & 3.7-11 & 5.3-11 & 1.8-8 & 8.0-7 & 6.5-7 & \textcolor{red}{--} & 32.3 & 31.7 & 570 & 740 & \textbf{28.1} & \textcolor{red}{$>$1800} \\*
 &  & 0.01 & 4.2-11 & 5.0-12 & 1.1-8 & 9.5-7 & 6.1-7 & \textcolor{red}{--} & 50.6 & 53.3 & 336 & 900 & \textbf{25.4} & \textcolor{red}{$>$1800} \\
\midrule
\texttt{pyrim5} & $10^{-3}$ & 1 & 5.9-10 & 5.9-10 & 7.7-7 & 2.0-8 & 2.4-7 & \textcolor{red}{4.0-3} & 5.48 & 5.82 & 116 & 36.2 & \textbf{4.57} & \textcolor{red}{753} \\*
 &  & 0.5 & 7.5-9 & 7.4-9 & 2.2-7 & 4.1-7 & 3.2-7 & \textcolor{red}{7.7-3} & 4.78 & 4.47 & 78.0 & 36.6 & \textbf{3.17} & \textcolor{red}{702} \\*
 &  & 0.1 & 2.0-8 & 2.0-8 & 9.4-8 & 9.1-10 & 5.7-7 & \textcolor{red}{9.9-3} & \textbf{3.96} & 3.97 & 262 & 55.6 & 4.21 & \textcolor{red}{729} \\*
 &  & 0.01 & 9.4-10 & 9.4-10 & 5.0-7 & 4.3-8 & 1.2-9 & \textcolor{red}{5.2-3} & 4.53 & 4.55 & 231 & 73.1 & \textbf{3.49} & \textcolor{red}{677} \\
\texttt{pyrim5} & $10^{-4}$ & 1 & 5.6-8 & 5.6-8 & 7.7-7 & 8.0-7 & 5.6-7 & \textcolor{red}{9.3-3} & \textbf{5.69} & 6.61 & 476 & 220 & 6.01 & \textcolor{red}{678} \\*
 &  & 0.5 & 1.1-9 & 1.1-9 & 7.2-7 & 1.7-7 & 5.1-7 & \textcolor{red}{6.0-3} & 8.35 & 8.05 & 462 & 196 & \textbf{4.59} & \textcolor{red}{706} \\*
 &  & 0.1 & 7.1-7 & 7.1-7 & 3.4-7 & 2.3-7 & 3.3-7 & \textcolor{red}{7.0-3} & 8.24 & 8.02 & 929 & 314 & \textbf{5.08} & \textcolor{red}{658} \\*
 &  & 0.01 & 4.5-8 & 4.5-8 & 7.3-7 & 8.7-7 & 4.7-8 & \textcolor{red}{7.9-3} & 9.30 & 9.01 & 1100 & 443 & \textbf{4.82} & \textcolor{red}{618} \\
\midrule
\texttt{triazines4} & $10^{-3}$ & 1 & 1.0-11 & 1.0-11 & \textcolor{red}{--} & 4.8-8 & 8.3-8 & \textcolor{red}{--} & 51.5 & 47.3 & \textcolor{red}{$>$1800} & 482 & \textbf{16.8} & \textcolor{red}{$>$1800} \\*
 &  & 0.5 & 3.9-9 & 3.9-9 & \textcolor{red}{--} & 2.9-7 & 3.5-7 & \textcolor{red}{--} & 36.0 & 34.5 & \textcolor{red}{$>$1800} & 994 & \textbf{18.0} & \textcolor{red}{$>$1800} \\*
 &  & 0.1 & 1.3-10 & \textcolor{red}{1.8-6} & \textcolor{red}{--} & 4.1-7 & 4.6-8 & \textcolor{red}{--} & 47.6 & \textcolor{red}{44.0} & \textcolor{red}{$>$1800} & 1192 & \textbf{20.4} & \textcolor{red}{$>$1800} \\*
 &  & 0.01 & 1.3-10 & 1.3-10 & \textcolor{red}{--} & \textcolor{red}{--} & 1.7-7 & \textcolor{red}{--} & 38.5 & 38.9 & \textcolor{red}{$>$1800} & \textcolor{red}{$>$1800} & \textbf{21.0} & \textcolor{red}{$>$1800} \\
\texttt{triazines4} & $10^{-4}$ & 1 & 3.9-8 & 3.9-8 & \textcolor{red}{--} & \textcolor{red}{--} & 1.4-7 & \textcolor{red}{--} & 32.6 & 32.6 & \textcolor{red}{$>$1800} & \textcolor{red}{$>$1800} & \textbf{24.3} & \textcolor{red}{$>$1800} \\*
 &  & 0.5 & 2.4-7 & 2.4-7 & \textcolor{red}{--} & \textcolor{red}{--} & 1.6-7 & \textcolor{red}{--} & 32.8 & 32.9 & \textcolor{red}{$>$1800} & \textcolor{red}{$>$1800} & \textbf{25.6} & \textcolor{red}{$>$1800} \\*
 &  & 0.1 & 7.7-7 & 7.7-7 & \textcolor{red}{--} & \textcolor{red}{--} & 7.6-8 & \textcolor{red}{--} & 36.1 & 36.1 & \textcolor{red}{$>$1800} & \textcolor{red}{$>$1800} & \textbf{27.9} & \textcolor{red}{$>$1800} \\*
 &  & 0.01 & 7.8-10 & 7.8-10 & \textcolor{red}{--} & \textcolor{red}{--} & 1.5-7 & \textcolor{red}{--} & 41.1 & 40.8 & \textcolor{red}{$>$1800} & \textcolor{red}{$>$1800} & \textbf{27.8} & \textcolor{red}{$>$1800} \\
\midrule
\texttt{abalone7} & $10^{-3}$ & 1 & 5.0-7 & 5.0-7 & 2.8-8 & 4.9-7 & 9.6-7 & 9.8-7 & 1.39 & 1.38 & 12.1 & 8.08 & \textbf{1.28} & 65.7 \\*
 &  & 0.5 & 8.0-8 & 8.0-8 & 2.6-9 & 1.2-8 & 3.9-8 & 9.5-7 & 1.46 & 1.42 & 12.7 & 5.82 & \textbf{1.30} & 82.7 \\*
 &  & 0.1 & 7.7-7 & 7.7-7 & 3.8-9 & 8.7-8 & 3.5-7 & 9.8-7 & 2.12 & 2.10 & 21.5 & 11.2 & \textbf{1.60} & 204 \\*
 &  & 0.01 & 2.4-11 & 2.4-11 & 1.2-7 & 3.7-8 & 8.1-8 & 9.2-7 & 1.99 & 1.96 & 23.1 & 11.9 & \textbf{1.58} & 258 \\
\texttt{abalone7} & $10^{-4}$ & 1 & 7.3-7 & 7.3-7 & 1.3-8 & 3.5-9 & 2.1-7 & 9.6-7 & \textbf{2.31} & 2.33 & 37.7 & 18.0 & 2.45 & 352 \\*
 &  & 0.5 & 2.6-7 & 2.6-7 & 5.1-7 & 1.4-8 & 2.6-7 & 9.9-7 & 2.48 & 2.50 & 36.3 & 20.2 & \textbf{2.44} & 395 \\*
 &  & 0.1 & 5.2-8 & 5.2-8 & 7.3-8 & 1.9-7 & 5.0-7 & \textcolor{red}{4.2-4} & 3.59 & 3.56 & 62.6 & 30.9 & \textbf{3.22} & \textcolor{red}{477} \\*
 &  & 0.01 & 6.4-7 & 6.4-7 & 3.1-7 & 6.5-8 & 5.6-8 & \textcolor{red}{4.0-4} & 3.86 & 3.78 & 74.5 & 28.0 & \textbf{3.66} & \textcolor{red}{474} \\
\midrule
\texttt{bodyfat7} & $10^{-3}$ & 1 & 9.9-13 & 9.9-13 & 3.2-14 & 6.1-12 & 6.1-8 & 7.0-7 & \textbf{1.06} & 1.07 & 6.39 & 6.88 & 1.60 & 45.3 \\*
 &  & 0.5 & 2.5-12 & 2.5-12 & 3.6-7 & 6.8-15 & 6.2-7 & 6.5-7 & 1.03 & \textbf{1.02} & 7.27 & 7.08 & 1.54 & 43.3 \\*
 &  & 0.1 & 4.5-12 & 4.5-12 & 3.2-14 & 1.1-7 & 8.4-8 & 1.2-7 & \textbf{1.50} & 1.51 & 10.2 & 7.95 & 1.74 & 51.2 \\*
 &  & 0.01 & 7.0-13 & 7.0-13 & 2.2-7 & 2.1-7 & 1.1-7 & 4.5-7 & 1.35 & \textbf{1.29} & 16.1 & 7.67 & 1.75 & 52.8 \\
\texttt{bodyfat7} & $10^{-4}$ & 1 & 7.0-13 & 7.0-13 & 2.0-14 & 1.4-10 & 1.7-7 & 8.5-7 & \textbf{1.53} & 1.54 & 18.0 & 9.94 & 2.19 & 75.6 \\*
 &  & 0.5 & 6.0-9 & 6.0-9 & 3.9-14 & 3.2-14 & 2.9-7 & 4.6-7 & 1.96 & \textbf{1.92} & 23.5 & 11.7 & 2.09 & 95.2 \\*
 &  & 0.1 & 1.2-8 & 1.2-8 & 8.1-9 & 2.5-13 & 1.2-8 & 8.1-7 & \textbf{2.38} & 2.39 & 31.9 & 14.3 & 2.69 & 126 \\*
 &  & 0.01 & 7.5-12 & 7.5-12 & 1.7-9 & 3.4-13 & 9.1-7 & 9.8-7 & \textbf{2.16} & \textbf{2.16} & 29.8 & 17.0 & 2.56 & 125 \\
\midrule
\texttt{housing7} & $10^{-3}$ & 1 & 7.6-7 & 7.6-7 & 8.3-7 & 3.5-8 & 5.9-7 & 1.0-6 & 3.40 & 3.42 & 53.1 & 24.9 & \textbf{3.08} & 912 \\*
 &  & 0.5 & 6.2-7 & 6.2-7 & 4.3-8 & 1.8-10 & 4.9-8 & \textcolor{red}{4.2-5} & 4.37 & 3.94 & 124 & 29.4 & \textbf{3.04} & \textcolor{red}{977} \\*
 &  & 0.1 & 1.1-7 & 1.1-7 & 2.4-7 & 1.5-7 & 2.5-7 & \textcolor{red}{2.8-3} & 5.32 & 5.31 & 124 & 50.0 & \textbf{3.60} & \textcolor{red}{1029} \\*
 &  & 0.01 & 2.6-8 & 2.6-8 & 1.5-7 & 5.3-8 & 2.6-7 & \textcolor{red}{3.1-2} & 6.03 & 6.02 & 140 & 59.0 & \textbf{3.57} & \textcolor{red}{883} \\
\texttt{housing7} & $10^{-4}$ & 1 & 2.9-12 & 2.9-12 & 1.6-7 & 2.0-7 & 7.5-7 & \textcolor{red}{6.3-3} & 6.86 & 6.92 & 868 & 89.0 & \textbf{4.09} & \textcolor{red}{918} \\*
 &  & 0.5 & 8.4-7 & 8.4-7 & 4.5-8 & 3.2-12 & 5.7-7 & \textcolor{red}{1.1-2} & 7.34 & 7.43 & 739 & 159 & \textbf{4.33} & \textcolor{red}{978} \\*
 &  & 0.1 & 6.9-8 & 6.9-8 & 5.6-8 & 6.8-7 & 6.7-7 & \textcolor{red}{2.0-2} & 7.92 & 7.88 & 400 & 489 & \textbf{4.69} & \textcolor{red}{894} \\*
 &  & 0.01 & 6.4-7 & 6.4-7 & 3.9-8 & \textcolor{red}{--} & 6.1-7 & \textcolor{red}{1.9-2} & 9.31 & 9.43 & 804 & \textcolor{red}{$>$1800} & \textbf{4.85} & \textcolor{red}{874} \\
\midrule
\texttt{mpg7} & $10^{-3}$ & 1 & 8.8-8 & 8.8-8 & 1.9-8 & 6.5-7 & 4.6-8 & 9.0-7 & 0.09 & \textbf{0.08} & 0.58 & 0.42 & 0.45 & 2.66 \\*
 &  & 0.5 & 9.6-9 & 9.6-9 & 3.3-8 & 6.7-13 & 5.9-7 & 1.0-6 & \textbf{0.09} & \textbf{0.09} & 0.59 & 0.42 & 0.48 & 3.15 \\*
 &  & 0.1 & 5.1-7 & 5.1-7 & 5.4-8 & 8.3-8 & 2.4-7 & 9.8-7 & \textbf{0.13} & \textbf{0.13} & 0.65 & 0.40 & 0.50 & 4.16 \\*
 &  & 0.01 & 6.7-8 & 6.7-8 & 2.8-8 & 2.6-9 & 1.3-7 & 1.0-6 & 0.14 & \textbf{0.13} & 0.82 & 0.52 & 0.46 & 4.87 \\
\texttt{mpg7} & $10^{-4}$ & 1 & 1.4-7 & 1.4-7 & 1.1-7 & 1.6-10 & 7.1-7 & 1.0-6 & \textbf{0.20} & \textbf{0.20} & 1.81 & 0.80 & 0.52 & 11.0 \\*
 &  & 0.5 & 7.8-8 & 7.8-8 & 2.8-7 & 9.3-7 & 9.1-7 & 9.9-7 & \textbf{0.20} & 0.22 & 1.95 & 0.97 & 0.62 & 11.8 \\*
 &  & 0.1 & 8.1-8 & 8.1-8 & 2.0-7 & 1.6-7 & 3.4-7 & 9.3-7 & \textbf{0.28} & 0.29 & 2.62 & 1.19 & 0.53 & 14.6 \\*
 &  & 0.01 & 1.9-7 & 1.9-7 & 2.2-7 & 3.2-8 & 1.9-7 & 8.3-7 & \textbf{0.35} & \textbf{0.35} & 3.01 & 1.08 & 0.62 & 15.6 \\
\midrule
\texttt{space\_ga9} & $10^{-3}$ & 1 & 3.6-11 & 3.6-11 & 2.3-7 & 1.6-13 & 3.6-7 & 9.1-7 & 0.26 & \textbf{0.25} & 2.15 & 0.81 & 0.73 & 3.20 \\*
 &  & 0.5 & 1.1-11 & 1.1-11 & 3.6-8 & 1.1-8 & 6.7-7 & 7.6-7 & \textbf{0.24} & 0.27 & 2.64 & 1.19 & 0.64 & 3.65 \\*
 &  & 0.1 & 4.0-9 & 4.0-9 & 7.7-7 & 1.0-8 & 1.1-7 & 8.8-7 & \textbf{0.29} & \textbf{0.29} & 3.69 & 1.24 & 0.92 & 8.88 \\*
 &  & 0.01 & 3.5-11 & 3.5-11 & 1.0-8 & 1.1-9 & 1.6-7 & 9.5-7 & \textbf{0.33} & \textbf{0.33} & 4.36 & 1.44 & 0.85 & 9.72 \\
\texttt{space\_ga9} & $10^{-4}$ & 1 & 2.0-10 & 2.0-10 & 2.5-8 & 2.7-12 & 1.6-7 & 7.9-7 & \textbf{0.72} & 0.75 & 13.2 & 2.63 & 1.15 & 22.3 \\*
 &  & 0.5 & 1.5-8 & 1.5-8 & 2.5-8 & 4.3-8 & 9.6-7 & 6.4-7 & 0.82 & \textbf{0.81} & 11.6 & 4.71 & 0.85 & 21.1 \\*
 &  & 0.1 & 2.8-8 & 2.8-8 & 7.9-8 & 1.3-8 & 7.3-7 & 8.9-7 & 1.16 & 1.23 & 15.4 & 4.31 & \textbf{1.15} & 32.5 \\*
 &  & 0.01 & 2.7-7 & 2.7-7 & 1.7-8 & 8.8-8 & 8.7-8 & 9.4-7 & 1.15 & \textbf{1.10} & 14.5 & 5.98 & 1.48 & 36.9 \\
\end{longtable}
\endgroup